\documentclass[12pt]{amsart} \usepackage{amsmath, amssymb} \newfont {\cyr} {wncyr10} 
\usepackage{amsfonts} \usepackage{amsmath} \usepackage{amssymb} \usepackage{setspace}
\usepackage{multicol}
\usepackage{color}
\usepackage{mathtools}
\usepackage{fancyref}
\usepackage[toc,page]{appendix}
\newcommand*{\fancyrefthmlabelprefix}{thm}
\frefformat{plain}{\fancyrefthmlabelprefix}{Theorem \textup{#1}}

\newcommand*{\fancyrefsectionlabelprefix}{sec}
\frefformat{plain}{\fancyrefsectionlabelprefix}{Section~\textup{#1}}

\newcommand*{\fancyrefnotlabelprefix}{not}
\frefformat{plain}{\fancyrefnotlabelprefix}{Notation \textup{#1}}

\newcommand*{\fancyrefprolabelprefix}{pro}
\frefformat{plain}{\fancyrefprolabelprefix}{Property~\textup{#1}}

\newcommand*{\fancyreflemmalabelprefix}{lem}
\frefformat{plain}{\fancyreflemmalabelprefix}{Lemma \textup{#1}}

\newcommand*{\fancyrefdeflabelprefix}{def}
\frefformat{plain}{\fancyrefdeflabelprefix}{Definition \textup{#1}}

\newcommand*{\fancyrefproplabelprefix}{prop}
\frefformat{plain}{\fancyrefproplabelprefix}{Proposition \textup{#1}}

\newcommand*{\fancyrefcorlabelprefix}{cor}
\frefformat{plain}{\fancyrefcorlabelprefix}{Corollary \textup{#1}}

\newcommand*{\fancyrefclmlabelprefix}{clm}
\frefformat{plain}{\fancyrefclmlabelprefix}{\textup{#1}}

\newcommand*{\fancyrefhyplabelprefix}{hyp}
\frefformat{plain}{\fancyrefhyplabelprefix}{Hypothesis \textup{#1}}

\newcommand*{\fancyrefequationlabelprefix}{eq}
\frefformat{plain}{\fancyrefequationlabelprefix}{\textup{(#1)}}

\mathchardef\tnode="020E

\def\arc{ 
 \hbox{\kern -0.15em \vbox{\hrule width 2.5em height 0.6ex depth -0.5 ex} \kern
-0.33em}}

\def\darc{
 \rlap{\lower0.2ex\arc}{\raise0.2ex\arc}}

\def\tarc{
 \rlap{\rlap{\lower0.4ex\arc}{\raise0.4ex\arc}}{\arc}}

\def\stroke#1{
 \kern 0.05
\rlap\arc{{\textstyle{#1}}\atop\phantom\arc} \kern -0.22em}

\def\dstroke#1{
 \kern 0.05em
\rlap\darc{{\textstyle{#1}}\atop\phantom\darc} \kern -0.22em}

\def\centerscript#1{
 \setbox0=\hbox{$\tnode$} \hbox to
\wd0{\hss$\scriptstyle{#1}$\hss}}

\def\node{
 \def\super{} \def\sub{}
\futurelet\next\dolabellednode}

  \let\sp=^ \let\sb=_

  \def\dolabellednode{%
   \ifx\next\sb\let\next\getsub \else \ifx\next\sp\let\next\getsuper
\else\let\next\donode \fi \fi \next}

  \def\getsub_#1{\def\sub{#1}\futurelet\next\dolabellednode}

\def\getsuper^#1{\def\super{#1}\futurelet\next \dolabellednode}

  \def\donode{%
   \rlap{$\mathop{\phantom\tnode}\limits_{\centerscript{\sub}}
^{\centerscript{\super}}$}\tnode}

\def\varcdn{
 \kern
0.3em\vbox{\kern -0.7ex \hbox to \wd0{\hss\vrule width 0.04em depth 5.8ex\hss} \kern -0.3ex \hbox{$\tnode_2$}}}

 \def\a3{\node_1\arc\node_2\arc\node_3}
 
\def\c3{\node_1\arc\node_2\darc\node_3}

\def\d4{\node_1 \arc\node_\varcdn^4 \arc\node_3}

\def\m24{\node\arc\node\dstroke{\sim}\node} \def\u43{\node\darc\node\dstroke{\sim}\node}

\newcommand{\varcdnl}[1]{ 
\kern -0.03em\vbox{\kern -0.5ex \hbox to \wd0{\hss\vrule width 0.04em depth 5.8ex\hss} \kern -0.3ex
\hbox{$\tnode^{#1}$}}}

\def\nodef{
\def\super{} \def\sub{} \futurelet\next\dolabellednodef}

  \let\sp=^ \let\sb=_

  \def\dolabellednodef{%
  \ifx\next\sb\let\next\getsubf \else
\ifx\next\sp\let\next\getsuperf \else\let\next\donodef \fi \fi \next}

\def\getsubf_#1{\def\sub{#1}\futurelet\next\dolabellednodef}

\def\getsuperf^#1{\def\super{#1}\futurelet\next \dolabellednodef}

  \def\donodef{%
  \rlap{$\mathop{\phantom\tnodef}\limits_{\centerscript{\sub}}
^{\centerscript{\super}}$}\tnodef}

\def\varcdnf{
 \kern -0.03em\vbox{\kern -0.5ex \hbox to \wd0{\hss\vrule width 0.04em depth
5.8ex\hss} \kern -0.3ex \hbox{$\tnodef$}}}

\newtheorem{theorem}{Theorem}[section]
\newtheorem*{theorem*}{Theorem}
\newtheorem{corollary*}{Corollary}
\newtheorem{lemma}[theorem]{Lemma}

\newtheorem{proposition}[theorem]{Proposition} 
\newtheorem{definition}[theorem]{Definition}
\newtheorem{hypothesis}[theorem]{Hypothesis}
\newtheorem{remark}[theorem]{Remark}

\newcounter{claim}[theorem]

\newcounter{cclaim}[theorem]

\def \udot {{}^{\textstyle .}} 
\newcommand{\E}{\mathrm{E}}
\newcommand{\F}{\mathrm{F}}\newcommand{\A}{\mathrm{A}}\newcommand{\B}{\mathrm{B}}
\newcommand{\M}{\mathcal{M}} \newcommand{\G}{\mathrm{G}} 
 \newcommand{\Q}{\mathrm{Q}} 
 \newcommand{\Aut}{\mathrm{Aut}} 
  \newcommand{\Out}{\mathrm{Out}}
\newcommand{\Inn}{\mathrm{Inn}} \newcommand{\Syl}{\mathrm{Syl}}\newcommand{\syl}{\mathrm{Syl}}
  
\newcommand{\GF}{\mathrm{GF}} \newcommand{\GL}{\mathrm{GL}} \newcommand{\Sp}{\mathrm{Sp}}
\newcommand{\SL}{\mathrm{SL}} \newcommand{\0}{\emptyset} \newcommand{\PGL}{\mathrm{PGL}}
\newcommand{\PSL}{\mathrm{PSL}}\newcommand{\PSp}{\mathrm{PSp}}
\newcommand{\PSU}{\mathrm{PSU}}
\newcommand{\Sym}{\mathrm{Sym}} \newcommand{\Alt}{\mathrm{Alt}} \newcommand{\Dih}{\mathrm{Dih}}
 \newcommand{\Frob}{\mathrm{Frob}} 
 \newcommand{\U}{\mathrm{U}}

\def \qedc {$\hfill \blacksquare$\newline}

\def \L {\hbox {\rm L}}

\def \Z {\mathbb Z} 

\def \syl {\hbox {\rm Syl}}\def \Syl {\hbox {\rm Syl}}

\def \ov {\overline}

\def \wt {\widetilde}

\def \Aut{ \mathrm {Aut}}

\def \Out{\mbox {\rm Out}}

\def \Mat{\mbox {\rm Mat}}

\def \J{\mbox {\rm J}}

\def \Ly{\mbox {\rm LyS}}

\def \B{\mbox {\rm B}}

\def \M{\mbox {\rm M}}

\def \HS{\mbox {\rm HS}}

\def \ON {\mbox {\rm O'N}}

\def \Co {\mbox {\rm Co}}

\def \Ru {\mbox {\rm Ru}}

\def \Suz{\mbox {\rm Suz}}

\def \McL{\mbox {\rm McL}}\def \POmega {\mbox {\rm P}\Omega}

\def \He {\mbox {\rm He}}

\def \diag {\mathrm {diag}}
\usepackage{xr}

\begin{document}

\title{The local structure theorem, the non-characteristic $2$ case}
 \author{Chris Parker}
  \author{Gernot Stroth}

\address{Chris Parker\\
School of Mathematics\\
University of Birmingham\\
Edgbaston\\
Birmingham B15 2TT\\
United Kingdom} \email{c.w.parker@bham.ac.uk}

\address{Gernot Stroth\\
Institut f\"ur Mathematik\\ Universit\"at Halle - Wittenberg\\
Theordor Lieser Str. 5\\ 06099 Halle\\ Germany}
\email{gernot.stroth@mathematik.uni-halle.de}

\maketitle

\begin{abstract}
Let $p$ be a prime, $G$ a finite $\mathcal{K}_p$-group, $S$ a Sylow $p$-subgroup of $G$ and $Q$ be a large subgroup of $G$  in $S$. The aim of the Local Structure Theorem \cite{Struc} is to provide structural information about subgroups $L$ with $S \leq L$, $O_p(L) \not= 1$ and $L \not\leq N_G(Q)$. There is, however, one  configuration  where no structural information about $L$ can be given using the methods in \cite{Struc}. In this paper we show that for $p=2$ this hypothetical configuration cannot occur. We anticipate that our theorem will be used in the programme to revise the classification of the finite simple groups.
\end{abstract}

\section{Introduction}\label{sec:1}
 The proof of the  classification of the finite simple groups  took different directions  depending upon the structure of normalizers of non-trivial $2$-subgroups. Such subgroups are called \emph{$2$-local subgroups}. If $M$ is such a 2-local subgroup, then there are two possibilities $C_M(O_2(M)) \leq O_2(M)$ or $C_M(O_2(M)) \not\leq O_2(M)$. In the former case, we say that $M$ has \emph{characteristic $2$}. If  all the  $2$-local subgroups  of a finite group $G$ have characteristic $2$, then we say that $G$ is of \emph{local characteristic $2$}. The classification divides  into the investigation of groups which are of local characteristic 2 and those which are not.  In the latter case the objective is to show that there is a  $2$-local subgroup which has a fairly simple structure (a subnormal $\SL_2(q)$, standard subgroups). One of the main obstructions for proving the existence of such a $2$-local subgroup is the existence of non-trivial normal subgroups of odd order in  $2$-local subgroups. A new approach due to M. Aschbacher (for an overview see \cite[Chapter 2]{AO}) using fusion systems avoids this problem. The first steps of this programme can be found in a preprint \cite{Asch}.

For groups of local characteristic 2, the problem is the complexity of the structure of the  2-local subgroups.  In the original classification, to avoid this complexity problem the strategy was to move to  $p$-local subgroups for suitable odd primes $p$, which then  eventually have a fairly restricted structure similar to standard subgroups.

In the years following the classification, methods  for working with $2$-local subgroups of  groups of local characteristic $2$ have been refined and developed. These new methods   inspired  a novel approach to the classification of groups of local characteristic $2$ initiated by  U. Meierfrankenfeld, B. Stellmacher and G. Stroth (see \cite{Stuc1} for an overview), the MSS-programme for short, which stays in the 2-local world and intends to pin down the structure of $G$. The Local Structure Theorem \cite{Struc}  provides information about  important subgroups and quotients of certain $2$-local subgroups  and further work is in progress. A tempting possibility is that there is  a bridge between Aschbacher's programme  and the MSS-programme which means they can be merged to give a new proof of  the classification of the finite simple groups. \text{One of the purposes} of this paper is to build   part of such a  bridge.

We now   explain how these two programmes  can possibly be joined.  For this we have to say a little bit more about Aschbacher's approach. As an example, let us assume that we have a 2-local subgroup $M \cong 2 \times \Alt(5)$. Then our target simple group is the sporadic simple group $\J_1$, but  the groups $\SL_2(16){:}2$ and $\Alt(5) \wr 2$ also have such a $2$-local subgroup, of course they are not simple. However to detect this fact takes a lot of work. To avoid this problem Aschbacher  assumes that  $M$ contains an  elementary abelian $2$-subgroup of $G$ of maximal order. With this extra condition $\J_1$ is the unique solution (assuming $O_2(G)=O(G)=1$). The problem is that an approach to the classification based on  Aschbacher's new work no longer has the tidy division into two cases:  local characteristic 2 or  not local characteristic 2.  To take the discussion further,  we introduce the notion of \emph{parabolic characteristic $2$}. This means that we require  $M$ has characteristic $2$ only for those $2$-local subgroups $M$ of odd index in $G$. If we could classify the groups of parabolic characteristic 2, then this would be a counterpart to Aschbacher's work and together they would provide an alternative proof of the classification.  {At the moment providing such  classification seems to be out of reach. However, it is also more than is required. Fix  $S\in \syl_2(G)$ and recall the \emph{Baumann subgroup} of $S$ is defined to be $B(S)=C_S(\Omega_1(Z(J(S)))$. For a saturated fusion system $\mathcal F$ on a $2$-group $T$, Aschbacher considers components of $C_{\mathcal F}(t)$ where $t$ is an involution with $m_2(T)= m_2(C_T(t))$ (see \cite[page 5]{Asch}). So, if Aschbacher's programme is successful, then we can assume that $C_G(t)$ has characteristic $2$ for all involutions in $\Omega_1(Z(J(S)))$.  Hence,  if we could determine the groups in which  every $2$-local subgroup containing $B(S)$ has characteristic $2$, then we could meld the two programme and produces an alternative proof of the classification. Such groups are called groups of \emph{Baumann characteristic $2$}.}

%

So far the investigation in MSS focuses on groups which possess a large subgroup $Q$ (the exact definition will be given later on). A consequence of the existence of such a group is that $G$ has parabolic characteristic 2.  The Local Structure Theorem in \cite{Struc} gives information about the structure of groups of parabolic characteristic 2, which have a large subgroup $Q$. In fact this has been done for arbitrary primes $p$. For a $p$-local subgroup $M$ of characteristic $p$,  there is a unique non-trivial normal elementary abelian $p$-subgroup $Y_M$ maximal subject to  $O_p(M/C_M(Y_M)) = 1$.
The Local Structure Theorem gives information about $Y_M$ and the action of $M/C_M(Y_M)$ on $Y_M$ provided $Q$ is not normal in $M$ and $M$ contains a Sylow $p$-subgroup of $G$.
{To take the investigation further there are two cases to be investigated.  Either $Y_M \not\leq Q$ for some such $M$ or $Y_M \leq Q$ for all such $M$. In  both instances define
  $$H = \langle K\mid O_p(K) \not= 1, S\leq K\rangle.$$
  In the first case,  the $H$-Structure Theorem (work in preparation) builds on the Local Structure Theorem and  determines
the group $H$. Using this, for $p=2$,   $F^\ast(G)$ can be identified. If $p$ is odd, then either $F^\ast(G)$ is   determined, or $F^\ast(H)$ is demonstrated to be a group of Lie type in characteristic $p$ and rank at least three, or $H$ is a weak $BN$-pair. Up to this point the MSS-programme fits well with Aschbacher's point of view. In the second case, $Y_M \leq Q$ for all $M$, and again we intend to determine the group $H$.  For this, the first question is: which of the  $p$-local subgroups from the Local Structure Theorem can show up?  This  has been partly answered in \cite{MMPS, PPS} but only  under the assumption that $G$ has local characteristic $p$ and this assumption  is not compatible with  Aschbacher's approach. Hence we must replace it by a more applicable premise.
The starting point for \cite{MMPS, PPS} is  \cite[Corollary B]{Struc} which lists the  cases from the Local Structure Theorem which may appear when $Y_M \leq Q$ and $G$ is of local characteristic $p$.
Using this information  \cite{MMPS, PPS}  basically exclude what is called the wreath product case in the Local Structure Theorem. From now on assume that $p = 2$ as this is the relevant prime for Aschbacher's approach. The first question is: what happens if we remove the assumption of local characteristic $2$ in \cite[Corollary B]{Struc}?  The answer is that two further configurations for the  group $M$  appear. One is that $M/C_M(Y_M)$ induces the natural $\mathrm{O}^\pm_{2n}(2)$-module on $[Y_M,M]$. This possibility has been handled by Chr. Pr\"oseler \cite{Pr} in his PhD thesis. The second is that \cite[Theorem A (10-1)]{Struc} holds and this is the situation  handled in this paper. We will show that no groups satisfy this hypothesis.
Hence we may investigate the proofs of \cite{MMPS, PPS} starting with the same set of possible $2$-local subgroups as provided by  \cite[Corollary B]{Struc}. However the proofs in \cite{MMPS, PPS} also exploit that $G$ has local characteristic $2$ but only for $2$-local subgroups $K$ which contain the Baumann subgroup $B(S)$. Hence the local characteristic $2$ assumption can be replaced with Baumann characteristic $2$ and we then obtain the same conclusion as in \cite{MMPS,PPS}. Therefore, provided we can prove the analogue of the $H$-structure theorem in the case that $Y_M \le Q$ for all $M$ when $G$ has  Baumann characteristic $2$, we will have a companion to  Aschbacher's approach.}

 To explain further the context of the results in this article, we  give a simplified overview of the  Local Structure Theorem for the particular case when $p=2$ and then outline the contribution of the research in this paper. We work in an environment compatible with being a counter example to the classification. Thus  we call $G$  a $\mathcal K_2$-group if and only if every simple section of every $2$-local subgroup of $G$ is in the set of known simple groups $\mathcal K$ where $\mathcal K$ consists of  the groups of prime order, the alternating groups, the simple groups of Lie type  and the  sporadic simple groups.

 A subgroup $Q$ of $S$ is called \emph{large} if
\begin{itemize}
\item $Q = O_2(N_G(Q))$;
\item $C_G(Q) \leq Q$; and
\item for any $1 \not= A \leq Z(Q)$ we have that $Q$ is normal in $N_G(A)$.
\end{itemize}

For $K \leq H \leq G$ we define $$\mathcal L_H(K) = \{K \leq L \leq H~|~O_2(L) \not= 1, C_H(O_2(L)) \leq O_2(L)\}.$$
As we asserted earlier, the existence of the large subgroup $Q$ implies that $G$ has parabolic characteristic $2$ and so, in this case,  $\mathcal L_G(S)$  contains all of the $2$-local subgroups of $G$ which contain $S$. Define  $ \mathfrak M_G(S)$ to be the subset of $\mathcal L_G(S)$  which contains the subgroups $M\in \mathcal L_G(S)$ such that, setting $M^\dagger= MC_G(Y_M)$,
\begin{itemize}\item  $ \mathcal L_G(M)=\mathcal L_{ M^\dagger}(M)$ and
 $Y_M= Y_{M^\dagger}$; and\item $C_M(Y_M)/O_2(M) \le \Phi(M/O_2(M))$.
\end{itemize}
For $L \in\mathcal L_G(S)$ with $L \not \le N_G(Q)$,
define
 $$L^\circ =\langle Q^L\rangle.$$
With these  assumptions and notation the  Local Structure Theorem states:

\emph{ Suppose that $G$ is a $\mathcal K_2$-group, $S \in \syl_2(G)$ and $S$ is contained in at least two maximal $2$-local subgroups of $G$. Assume that $Q$ is a large subgroup of $G$ contained in $S$ and
 $M\in \mathfrak M_G(S)$ with $M \not \le N_G(Q)$.
 Then the structure of  $M^\circ/C_{M^\circ}(Y_M)$ and its action on $Y_M$ are described {\bf unless \cite[Theorem A (10--1)]{Struc} holds}.
}
\medskip

So far so good, but what is this mysterious clause   (10--1) in \cite[Theorem A]{Struc}? In this case, all that  is  proved is that $Y_M$ is tall and asymmetric in $G$, but importantly $Y_M$ is not characteristic  $p$-tall in $G$. We will now explain in detail what this means, as our intention is to remove this restriction to the Local Structure Theorem for the   situation when $p=2$ so that the we can provide the bridge to the Aschbacher project. \\
\\
Let $M\in \mathfrak M_G(S)$ and $T \in \syl_2(C_G(Y_M))$. The subgroup $Y_M$ is
\begin{itemize}
\item  \emph{tall}, if there exists $K$ with $T \le K \le G$ such that $O_2(K) \ne 1$ and $Y_M \not \le O_2(K)$,
and
  \item   \emph{asymmetric} in $G$, if whenever $g\in G$ and $[ Y_M,Y_M^g]\le Y_M \cap Y_M^g$, then $ [ Y_M,Y_M^g]=1$.
  \end{itemize}
Further $Y_M$ is  \emph{characteristic $2$-tall} provided
\begin{itemize}
\item   there is some $K$ with  $T\le K\le G$  such that  $C_K(O_2(K)) \le O_2(K)$ and $Y_M\not \le O_2(K)$.
\end{itemize}

We can now state the theorem we shall prove in this paper

\begin{theorem*}\label{thm:thm1}
Let $G$ be a finite ${\mathcal K}_2$-group and $S \in \syl_2(G)$. Suppose that $S$ is contained in at least two maximal $2$-local subgroups and that $Q$ is a large subgroup of $G$ in $S$.  Assume that there exists $M \in \mathfrak M_G(S)$ such that $Y_M$ is asymmetric and tall. Then $Y_M$ is characteristic  $2$-tall.
\end{theorem*}

This theorem  shows that for $p=2$ \cite[Theorem A (10-1)]{Struc} does not arise for $M \in \mathfrak M_G(S)$ and so also implies that Theorem A (10-1) of the Local Structure Theorem   does not occur for arbitrary $L \in \mathcal L_G(S)$ with $L \not\le N_G(Q)$ as \cite[Theorem A (10-1)]{Struc} states that there exists $M\in \mathfrak M_G(S)$ with $Y_L=Y_M$ and $L^\circ=M^\circ$.
\\
\\
For the proof of the theorem  of this paper we   assume that there exists $M \in \mathfrak M_G(S)$ with  $Y_M$ asymmetric and tall, but not characteristic 2-tall. This means that for $T \in \syl_2 (C_M(Y_M))$ if $K$ is a subgroup of $G$ containing $T$ with $O_2(K)\ne 1$ and $Y_M \not\leq O_2(K)$, then $C_G(O_2(K)) \not\le O_2(K)$. Thus there are involutions $y \in  Y_M$ such that $C_G(O_2(C_G(y))) \not\leq O_2(C_G(y))$. We study these centralizers and   would like to show that $E(C_G(y)) \not= 1$. That is, the centralizers have components.  The obstruction to this  is the existence of normal subgroups of odd order. The key for removing this obstacle is that $C_G(y)$ contains a 2-central element $z$ and, as $z \in Q$, using $Q$ large, yields $C_G(O_2(C_G(z))) \leq O_2(C_G(z))$. This implies that $z$ inverts any normal subgroup of odd order in $C_G(y)$ and so such subgroups are abelian.  In addition, we prove that $|Y_M| \ge 2^3$ and  we know $Y_M \le C_G(y)$. So signalizer functor methods can be employed to obtain  $E(C_G(y)) \not= 1$ (see Lemmas~\ref{lem:sig1} and \ref{lem:sig}). The arguments used to prove this do not transfer to odd primes as in this case we only find  that $O_{p^\prime}(C_G(y))$ is nilpotent and this prevents us demonstrating the balance condition required for use of the signalizer functor theorem.

From among all the components involved in the centralizers of elements in $Y_M$ we select one, $K$ say, with first $|K/Z(K)|$ maximal and then  $|K|$ maximal. Then from all the elements of $Y_M^\#$ that contain components we select those $y$ that have $|E_y|$-maximal where $E_y$ is the subgroup of $E(C_G(y))$ generated by components $J$ with $J/Z(J) \cong K/Z(K)$ and $|J|=|K|$. The set of such elements is denoted by $\mathcal Y^*$ and the members of $\mathcal Y^*$ are the focus of attention. With these choices, for $y \in \mathcal Y^*$, \fref{lem:Ty TI} shows that,  if $C_S(y)\in \syl_2(C_G(y))$ and $|C_S(y)|$ is maximal, then $C_{O_2(M)}(E_y)$ is a trivial intersection subgroup of $M$.  Roughly speaking, the contradictions which lead to the proof of the Theorem come about by finding that either $M$  normalizes a non-trivial subgroup of $Z(Q)$ which, as $Q$ is large,  contradicts $M \not\le N_G(Q)$, or that  $M$ is the unique maximal $2$-local subgroup containing $S$ which contradicts the fact that $S$ is contained in at least two such subgroups. These two observations are encoded in Lemmas \ref{lem:McircZ(Q)} and \ref{lem:YMnotmaxabelian}.

We give a little more detail, select $y \in \mathcal Y^*$, fix a component $K\le E_y$ and set $L_K= C_K(z)$ where $z$ is an involution in $Z(S)$. Then $L_K$ has characteristic $2$, and the examination of the various possibilities for $K$ take markedly different routes dependent upon whether or not $L_K$ is a $2$-group.
If $L_K$ is not a $2$-group, it is often possible to show that $K= E(C_G(y))$. Furthermore,   \fref{lem:special} asserts that $O^2(L_K)$ cannot act irreducibly on $O_2(L)/Z(O_2(L_K))$(the root of this observation lies in \fref{lem:McircZ(Q)}).  This fact eliminates many  candidates for $K/Z(K)$.  The detailed  arguments  are presented in  Sections~\ref{sec:spor}, \ref{sec:LieOdd}, \ref{sec:Alt} and \ref{sec:LieChar2} where, for the more difficult cases,  the $2$-local structure of $K$ plays a central role in the proof.
The data needed for this is provided in \fref{sec:Kgrp}.

By the end of \fref{sec:LieChar2}, we are left with two possibilities. Either $K/Z(K) \cong \PSL_3(4)$ or $K\cong \Sp_4(2^a)$. Interestingly in this situation we are unable to bound the number of components involved in $E_y$. We quickly prove that $Z(K)$ is elementary abelian and that $z \in \Omega_1(Z(S)) \le Y_M$ does not project to a root element when $K\cong \Sp_4(2^a)$.  In \fref{lem:rootinYM} we show that the Thompson subgroup of $O_2(M)$ is equal to  $(S \cap E_y)J(C_S(E_y))$ and this provides  our way into the study of these cases.  We eventually show that $M$ either normalizes $E_y$, or there is a further subgroup $K_{r+1} \cong K$ which commutes with $E_y$ such that $M$ normalizes $E_yK_{r+1}$. Our objective is to prove that every elementary abelian normal subgroup of $S$ is contained in $Y_M$, once this is done the contradiction is provided by \fref{lem:YMnotmaxabelian}.

\section{Preliminary group theoretical results}\label{sec:prel}

In this section we collect some group theoretical facts that we require. In this work we assume that all groups are finite.
Recall that for a prime $p$, a group $X$ has \emph{characteristic $p$} provided $C_X(O_p(X))\le O_p(X)$ or, equivalently, if $F^*(X)=O_p(X)$.  Our first lemma which is a consequence of coprime action and the Thompson $A\times B$-Lemma \cite[Lemma 11.7]{gls1}  is well-known and plays a critical role in our proof of the Theorem.

\begin{lemma}\label{lem:fund-charp} Suppose that $p$ is a prime,  $X$ is a group, $B$ is a $p$-subgroup of $X$ and  $C$ is a normal subgroup of $B$. If $N_X(C)$ has characteristic $p$, then $N_X(B)$ and $C_X(B)$ have characteristic $p$.
\end{lemma}
\begin{proof}
Set  $A= O_{p^\prime}(N_X(B))E(N_X(B))$. Then $A$ centralizes $B$ and so $A$ also centralizes $C \leq B$. Therefore $AB \le N_X(C)$ and $AB$ normalizes $P=O_p(N_X(C))$.  We have $C_P(B)$ normalizes $A$ and, as $[B,A]=1$, $C_P(B)$ is normalized by $A$. Hence $$[C_P(B),A]= [C_P(B),A,A] \le [O_p(A),A]\le [Z(E(N_X(B))),A] =1.$$
 As $A$ is generated by $p^\prime$-elements, the Thompson $A\times B$-Lemma implies that $A$ centralizes $P$ and hence $A=1$ as $N_X(C)$ has characteristic $p$. Therefore $F^*(N_X(B)) = O_p(N_X(B))$  and so $N_X(B)$ has characteristic $p$. Since $F^*(C_X(B)) \le F^*(N_X(B))$, we also have $C_X(B)$ has characteristic  $p$.
\end{proof}

As an example of how we might use \fref{lem:fund-charp} consider the case $X$ has characteristic $p$. Then we may take $C=1$, and obtain $N_X(B)$ has  characteristic $p$.

\begin{lemma}\label{lem:subnormal2} Let $X$ be a group of characteristic $p$ and $Y$ be subnormal in $X$. Then   $Y$ is a group of characteristic $p$.
\end{lemma}

\begin{proof} If $Y $ is subnormal in $X$, then $F^*(Y) \le F^*(X)$. Hence $F^*(Y)$ is
a $p$-group.
\end{proof}

The next lemma will be used to show that certain involutions have components in their centralizers.

\begin{lemma}\label{lem:sig1} Suppose that $X$ is a  group and $Y$ is an elementary abelian $2$-subgroup of $X$ of order at least $8$.  Assume that $E(C_X(x))= 1$ for all $x \in Y^\#$ and that there exists $z \in Y^\#$ such that $F^*(C_X(z))= O_2(C_X(z))$. Then  $\langle O(C_X(y)) \mid y \in Y^\#\rangle$  has  odd order and is normalized by $N_X(Y)$.
\end{lemma}

\begin{proof} Suppose that $a, b \in Y^\#$ are such that $F^*(C_X(a))= O_2(C_X(a))$ and  $O(C_X(b)) \ne 1$.  Then $C_{C_X(a)}(b) = C_{C_X(b)}(a)$ has characteristic $2$ by \fref{lem:fund-charp}. In particular, $C_{O(C_X(b))}(a)=1$ and so  $a$ inverts $O(C_X(b))$. This means that $O(C_X(b))$ is abelian. Since there exists $z \in Y^\#$ such that $F^*(C_X(z))= O_2(C_X(z))$, we have $O(C_X(b))$ is abelian and
 is inverted by $z$ for all $b \in Y^\#$.

 Suppose that $a,b \in Y^\#$  are arbitrary. We claim that $$O(C_X(b)) \cap C_X(a) \le O(C_X(a)).$$
If $F^*(C_X(b)) = O_2(C_X(b))$, then $O(C_X(b))=1$ and there is nothing to prove.
 Suppose that
 $F^*(C_X(a))= O_2(C_X(a))$, then we have already argued that $O(C_X(b)) \cap C_X(a)=1$ and so the claimed containment also holds in this case. Suppose   that $O(C_X(b)) \ne 1 \ne O(C_X(a))$.  Set $U= O(C_X(b)) \cap C_X(a) $. Then $\langle b \rangle \times U$ normalizes $O_2(C_X(a))$ and  $[C_{O_2(C_X(a))}(b), U] \le O_2(C_X(a)) \cap O(C_X(b))=1$. Thus again the Thompson $A\times B$-Lemma  implies that   $[U,O_2(C_X(a))]=1$.  Now consider $UO(C_X(a))$.  This group is normalized by $z$ and, as $z$ inverts $U$ and inverts $O(C_X(a))$, we have $z$ inverts $UO(C_X(a))$. But then $UO(C_X(a))$ is abelian.  Consequently $U$ centralizes $F^*(C_X(a))= O(C_X(a))O_2(C_X(a))$ and so $U \le O(C_X(a))$ as claimed.

  As $|Y|\ge 8$ by hypothesis,  the Soluble Signalizer Functor Theorem \cite[Theorem 21.3]{gls1}  implies that the completeness subgroup $$\Sigma= \langle O(C_X(b))\mid b \in Y_M^\#\rangle$$   has odd order. Finally we note that $N_X(Y)$ normalizes $\Sigma$ as it permutes the generating subgroups by conjugation. This completes the proof of the lemma.
  \end{proof}

Recall from \cite[Definition 4.5]{gls1} that a \emph{$2$-component} of a group $X$ is a subnormal perfect subgroup $F$ of $X$  such that $F/O(F)$ is quasisimple. The subgroup $L_{2'}(X)$ is defined to be the subgroup of $X$ generated by all the $2$-components of $X$ and is called the \emph{$2$-layer} of $X$. The subgroup  $X^\infty$ of $X$ is the last member in the derived series of $X$.

\begin{lemma}\label{lem:C Core} We have $$C_{L_{2'}(X)}(O(L_{2'}(X)))=  E(X)Z(O(L_{2'}(X)))O_2(L_{2'}(X)).$$
\end{lemma}

\begin{proof} Plainly $C_{L_{2'}(X)}(O(L_{2'}(X)))\ge   E(X)Z(O(L_{2'}(X)))O_2(L_{2'}(X))$. We may as well suppose that $X= L_{2'}(X)$. Set $\ov X= X/O(X)$ and $C= C_{X}(O(X))$. Then $\ov X= E(\ov X)$  by \cite[Proposition 4.7 (iii)]{gls1}. Therefore  $\ov C$ is a product of components of $\ov X$ together with  $\ov{O_2(X)}.$ Assume that $K \le C$ is such that $\ov K$ is a component  in $\ov C$. Then $KO(X)$ is normal in $X$ and $K^\infty$ is a $2$-component of $X$. If $K^\infty$ is not a component of $X$, then $O(K^\infty)\not \le Z(K^\infty)$. As $K \le C$, this is impossible.  Hence $K \le E(X)$.  Thus $\ov C= \ov {E(X) O_2(X)}$ and consequently  $C \le E(X)O_2(X) O(X)$. Using the Dedekind Modular Law we obtain \begin{eqnarray*}C&=& C \cap E(X)O_2(X) O(X)= E(X)O_2(X)(C \cap O(X))\\& =& E(X)O_2(X)Z(O(X)),\end{eqnarray*}
 as claimed.
 \end{proof}

\begin{lemma}\label{lem:compsnotnormal} Suppose that $K$ is a component of the group $X$ and $T \in \syl_2(X)$. If $Y$ is an abelian normal  subgroup of $T$ and $Y$ does not normalize $K$, then  $K/Z(K)$ has abelian Sylow $2$-subgroups.
\end{lemma}

\begin{proof} See \cite[Lemma 2.28]{ParkerRowleySymplectic}.  \end{proof}

\begin{lemma}\label{lem:non-iso chief factors}
Assume that $R \leq G$ be a $2$-group which normalizes the subgroup $P \le N_G(R)$. Set $V= [O_2(P),O^2(P)]$ and assume that the non-central $O^2(P)$-chief factors in $V/\Phi(V)$ are pairwise non-isomorphic. Then $[V,R] \le \Phi(V)$.
\end{lemma}

\begin{proof} Set $\ov V = V/\Phi(V)$. As $R$ normalizes $P$, $R$ operates on $\ov V$ and, as $R$ is normalized by $P$, coprime action yields $$[R, O^2(P)]= [R,O^2(P),O^2(P)] \le [O_2(P),O^2(P)]=V \le C_{O_2(P)}(\ov V).$$

 Assume $R >C_R(\ov V)$  and select $x \in R\setminus C_R(\ov V)$ such that $x^2 \in C_R(\ov V)$. Then $\ov V> C_{\ov V}(x) $ and $[\ov V,x]\ne 1$  are $O^2(P)$-invariant  as $[x,O^2(P)]$ centralizes $\ov V$. Additionally,  $\ov V /C_{\ov V}(x) \cong [\ov V,x]$ as $O^2(P)$-modules. In addition, as $x^2 \in C_R(V)$, we have $C_{\ov V}(x) \ge [\ov V,x]$.   Thus, the condition on the non-central $O^2(P)$-chief factors in $\ov V$ implies that $\ov V /C_{\ov V}(x)$ is centralized by $O^2(P)$. But then $V= [V,O^2(P)] < V$, a contradiction. Hence $R =C_R(\ov V)$ as claimed.
\end{proof}

We recall that the   \emph{Thompson subgroup $J(X)$} of a group $X$, is the subgroup of $X$ generated by the elementary abelian subgroups of  $X$ of maximal rank.
One of the main tools in the proof of the theorem of this paper requires that we locate $J(O_2(M))$ in certain subgroups  of centralizers of elements in $Y_M$.
 The next two results  are important  when such subgroups have  more than one component.

\begin{proposition}\label{prop:JS normalizes K}  Suppose that $X$ is a group, $O(X) = 1$, $K$ is a component of $X$ and $S$ is a Sylow $2$-subgroup of $X$. Assume that $K$ satisfies the following two properties.
\begin{itemize}
\item[(i)] For $\bar{x} \in K/Z(K)$  an involution, there is a preimage $x$ such that
\begin{enumerate}
\item[(a)]    $x$  is an involution; and
\item[(b)] any involution in $\Aut(K)$, which centralizes $\bar{x}$ also centralizes $x$.
\end{enumerate}
\item[(ii)] If $K/Z(K)$ has dihedral or semidihedral Sylow $2$-subgroups, then $\Aut(K/Z(K))$ does not contain a fours-group disjoint from $\Inn(K/Z(K))$.
\end{itemize}
Then $J(S)$ normalizes $K$.
In particular, if  $K \in \mathcal{K}$ is simple, then $J(S)$ normalizes $K$.
\end{proposition}

\begin{proof} This is \cite[Proposition 8.5]{gls1} and the remark thereafter.
\end{proof}

\begin{lemma}\label{lem:J structure}
Suppose $G= SE$ where $S \in \Syl_2(G)$ and $E$ is a direct product of   simple components $K \in \mathcal{K}$ of $G$.  Assume that  each   component   $K$ of $E$ satisfies \begin{eqnarray}\label{eq.1} \text{ if }T \in \syl_2(\Aut(K)), \text{ then  }J(T) = J(T \cap \Inn(K)).\end{eqnarray}  Then $$J(S)= J(C_S(E)) \times J(S\cap E)$$ and $$J(S\cap E) = \prod _{K \text{ a component of G}}J(S\cap K).$$
\end{lemma}

\begin{proof}  Assume that $A$ is an elementary abelian $2$-subgroup of $S$ of maximal rank. By \fref{prop:JS normalizes K} every component $K$ of $G$ is normalized by $A$. Furthermore  $AC_G(K)/C_G(K)$ is an elementary abelian $2$-subgroup of $\Aut(K)$. Assume that  $AC_G(K)/C_G(K)$ is not a maximal rank elementary abelian $2$-subgroup of  $ N_G(K)/C_G(K)$. Then, by   (\ref{eq.1}),  there exists an elementary abelian $p$-subgroup $B \le K$ such that $$|B|=|BC_G(K)/C_G(K)| > |AC_G(K)/C_G(K)| = |A:A\cap C_G(K)|.$$ Set $B_0 = B(A \cap C_G(K))$.  Then $B_0$ is elementary abelian and $|B_0|= |B||A\cap C_G(K)| > |A|$, contrary to the choice of $A$.  Hence  $AC_G(K)/C_G(K)$ is a maximal rank elementary abelian $2$-subgroup of  $ N_G(K)/C_G(K)$ and therefore $A \le K C_G(K)$ and $A= (A\cap K)(A \cap C_G(K))$ with $A\cap K$ a maximal rank elementary abelian $2$-subgroup of $K$.
Assume that $E = K_1 \cdots K_\ell$.  Then, for $1 \le j \le \ell$,  we have shown that $$A= C_A(K_j)(A \cap K_j)$$ where $A \cap K_j$ is a maximal rank elementary abelian $2$-subgroup of $K_j$ and $C_A(K_j)$ is a maximal rank elementary abelian $2$-subgroup of $C_S(K_j)$. Notice that by the Modular Law \begin{eqnarray*}A&=&C_A(K_1)(A\cap K_1)\cap C_A(K_2)(A \cap K_2)\\&=&  (C_A(K_1)(A\cap K_1)\cap C_A(K_2))(A \cap K_2)\\&=&  (C_A(K_1)\cap C_A(K_2))(A\cap K_1)(A \cap K_2)\\&=& C_A(K_1K_2)(A\cap K_1)(A \cap K_2)\end{eqnarray*} and continuing in this way yields $A= C_A(E)(A\cap K_1) \dots (A\cap K_\ell)$. This proves the claim.
\end{proof}

\begin{remark}
\rm{  If $K$ is a simple group, then the statement in \fref{prop:JS normalizes K} can be proved for all primes $p$ provided $O_{p^\prime}(X) = 1$, where we do not need $K \in \mathcal{K}$ for $p > 2$.  The statement of Lemma~\ref{lem:J structure} also holds for  all primes.}
\end{remark}

\section{Properties of $\mathcal K$-groups}\label{sec:Kgrp}

We  require detailed information about the $2$-local structure of certain of the  groups of Lie type defined in characteristic $2$. What we require can mostly be found in \cite{Memoir}, but we present the statements here for the convenience of the reader.  We start with groups defined over a field of characteristic $2$. In the next lemma we use the notation $V_n$ to denote a natural module for a classical group defined in dimension $n$ but considered as a $\GF(2)$-module. Thus, if $X$ is a classical group defined over $\GF(2^e)$, then $|V_n|= 2^{ne}$.

\begin{lemma}\label{lem:irr} Let  $X$ be a simple group of Lie type defined in characteristic $2$ and $R$ be a long root subgroup of $X$. Set $Q = O_2(N_X(R))$ and $L = O^{2^\prime}(N_X(R)/Q)$. Then   for specified $X$, the following table displays  the Levi section $L/Z(L)$, the $2$-rank of $Q/R$ and, for the classical groups $X$, describes the action of $L$ on $Q/R$.

{\tiny $$\begin{array}{|c|c|c|c|} \hline
X & L/Z(L) & m_2(Q/R) & Q/R\\ \hline
\PSL_m(2^e),m\geq 5 & \PSL_{m-2}(2^e) & 2(m-2)e & V_{m-2}\oplus V_{m-2}^*  \\
\PSU_m(2^e),m\geq 5 & \PSU_{m-2}(2^e) & (m-2)2e & V_{m-2} \\
\POmega_{2m}^{\pm}(2^e), m\geq 4 & \PSL_2(2^e)\times \POmega_{2(m-2)}^{\pm}(2^e) & 4(m-2)e &V_2\otimes V_{2m-4} \\
\POmega_{6}^\pm(2^{e})&\PSL_2(2^{e})&4e&V_2 \oplus V_2\\
\E_6(2^e) & \PSL_6(2^e) & 20e & \\
{^2}\E_6(2^e) &\PSU_6(2^e) & 20e & \\
\E_7(2^e) & \POmega_{12}^+(2^e) & 32e & \\
\E_8(2^e) & \E_7(2^e) & 56e & \\
\G_2(2^{e}), e\ge 3& \PSL_2(2^{e})& 4e & \\
^3\mathrm D_4(2^{e}) & \PSL_2(2^{3e})& 8e&\\\hline
\end{array}$$ }

\noindent Furthermore, other than for $X\cong \PSL_m(2^e)$ and $\POmega^\pm_6(2^{e})$, $Q/R$ is an irreducible $L$-module and, for the exceptional groups, it is defined over $\GF(2^e)$. If $X \cong \POmega_6^-(2^e)$, then $C_X(R)$ acts irreducibly on $Q/R$. \end{lemma}

\begin{proof} This is \cite[Lemmas D.1 and D.10]{Memoir}. \end{proof}

\begin{lemma}\label{lem:irr2} Let  $X\cong \PSL_4(2^a)$ with $a > 1$, $R $ be a root subgroup of $X$, $L= C_X(R)$ and $Q= O_2(L)$. Then the non-central chief factors of $Q/R$ are not isomorphic as $L$-modules.
\end{lemma}

\begin{proof} This is checked by direct calculation. Let $\lambda$ be a primitive element in $\GF(2^a)$ and put $\delta=\diag(\lambda,\lambda^{-2},1,\lambda)$.  Then $\delta$ is non-central in $X$ and centralizes $Z(S)$, where $S$ is taken to be the subgroup of lower unitriangular matrices. Let $$E_1= \left\{\left(\begin{smallmatrix} 1&0&0&0\\\alpha&1&0&0&\\\beta&0&1&0\\\gamma&0&0&1\end{smallmatrix}\right) \mid \alpha,\beta,\gamma\in \GF(2^a)\right\}$$ and
$$E_2= \left\{\left(\begin{smallmatrix} 1&0&0&0\\0&1&0&0&\\0&0&1&0\\\gamma&\beta&\alpha&1\end{smallmatrix}\right) \mid \alpha,\beta,\gamma\in \GF(2^a)\right\}.$$
Then $Q=E_1E_2$ and we calculate that conjugation of $E_1$ by $\delta$ scales $\beta$ by $\lambda$ and conjugation of $E_2$ by $\delta$ leaves $\beta$ unchanged. It follows that the  $\langle \delta\rangle$-invariant subgroups of $Z_2(S)$ are in $Z_2(S)\cap E_1$ or $Z_2(S)\cap E_2$. From this we deduce that $E_1$ and $E_2$ are the only normal subgroups of $L$ contained in $Q$ which have  order $2^{3a}$. This proves the result.
\end{proof}

\begin{lemma}\label{lem:rank1-Thompson}  Suppose that $K \cong \SL_2(2^{e+1})$ or ${}^2\B_2(2^{2e+1})$ with $e \ge 1$. Let $T \in \syl_2(\Aut(K))$. Then either   $K \cong \SL_2(2^2)$ or $J(T) = J(T \cap \Inn(K))= \Omega_1(T\cap \Inn(K))$.
\end{lemma}
\begin{proof} We identify $K$ with $\Inn(K)$.
  If $K$ is a Suzuki group then $T\le  K$ by \cite[Theorem 2.5.12]{gls2} and $\Omega_1(T\cap K)= Z(T\cap K)$ and we are done. So suppose that $K \cong \SL_2(2^{e+1})$. Let $x \in T \setminus K$ be an involution. Then $x$ acts as a field automorphism on $K$ and $e+1$ is even.  Thus $C_{K}(x) \cong \SL_2(2^{(e+1)/2})$ by \cite[Theorem 4.9.1]{gls2}.

   Assume that $A \le T$ has maximal rank. Then $|A| \ge 2^{e+1}$ and $T\cap K$ is elementary abelian of order $2^{e+1}$. Assume $A \not \le T$. Then
$$1+(e+1)/2  \ge e+1$$ as $K$ has $2$-rank $e+1$.  Hence either $K \cong \SL_2(4)$ or $J(T) \le K$.  In the latter case, we have $J(T) = J(T \cap  K)= \Omega_1(T\cap K)=T \cap K$.
\end{proof}

We need the following well-known result about representations of $\SL_2(2^e)$.

\begin{lemma}\label{lem:fourL2} Let $V$ be a non-split extension of a trivial module by the natural
module for $X = \SL_2(2^e)$. Let $S$ be a Sylow $2$-subgroup of $X$ and $A$ be a fours-group in $S$. Then $[V,A] = [V,S]$.
\end{lemma}

\begin{proof}  By a result of Gasch\"utz \cite[Satz I.17.4]{Hu}, we may assume that $C_V(X) \leq [V,S]$. Hence, if $[V,A] \not= [V,S]$,  as $[V/C_V(X),S] = [V/C_V(X),A]$, there is a hyperplane in $C_V(X)$ which contains $[V,A] \cap C_V(X)$. Thus   we may assume that $|C_V(X)| = 2$.  Choose  $\nu \in X$, of order $q+1$ and $\nu^{a} = \nu^{-1}$ for
some $a~\in~A$. We have that $|[V,\nu]| = q^2$ has index $2$ in $V$ and $V = [V,\nu]+C_V(X)$.  Therefore $[V,a] \leq [V,\nu]$.
Let $A = \langle a,b \rangle$. We have that $[V,\nu]+ [V,b] $ is invariant under
$\langle A, \nu \rangle = X$. Hence  $ [V,\nu] + [V,b]  = V$
and so $[V,A] > [V,a]$, which implies $C_V(X) \leq [V,A]$ and then $[V,A] = [V,S]$.
\end{proof}

\begin{lemma}\label{lem:spstruk} Suppose that $X \cong \Sp_{2n}(q)$ with $q = 2^e$  and $n \geq 3$, and let  $R_1$ be a long  root subgroup  and $R_2$ be a short root subgroup of $X$. For $i=1,2$, set $Q_i = O_2(N_X(R_i))$ and $$L_i = O^{2^\prime}(N_X(R_i)/Q_i).$$  Then

    \begin{enumerate}
    \item $L_1 \cong \Sp_{2n-2}(q)$, $Q_1$ is elementary abelian and $Q_1/R_1$ is a natural $\Sp_{2n-2}(q)$-module; and
    \item  $L_2 \cong \Sp_{2n-4}(q) \times \SL_2(q)$,  $\Phi(Q_2)=Q_2^\prime = R_2$, $Z(Q_2)/R_2$ is a natural $\SL_2(q)$-module and $Q_2/Z(Q_2)$ is the tensor product of   natural modules of the direct factors of $L_2$.  In addition, if $q > 2$, then $Z(Q_2)$ does not split over $R_2$ as an $L_2$-module.
        \end{enumerate}
        \end{lemma}
\begin{proof} See \cite[Lemma D.5]{Memoir}.\end{proof}
\begin{lemma}\label{lem:JS Sp} Suppose that $X\cong \Sp_{2n}(q)$, $q = 2^e$, with $n \ge 3$. Let $V$ be the natural  symplectic module, $P$ be the stabilizer of a maximal isotropic subspace of $V$ and $S \in \syl_2(P)$. Then $J(S) = O_2(P)$ is elementary abelian.
\end{lemma}

\begin{proof} See \cite[Lemma D.6]{Memoir}.\end{proof}

\begin{lemma}\label{lem:sp4sylow} Suppose that  $X \cong \PSp_4(q)$, $q = 2^e > 2$,  let $T$ be a Sylow $2$-subgroup of $\Aut(X)$ and set $S= T \cap X$. Then $X$ has exactly  two parabolic subgroups $P_1$, $P_2$
which contain $S$.  For $i = 1,2$,  $E_i = O_2(P_i)$ is elementary abelian of order $q^3$ and $P_i/E_i \cong \GL_2(q)$. We have that $E_i$ is an
indecomposable module for $P_i$  and  $Z(O^{2^\prime}(P_i)) = R_i$ is a root group. Furthermore $Z(S) = R_1R_2 = S^\prime$, $J(T)= J(S)=S = E_1E_2$ and any
involution in $S$ is contained in $E_1 \cup E_2$.
\end{lemma}

\begin{proof} See \cite[Lemmas D.3 and D.4]{Memoir}. \end{proof}

\begin{lemma}\label{lem:sp4sylow+} Suppose that  $X \cong \PSp_4(q)$, $q = 2^e > 2$, and $S \in \syl_2(X)$. If $D$ is a non-abelian normal subgroup of $S$, then either $Z(S) \le D$ or $C_S(D)= Z(S)$ and $|DE_1/E_1|= |DE_2/E_2|=2$.
\end{lemma}

\begin{proof} We use the notation from \fref{lem:sp4sylow}.
Assume that $Z(S) \not \le D$.  Then $|DE_1/E_1|= |DE_2/E_2|=2$ for otherwise $Z(S)=[E_1,D]$ by \fref{lem:fourL2}. Since $D$ is non-abelian, $|D Z(S)/Z(S)| \ge 4$. Hence there exists $t_i \in (E_i\cap D)\setminus Z(S)$ for $i=1,2$. As $C_{E_{3-i}}(t_i)= Z(S)$ by the last line of \fref{lem:sp4sylow}, we have $C_S(t_i)= E_i$.  Therefore $C_S(D) \le E_1 \cap E_2= Z(S)$.
\end{proof}

\begin{lemma}\label{lem:ZK1} Suppose that $X$ is quasisimple and $X/Z(X) \cong  \PSL_3(4)$ and $S \in \syl_2(X)$.  If $Z(X)$ has an element of order $4$, then $Z(S)\le Z(X)$.
\end{lemma}

\begin{proof}  See \cite[Chapter 10, Lemma 2.3 (a)]{gls5}.
\end{proof}

\begin{lemma}\label{lem:L3qsylow}  Suppose that $X$ is a group with $F^*(X) \cong \PSL_3(2^e)$, $e \ge 1$. Let $T \in \syl_2(X)$ and $S = T \cap F^*(X)$.  Then
\begin{enumerate}
\item $F^\ast(X)$ possesses exactly two parabolic subgroups $P_1$, $P_2$ which
contain $S$. For $i=1,2$, $E_i = O_2(P_i)$ is elementary abelian of order $2^{2e}$, $O^{2^\prime}(P_i/E_i) \cong \SL_2(2^e)$ and  $E_i$ is a natural module for $O^{2'}(P_i)$. Furthermore $S = E_1E_2$ and any involution in $S$ is contained in $E_1 \cup E_2$.
\item  every  elementary abelian normal subgroup of $T$  is contained in $S$;
\item $J(T)= J(S)=E_1E_2$.
\end{enumerate}
\end{lemma}

\begin{proof} See \cite[Lemmas D.2 and  D.4]{Memoir}.
\end{proof}

\begin{lemma}\label{lem:centsylow} Suppose that $X$ is  a group of Lie type
in characteristic $2$. If  $\sigma$ is an automorphism of $X$ of order $2$ which centralizes a Sylow $2$-subgroup of $X$, then  either $\sigma$ is inner or $X \cong \PSp_4(2)^\prime$.
\end{lemma}

\begin{proof} This follows from \cite[Chapter 19]{AschSe} when $X \not \cong {}^2\F_4(q)$.  For $X \cong {}^2\F_4(q)$ we can use  \cite[Theorem 9.1]{gls2} for $q > 2$ and for $q = 2$ the result follows from  \cite[Theorems 2.5.12, 2.5.15 and 3.3.2]{gls2}.
\end{proof}

\begin{lemma}\label{lem:L3qsylow1} Suppose that $X$ is a group with $F^*(X) \cong \PSL_3(2^e)$, $e \ge 1$. Let $T \in \syl_2(X)$ and $S = T \cap F^*(X)$.  Then $C_T(S) = Z(S)$.
\end{lemma}

\begin{proof}  Set $Y= C_T(S)$. Then $Y$ is normalized by $B=N_{F^*(X)}(S)$.  Let $C$ be a Cartan subgroup of $B$, then $Y= C_Y(C) [Y,C]$ and $[Y,C]= Z(S)$.  In particular, if $C_Y(C) \ne 1$, then $C_Y(C)$ contains an involution. This contradicts \fref{lem:centsylow}.
\end{proof}

\begin{lemma}\label{lem:L34facts} Suppose that $X$ is quasisimple with $X/Z(X) \cong \PSL_3(4)$ and $Z(X)$ elementary abelian. Let $T \in \syl_2(\Aut(X))$, $S= T\cap X \in \syl_2(X)$ and $\ov X = X/Z(X)$.
\begin{enumerate}
\item $\ov S$ has exactly two elementary abelian subgroups $\ov{E_1}$ and $\ov{E_2}$ of order $16$. Every involution of $\ov S$ is in $\ov {E_1} \cup \ov {E_2}$, $\ov S= \ov{E_1}\ov{E_2}=J(\ov T)=\ov{J(T)}$ and $C_{\ov S}(x )= \ov {E_i} $ for all $x \in \ov{E_i}\setminus Z(\ov S)$. For $i= 1,2$, let $E_i$ be the preimage of $\ov{E_i}$. Then $E_i$ is elementary abelian.

\item $[S,E_1]=[S,E_2]= S'=Z(S) = E_1\cap E_2 \ge Z(X)$.
\item If $D$ is a non-abelian normal subgroup of $S$, then $[\ov S,\ov D]= Z(\ov S)= \ov{Z(S)}= C_{\ov S}(\ov D)$.
\item Every normal elementary abelian subgroup of $T$ is contained in~$S$.
\end{enumerate}
\end{lemma}

\begin{proof} For part (i) and (iv) see \fref{lem:L3qsylow} and \cite[Chapter 10, Lemma 2.1 (h) and   2.2]{gls5}.

We now prove (ii). Since $E_1$ is elementary abelian, we may regard it as a $\GF(2)H$-module for  $H = N_X(E_1)$.   As $E_1$ centralizes $\ov{E_1}$, \fref{lem:L3qsylow} implies that $N_X(E_1)$ induces the natural $\SL_2(4)$-module on $\ov {E_1}$. We claim $[E_1,S]= Z(S) \ge Z(X)$. Certainly we have $$[E_1,S]=[E_1,E_1E_2]=[E_1,E_2]\le E_1 \cap E_2 \le Z(S)$$ and $[E_1,S]Z(X)= E_1 \cap E_2= Z(S)$. To prove that $Z(X) \le [E_1,S]$, we may suppose that $|Z(X)|=2$. Suppose that $[E_1, S] < E_1 \cap E_2$.  Then $[E_1,S]\cap Z(X)=1$. For $x \in N_X(E_1)\setminus N_X(S)$, we have   $N_X(E_1)= \langle S,S^x\rangle$ and so $[E_1,S][E_1,S^x] $ as order $2^4 $ and is normalized by $N_X(E_1)$.  Since $E_2$ is elementary abelian, we obtain $S/[E_1,S][E_1,S^x] $ is elementary abelian. Hence $N_X(E_1)/[E_1,S][E_1,S^x] $ splits as $2 \times \SL_2(4)$.  It follows that $S$ splits over $Z(X)$ and we have a contradiction via Gasch\"utz's Theorem \cite[(I.17.4)]{Hu}. Hence (ii) holds.

For (iii), suppose that $D$ is a non-abelian normal subgroup of $S$. Then $D \not \le E_1$ and so $[\ov{E_1},\ov D]= Z(\ov S)= \ov {Z(S)}$ as $\ov {E_1}$ is a natural $N_X(E_1)/E_1$-module. We now determine $C_{\ov S}(\ov D)$. We have that $\ov D$ has order at least $16$ and $\ov D$ contains  $Z(\ov S)$. If $\ov D \cap \ov{E_1}> Z(\ov S)$, then $C_{\ov{S}}(\ov D) \le C_{\ov E_1}(\ov D) = Z(\ov S)$,  the assertion. So assume $\ov D \cap \ov{E_1}= Z(\ov S)$. Then $S= DE_1$ and $|\ov D| = 16$. Thus  we can apply \fref{lem:fourL2} to see that $D \ge [E_1,D]=[E_1,S]= Z(S)$.  In particular  $\ov{ Z(S)} = Z(\ov S) = Z(\ov D) = C_{\ov S}(\ov D)$, as claimed.
\end{proof}

\begin{lemma}\label{lem:l34J} Let $X$ be quasisimple with $X/Z(X) \cong \PSL_3(4)$ and $Z(X)$ elementary abelian. Then $X$ satisfies   assumption (i)  of \fref{prop:JS normalizes K}.
\end{lemma}

\begin{proof}
 By \cite[Corollary 5.1.4]{gls2} we can lift $\Aut(X)$  to  a group of automorphisms of the universal covering group of $X/Z(X)$ and then restrict it to a group $X_1$ such that $|Z(X_1)|$ is elementary abelian of order 4 and $X_1/Z(X_1) \cong X/Z(X)$. Hence it is enough to prove the assertion when $|Z(X)| = 4$.

We follow the notation in \fref{lem:L34facts}. Set $P= N_X(E_1)$. Since $E_1$ is elementary abelian by \fref{lem:L34facts} (ii) and $X/Z(X)$ has just one conjugacy class of involutions, there are no elements of $X$ of order $4$ with square in $Z(X)$.  This is the condition (i)(a) of \fref{prop:JS normalizes K}.

Assume that $\ov x$ is an involution in $X/Z(X)$ and let $Y$ be the preimage of $\langle \ov x \rangle$. Then $Y$ is elementary abelian of order $8$.  We will show that there is some $x \in Y \setminus Z(X)$ which is centralized by any   automorphism of $X$ centralizes $\ov x$. From \cite[Table 6.3.1]{gls2} we know  $\Out(X/Z(X)) \cong \Sym(3) \times 2$ and acts on $Z(X)$ with an element  of order three  non-trivial.
Since $\Inn(X)$ acts transitively on the involutions in $X/Z(X)$, $N_{\Aut(X)}(Y) \Inn(X) = \Aut(X)$. As $C_{\Inn(X)} (Y) = \ov T$, and $|Y|= 2^3$, the subgroup structure of $\SL_3(2)$ yields  $N_{\Aut(X)}(Y)/C_{\Aut(X)}(Y)\cong \Sym(3)$.   Let $\rho \in N_{\Aut(X)}(Y)$  have order three.  Then  $Y = [Y,\rho] \times \langle x \rangle$ and $\langle x \rangle$ is centralized by $N_{\Aut(X)}(Y)$. Thus  $x$ is a preimage of $\ov x$, which is centralized by any automorphism which normalizes $Y$.  This element   satisfies the assumption (i)(b) of \fref{prop:JS normalizes K}.
\end{proof}

\begin{lemma}\label{lem:F42struk}  Suppose that $X \cong \F_4(q)$ with $q = 2^e$ and let  $R_1$ be a long  root subgroup  and $R_2$ be a short root subgroup of $X$. For $i=1,2$, set $Q_i = O_2(N_X(R_i))$ and $L_i = O^{2^\prime}(N_X(R_i)/Q_i).$  Then,
for $i = 1,2$, we have  $L_i\cong \Sp_6(q)$ and  $\Phi(Q_i)= R_i.$ Furthermore, as $L_i$-modules,  $Z(Q_i)/R_i$ is a natural module of dimension $6$, $Q_i/Z(Q_i)$ is a spin module of dimension $8$ and the modules $Z(Q_i)$ and $Q_i/R_i$ are indecomposable.
\end{lemma}

\begin{proof} See  \cite[Lemma D.7]{Memoir}. \end{proof}

\begin{lemma}\label{lem:F42middle}  Suppose that $X \cong \F_4(q)$ with  $q = 2^e$, $S \in \Syl_2(X)$ and $\Omega_1(Z(S)) = R_1R_2$ with $R_1$ a long root subgroup of $X$ and $R_2$ a short root subgroup of $X$. We use the notation introduced in \fref{lem:F42struk} and additionally set $I_{12} = C_X(R_1R_2)$, $Q_{12}= O_2(I_{12})$ and  $L_{12} = I_{12}/Q_{12}$.  For $i=1,2$, define $$V_i = [Z(Q_i),Q_{12}]R_1R_2,$$
 put $V_{12}= V_1V_2$ and $W_{12}=Z(Q_1)Z(Q_2).$

Then the following hold:
\begin{enumerate}
\item $L_{12} \cong \Sp_4(q)$ and  $Q_{12}= Q_1Q_2$.
\item  $V_{12}$ and $W_{12}$ are normal in $I_{12}$ and
$$1 < R_1R_2 < V_{12} <W_{12} < Q_{12}.$$
In addition, we have $Z(Q_1) \cap Z(Q_2)= R_1R_2$, $Q_1 \cap Q_2= V_{12}$ is elementary abelian  and, setting  $\ov {V_{12}}= V_{12}/R_1R_2$, $$\ov{V_{12}}= \ov{V_1} \oplus \ov{V_2},$$ where  $\ov{V_1}$ and $\ov{V_2}$  are irreducible $L_{12}$-modules of $\GF(q)$-dimension $4$ which are not isomorphic as $\GF(2)L_{12}$-modules. Furthermore, if $q >2$, $W_{12}'= R_1R_2$ whereas, if $q=2$, $W_{12}'= \langle r_1r_2\rangle$ where $r_i \in R_i^\#$.
\item  $[V_{12},W_{12}] = 1$ and $W_{12}/V_{12}$ has order $q^2$ and  is  centralized by $L_{12}$.
 \item We have  $$Q_{12}/W_{12} \cong Q_1W_{12}/W_{12} \oplus Q_2W_{12}/W_{12},$$ $Q_1W_{12}/W_{12}$ and $Q_2W_{12}/W_{12}$ are irreducible,  non-isomorphic   $L_{12}$-modules of $\GF(q)$-dimension $4$. Furthermore, as $L_{12}$-modules, for $i=1,2$, $$Q_iW_{12}/W_{12}  \cong V_{3-i}/R_1R_2.$$
 \item We have $$Q_{12}/V_{12}= Q_1/V_{12} \oplus Q_2/V_{12}$$ is a direct sum of two indecomposable $L_{12}$-modules of $\GF(q)$-dimension $5$.
 \item The group $\Aut(Q_{12})$ has a subgroup of index $2$ which normalizes all of $R_1$, $R_2$, $Q_1$, $Q_2$, $Z(Q_1)$, $Z(Q_2)$, $V_{12}$ and $W_{12}$.
\end{enumerate}
\end{lemma}
\begin{proof} See \cite[Lemma D.8]{Memoir}. \end{proof}

\begin{lemma}\label{lem:F4twiststruk} Suppose that $X \cong  {}^2\F_4(q)$ with $q = 2^{2e+1}$,  $S\in \syl_2(X)$, $R$ is a long root subgroup in  $Z(S)$,  $P =
C_X(R)$ and $Q= O_2(P)$. Then
\begin{enumerate}
\item $P/Q \cong {}^2\B_2(q)$.
\item $R = Z(Q)$, $Z_2(Q)$ is elementary abelian and  $Z_2(Q)/R$ is an irreducible $4$-dimensional module for $P/Q$.
\item  $C_{Q}(Z_2(Q))$ is non-abelian of order $q^6$, $\Phi(C_{Q}(Z_2(Q)))=R$ and $Q/C_{Q}(Z_2(Q))$ is the natural $P/Q$-module.
\item If $q> 2$, then $Q/Z_2(Q)$ is an indecomposable module.
\item If $q =2$, then $F^\ast(X) ={}^2\F_4(2)^\prime$ has index $2$ in $X$. We have that $R = Z(O_2(P \cap F^\ast(X)))$, $Z_2(Q) = Z_2(Q \cap
F^\ast(X))$ and $|(Q \cap F^\ast(X))/Z_2(Q)| = 16$. Furthermore, $(Q \cap F^\ast(X))/Z_2(Q)$ and $Z_2(Q)/R$  admit $P \cap F^*(X)$ irreducibly.
\item    Let $P_1= N_X(Z_2(S))$. Then $P_1$ is a maximal parabolic subgroup of $X$,  $P_1 \ne P$, $P_1$ normalizes $Z_3(S)$ which has  order $q^3$ and   $P_1$  induces $\GL_2(q)$ on $Z(O_2(P_1))=Z_2(S)$. Furthermore for $q > 2$, we have $W = \langle (Z_4(S) \cap Z_2(Q))^{P_1} \rangle$ is elementary abelian
    and $W/Z_3(S)$ is the natural $\GL_2(q)$-module. Further $C_S(Z_3(S))/W$ is an irreducible $4$-dimensional module for $\SL_2(q)$ and $O_2(P_1)/C_S(Z_3(S))$ is the natural $\SL_2(q)$-module.
\end{enumerate}

\end{lemma}

\begin{proof} For the structure of $P$ see \cite[Example 3.2.5, page 101]{gls2} or \cite[12.9]{DeSte}. For part (vi) we refer to \cite[12.9]{DeSte}.
\end{proof}

\begin{lemma}\label{lem:G2irr}  Suppose that  $X \cong \G_2(4)$, $S$ is a Sylow $2$-subgroup of $X$ and   $R$ is a long root subgroup  contained in $Z(S)$. Set $P = N_X(R)$, $Q = O_2(P)$ and $L = O^{2^\prime}(P)$. Then $Z(Q) = R = Q^\prime$, $L/Q \cong \SL_2(4)\cong \Alt(5)$, $P$ acts irreducibly on $Q/R$ while $L$ induces a direct sum of two natural $\Alt(5)$-modules on $Q/R$. Furthermore,   if $R<E\le Q$ is normalized by $L$, then $E$ is not abelian.
\end{lemma}

\begin{proof} See \cite[Lemma D.10]{Memoir}. \end{proof}

\begin{lemma}\label{lem:Sz8schur} Suppose that $X$ is quasisimple and $X/Z(X)\cong{}^2\B_2(8)$. Let $S$ be a Sylow $2$-subgroup of $X$. If $Z(X) \ne 1$, then $Z(S) = Z(X)$.
\end{lemma}

\begin{proof} We may assume that $|Z(X)| = 2$. There is an element $\nu$ of order $7$ normalizing $S$ such that $[Z(S/Z(X)),\nu] = Z(S/Z(X))$. Let $Y$ be the preimage of $Z(S/Z(X))$. Then $|[Y,\nu]| = 8$. Assume $Z(S) > Z(X)$, then $Z(S) = Y$ as $\nu$ normalizes $Z(S)$. Now $[Y,\nu]$ is normal in $S$ and so $S/[Y,\nu]$ is of order 16. Since  $C_{S/[Y,\nu]}(\nu) = Y/[Y,\nu]$ and $S/[Y,\nu]$ is not extraspecial,  $S/[Y, \nu]$ is elementary abelian. Thus $S= [S, \nu] \times Z(X)$ and Gasch\"utz's Theorem \cite[(I.17.4)]{Hu} provides a contradiction. Hence $Z(S)=Z(X)$.
\end{proof}

In the next lemma we adopt the notation introduced in \cite[Table 4.5.1]{gls2} for inner-diagonal and graph automorphisms of order 2 of groups of Lie type defined over fields of odd characteristic.

\begin{lemma}\label{lem:Lie odd invs}Suppose that $p$ is an odd prime and $K$ is quasisimple  with  $K/Z(K)$ a  group of Lie type defined in characteristic $p$. Let $\alpha \in \Aut(K)$ be an automorphism of order $2$. If $E(C_K(\alpha))=1$, then $C_K(\alpha)$ is soluble and one of the following holds where the bold face notation indicates an automorphism which centralizes a Sylow $2$-subgroup of $K$.
\begin{enumerate}
\item $K/Z(K) \cong \PSL_2(p^e)$ and either $\alpha$ is an inner-diagonal automorphism or $p^e=9$ and $\alpha$ is a field automorphism;
\item $K \cong \PSL_3(3)$ and $\alpha \in\{\mathbf{t_1}, \gamma_1\}$;
\item $K \cong \PSU_3(3)$ and $\alpha \in\{\mathbf{ t_1}, \gamma_1\}$;
\item $K/Z(K) \cong \PSL_4(3)$ and $\alpha \in \{\mathbf{t_2}, \gamma_2\}$;
\item $K/Z(K) \cong \PSU_4(3)$ and $\alpha \in \{\mathbf{t_2}, \gamma_2\}$;
\item $K/Z(K) \cong \PSp_4(3)$ and $\alpha\in \{\mathbf{t_1}, t_2, t_2'\}$;
\item $K/Z(K) \cong \mathrm P \Omega_7(3)$ and $\alpha = \mathbf{t_2}$;
\item $K/Z(K) \cong \POmega_8^+(3)$ and $\alpha= \mathbf {t_2}$;
\item $K/Z(K) \cong \G_2(3)$ and $\alpha= \mathbf {t_1}$; or
\item $K \cong {}^2\G_2(3)'$ and $\alpha= \mathbf {t_1}$.
\end{enumerate}
In particular, in all but case (i), $C_K(\alpha)$ is a $\{2,3\}$-group and $F^*(C_K(\alpha))$ is a $2$-group.
\end{lemma}

\begin{proof}
If $K /Z(K) \cong \PSL_2(p^e)$, then we read the statement from \cite[Table 4.5.1 and Proposition  4.9.1 (a) and (e)]{gls2}. Suppose that $K/Z(K) \not \cong \PSL_2(p^e)$.  Then \cite[Table 4.5.1 and Proposition  4.9.1 (a) and (e)]{gls2} yields   $p^e=3$ and that $C_K(\alpha)$ can only involve Lie components of type $\A_1(3)$ and $\mathrm D_2(3)$. This observation then leads to the groups listed.
\end{proof}

\section{Elementary properties of the configuration}\label{sec:basics}

For the convenience of the reader we  repeat the most important notions that we presented in the introduction.  The group $G$ is a $\mathcal K_2$-group,  $S$ is a Sylow $2$-subgroup of $G$ and $Q$ is a large subgroup of $S$. This means $C_G(Q) \leq Q$, $Q = O_p(N_G(Q))$ and $Q \unlhd N_G(A)$ for all $ 1\not= A \leq Z(Q)$. We define
$$\mathcal L_G(S) = \{S \leq L \leq G~|~O_2(L) \not= 1, C_G(O_2(L)) \leq O_2(L) \}$$
and for  $L \in\mathcal L_G(S)$ with $L \not \le N_G(Q)$,
set
 $$L^\circ =\langle Q^L\rangle.$$
Denote by $Y_L$ the largest normal elementary abelian subgroup of $L$ such that $O_2(L/C_L(Y_L)) = 1$ and set $C_L = C_L(Y_L)$.
\\
\\
Let $\mathfrak M_G(S)$ be the subset of those $M \in \mathcal L_G(S)$, for which $C_M$ is 2-closed, $C_M/O_2(M) \leq \Phi(M/O_2(M))$ and $M^\dagger = MC_G(Y_M)$ is the only maximal element in $\mathcal L_G(S)$ with $M \leq M^\dagger$. In particular, $Y_M = Y_{M^\dagger}$ by \cite[Lemma 1.24 (h)]{Struc}.

Suppose that $M\in \mathfrak M_G(S)$ and $T \in \Syl_2(C_G(Y_M))$. Then  $Y_M$ is
\begin{itemize}
\item  \emph{tall}, if there exists $K$ with $T \le K \le G$ such that $O_2(K) \ne 1$ and $Y_M \not \le O_2(K)$,
\item     \emph{characteristic $2$-tall} provided there is some $K$ with  $T\le K\le G$  such that  $C_K(O_2(K)) \le O_2(K)$ and $Y_M\not \le O_2(K)$,
and
  \item   \emph{asymmetric} in $G$, if whenever $g\in G$ and $[ Y_M,Y_M^g]\le Y_M \cap Y_M^g$, then $ [ Y_M,Y_M^g]=1$.
  \end{itemize}

We intend to prove the theorem of this paper by contradiction. Specifically, we work under the following hypothesis.

\begin{hypothesis}\label{hyp:MainHyp} The group  $G$ is a $\mathcal K_2$-group, $S \in \syl_2(G)$ is contained in at least two maximal $2$-local subgroups and $Q\le S$ is a large subgroup of $G$.  Furthermore,
 there exists $M \in \mathfrak M_G(S)$ such that $M \not \le N_G(Q)$ and $Y_M$ is asymmetric and tall but  not characteristic $2$-tall.
\end{hypothesis}

In this section we collect the rudimentary facts about the configuration of \fref{hyp:MainHyp}.

\begin{lemma} \label{lem:basic1} Suppose that $Q\le K \le G$ and $O_2(K) \ne 1$.  Then
\begin{enumerate}
\item $C_G(O_2(K))$ is a $2$-group;
\item $K$ has characteristic $2$; and
\item If, in addition, $O_2(C_M) \leq K$, then $Y_M \leq O_2(K)$.
\end{enumerate}
In particular, $Y_M \le Q$.
\end{lemma}

\begin{proof} Parts (i) and (ii) are \cite[Lemma 1.55 (a)]{Struc}.   Taking $T= S\cap C_G(Y_M)= O_2(M)$, part (iii) follows from the fact that $Y_M$ is not characteristic $2$-tall. The final statement follows from (iii) by taking $K= N_G(Q)$.\end{proof}

\begin{lemma}\label{lem:charO2M} If  $1\not= R \le O_2(M)$ is normalized by $M$, then $$M \le N_G(R) \le M^\dagger.$$

\end{lemma}

\begin{proof}
We have that $M \leq N_G(R)$. By assumption $M^\dagger$ is the unique maximal element in $\mathcal L_G(S)$, which contains $M$. As $N_G(R) \in \mathcal L_G(S)$ by \fref{lem:basic1}(ii),  $N_G(R) \leq M^\dagger$.
\end{proof}

Recall that if $X$ is a group, $A \le B \le X$, then $A$ is \emph{weakly closed} in $B$ with respect to $X$ provided whenever $x \in X$ and $A^x \le B$, then $A^x=A$.

\begin{lemma}\label{lem:weakcl} The following hold:
\begin{enumerate}
\item $O_2(M)\in \Syl_2(C_G(Y_M))$;
\item $Q$ is weakly closed in $S$ with respect to $G$;
\item $O_2(M)$ is weakly closed in $S$ with respect to $G$;
\item $Y_M = \Omega_1(Z(O_2(M))) \ge \Omega_1(Z(S))$;
\item $O_2(\langle   \mathcal L_G(S)\rangle)=1$;
\item if $ N \ge O_2(M)$ with $O_2(N) \ne 1$ and $N$ has characteristic $2$, then $Y_M \le O_2(N)$;
\item if $U$ is a $2$-group which is normalized by $O_2(M)$, then $U \leq M^\dagger$, in particular, $U$ normalizes $Y_M$;
\item $M^\circ = \langle Q^{M^\circ}\rangle$ and $[C_G(Y_M),M^\circ]\le O_2(M^\circ)$.
\end{enumerate}
\end{lemma}

\begin{proof} The first four parts come from \cite[Lemma 2.2(b), (e), (f) and (e) and Lemma 2.6(b)]{Struc}.

Part (v) follows from the fact that $\mathcal L_G(S)$ contains at least two maximal members and part (vi) is a consequence of $Y_M$ being non-characteristic $2$-tall and part (i).

For (vii) consider $O_2(M)U$, which is a $2$-group. By (iii) we have that $O_2(M)$ is normal in $UO_2(M)$ and so $U \leq N_G(O_2(M)) \leq M^\dagger$. Thus $Y_{M^\dagger}$ is normalized by $U$, since $Y_M= Y_{M^\dagger}$ by definition, $U$ normalizes $Y_M$.

Part (viii) follows from \cite[Lemmas 1.46 (c) and 1.52 (c)]{Struc}.
\end{proof}

We can now formulate the fact that $Y_M$ is not characteristic $2$-tall, in terms of $O_2(M)$: if $K \ge O_2(M)$ with $O_2(K) \ne 1$ and $Y_M \not \le O_2(K)$, then $F^*(K) \ne O_2(K)$.

\begin{lemma}\label{lem:McircZ(Q)}
If  $X$ is a non-trivial $2$-group normalized by $Q$, then $X$ does not centralize  $O^2(M^\circ)$.
\end{lemma}

\begin{proof} Suppose that $O^2(M^\circ)  \le C_G(X)$. As $Q$ normalizes $X$, we have  $Z(Q)\cap X \ne 1$ and so $O^2(M^\circ) \le N_G(Z(Q) \cap X)$. As $Q$ is large, $O^2(M^\circ) \le N_G(Q)$. Therefore, $$Q= \langle Q^{O^2(M^\circ)}\rangle = \langle Q^{QO^2(M^\circ)} \rangle = \langle Q^{M^\circ}\rangle= M^\circ$$ by \fref{lem:weakcl} (viii). Thus  $M \le N_G(Q)$, which is a contradiction.
\end{proof}

The next lemma plays a very important role in the proof of our theorem.
\begin{lemma} \label{lem:there exists y} There exists $y \in Y_M^\#$ such that $F^*(C_G(y)) \ne O_2(C_G(y))$. That is $C_G(y) $ does not have characteristic $2$. In particular, $M^\dagger$ does not act transitively on $Y_M^\#$.
\end{lemma}

\begin{proof} The first statement  is \cite[Theorem F (page 131)]{Struc}. The remaining part follows as $\Omega_1(Z(S)) \cap Y_M \ne 1$ and the centralizer of this group has characteristic $2$ by \fref{lem:basic1} (ii). \end{proof}

\begin{lemma}\label{lem:charpy} Suppose that  $y \in Y_M^\#$ and $1\ne R_1\le R \le C_G(y) $  are $2$-groups.
\begin{enumerate}
\item If $Q \le N_G(R_1)$ and $R \le N_G(Q)$, then $N_G(R)$, $C_G(R)$, $N_{C_G(y)}(R)$ and $C_{C_G(y)}(R)$ have characteristic $2$.
\item If $Z(Q) \cap R \ne 1$, then  $N_G(R)$, $C_G(R)$, $N_{C_G(y)}(R)$ and $C_{C_G(y)}(R)$ have characteristic $2$.
\item If   $R$ is a non-trivial   subgroup of $O_2(M)$ which is normalized by $M$, then  $N_{C_G(y)}(R)$ and $C_{C_G(y)}(R)$ have characteristic $2$.
    \item $N_{C_G(y)}(Y_M)$  has characteristic $2$.
    \item If $z \in Z(Q)^\#$ and $J$ is a subnormal subgroup of $C_G(y)$, then $C_J(z)$ has characteristic $2$.
    \item If $z \in Z(Q)^\#$  and $K$ is a component of $C_G(y)$, then $C_K(z)$ has characteristic $2$.
 \end{enumerate}
\end{lemma}

\begin{proof}  Suppose the hypotheses of (i) hold. Then, as $R_1$ is normalized by $Q$, we have  $R_1 \cap Z(Q)\ne 1$.  As $R_1 \leq R$ also $R \cap Z(Q) \not= 1$. Set $K = N_G(R \cap Z(Q))$. Then $ Q \leq K$ and $K$  has characteristic $2$. As $R$ normalizes $Q$, it also normalizes $R \cap Z(Q)$ and so $R \leq K$. Furthermore $C_G(R) = C_K(R)$. Now application of \fref{lem:fund-charp} (with $C = R \cap Z(Q)$, $X = G$ and $B = R$) yields  $N_G(R)$ and
$C_G(R)$ have  characteristic 2.

Since $y \in N_G(R)$, $N_{C_G(y)}(R)=C_{N_G(R)}(y)$ and $C_{C_G(y)}(R)=C_{C_G(R)}(y)$ have characteristic $2$ by \fref{lem:fund-charp} (with $C=1$, $B=\langle y\rangle$ and $X=N_G(R)$, $X = C_G(R)$, respectively). This proves (i).

For (ii) take $R_1=R \cap Z(Q)$, for (iii) take $R_1=R$, and then apply (i).

Part (iv) is a special case of (iii).

For (v), we take $R= R_1=\langle z\rangle$ and use (i) to get  $C_{C_G(y) }(z)$ has characteristic 2. By \fref{lem:basic1}, $Y_M \leq Q$ and so $[y,z] = 1$. As this  property passes to subnormal subgroups  by \fref{lem:subnormal2},  we have $C_J(z)$   has characteristic $2$. Part (vi) follows from (v).
\end{proof}

The next lemma is often used to help conclude that  $|Y_M|$ small.
\begin{lemma}\label{lem:YM in O2J} Suppose that    $J \leq G$ is normalized by $O_2(M)$ and  $J$ has characteristic $2$. Then $Y_M \unlhd O_2(JY_M)$ and $\langle Y_M^J\rangle$ is elementary abelian.

\end{lemma}
\begin{proof} We have $O_2(M) J$ has characteristic $2$ and $O_2(M) \in \syl_2(C_G(Y_M))$ by \fref{lem:weakcl}(i).  Since $Y_M$ is not characteristic $2$-tall, $Y_M \le O_2(O_2(M)J) $. Hence $$Y_M \le O_2(O_2(M)J) \cap Y_MJ\le O_2(JY_M).$$  By \fref{lem:weakcl}(vii), $Y_M $ is normal in $O_2(JY_M)$ and, as $Y_M$ is asymmetric, we also have $\langle Y_M^J\rangle$ is elementary abelian.
\end{proof}

 Define
$$U_Q =\langle Y_M^{N_G(Q)} \rangle.$$

\begin{lemma}\label{lem:NQclos} The following hold:
\begin{enumerate}
\item $Y_M \le U_Q\le Q\cap O_2(M)$ and $U_Q$ is elementary abelian;
\item $1\ne \Omega_1(Z(Q))\cap Y_M  < Y_M $; and
\item $Y_M \not \le [U_Q,Q]< U_Q$.
\end{enumerate}
\end{lemma}

\begin{proof} Since $N_G(Q)\in \mathcal L_G(S)$, $Y_M \le Q$ by \fref{lem:basic1}(iii)  and $U_Q$ is elementary abelian  by \fref{lem:YM in O2J}. Thus $U_Q \le C_Q(Y_M)= Q\cap C_S(Y_M) =Q \cap O_2(M) $ by \fref{lem:weakcl} (i).  This is (i).

If $ Y_M \le \Omega_1(Z(Q))$, then $M \le M^\dagger \le N_G(Y_M) \le N_G(Q)$ as $Q$ is large. This is against the choice of $M$ and so (ii) holds.

As $Q$ acts on $U_Q$, we have $[Q,U_Q] < U_Q$. As $[Q,U_Q]$ is normal in $N_G(Q)$, we get $Y_M \not\leq [Q,U_Q]$, which is (iii).
\end{proof}

\begin{lemma}\label{lem:more} Assume $N_G(Q) \leq M^\dagger$. Then there is at least one $L \in \mathcal L_G(S)$ such that $Y_L \not\leq Y_M$.
\end{lemma}

\begin{proof} Assume that for all $L \in \mathcal L_G(S)$,  $Y_L \leq Y_M$.  Then  $O_2(M) \le C_M \leq C_L$.
As $O_2(M) \leq S \cap C_L$ and $O_2(M)$ is weakly closed in $S$ by \fref{lem:weakcl}, we have that $N_L(S \cap C_L) \leq N_L(O_2(M)) \leq M^\dagger$ by \fref{lem:charO2M}.
As $C_L \leq N_G(Q) \leq M^\dagger$ by assumption, we get $L= N_L(S \cap C_L)C_L \leq M^\dagger$. Hence $\langle \mathcal L_G(S)\rangle \le M^\dagger$ and this contradicts \fref{lem:weakcl} (v).  Thus there exists $L\in \mathcal L_G(S)$ with $Y_L \not\leq Y_M$.
\end{proof}

We use the previous lemma as follows:

\begin{lemma}\label{lem:YMnotmaxabelian}
There exists an elementary  abelian normal subgroup of $S$ contained in $O_2(M)$ which strictly contains $Y_M$.  In particular, $Y_M \ne \Omega_1(O_2(M))$, $O_2(M) $ is not abelian and $Y_M \ne J(O_2(M))$.
\end{lemma}

\begin{proof} Suppose that $Y_M$ is a maximal elementary abelian subgroup of $O_2(M)$, which is normal in $S$.   By \fref{lem:NQclos} (i),  $Y_M \le U_Q\le O_2(M)$ and $U_Q$ is elementary abelian. Thus  $U_Q = Y_M$ and so $N_G(Q) \leq M^\dagger$. Let $L\in \mathcal L_G(S)$, then by \fref{lem:basic1}(iii) $Y_M \leq O_2(L)$ and so $[Y_L,Y_M]= 1$. Hence   $Y_L \leq C_S(Y_M) =O_2(M)$ by \fref{lem:weakcl}(i). As $Y_L$ is normal in $S$, we have $Y_L  \leq  Y_M$ by assumption. Now \fref{lem:more} yields a contradiction. This proves the first claim. Furthermore $Y_M< \Omega_1(O_2(M))$.

As   $Y_M \ne \Omega_1(O_2(M))$, $Y_M \ne J(O_2(M))$. That  $O_2(M) $ is not abelian, follows as $Y_M=\Omega_1(Z(O_2(M)))$ by \fref{lem:weakcl} (iv).
\end{proof}

We finish this section with a look at what happens when $Y_M$ has small order.

\begin{lemma}\label{lem:YM8}
We have $|Y_M |\ge   16$.
\end{lemma}

\begin{proof} Assume false.
Since $1 \not= \Omega_1(Z(Q))\cap Y_M \le Y_M  $,  \fref{lem:NQclos}(ii) implies  $|Y_M| = 4$ or $8$.  By \fref{lem:there exists y}  $M^\dagger/C_M$ cannot act transitively on $Y_M^\#$.

If $|Y_M| = 4$, then $|M^\dagger/C_M|= 2$, but by the definition of $Y_M = Y_{M^\dagger}$ we have that $O_2(M^\dagger/C_M) = 1$, a contradiction.

Thus $|Y_M| = 8$. Then $M^\dagger/C_M$ is a subgroup of $\SL_3(2)$ and, as $M^\dagger$ does not act transitively on $Y_M^\#$,  $M^\dagger/C_M$  is a $\{2,3\}$-group. In particular, $M^\dagger/C_M$ is soluble.  As $O_2(M^\dagger/C_M) = 1$, we have that $M^\dagger/C_M \cong \Sym(3)$ or is cyclic of order $3$. In both cases there is some $ w \in Y_M^\#$, with $M^\dagger \leq C_G(w)$. In particular $[w,Q] = 1$ and so, as $Q$ is large, $M^\dagger \le C_G(w) \leq N_G(Q)$, a contradiction.
\end{proof}

\section{The components of $C_G(y)$}\label{sec:comps}
By \fref{lem:there exists y},  there is some $ y \in Y_M^\#$ such that $F^\ast(C_G(y)) \not= O_2(C_G(y))$. In this section we  show that we can carefully select  $y$ such that $C_G(y)$ has a structure which   can be used to reach a contradiction in the sections which will follow.

\begin{lemma}\label{lem:char2}
Assume that $z \in Z(Q)$ and $y \in  C_S(z)$ are involutions. Then $C_{C_G(y)}(z)$ has characteristic $2$ and $z$ inverts $O(C_G(y))$. Furthermore, if   $K$ is a component of $C_G(y)$ which is normalized by $z$, then $K=[K,z]$ and, if $z$ induces an inner automorphism on $K$, then $Z(K)$ is a $2$-group.
\end{lemma}

\begin{proof} By \fref{lem:basic1}(ii), $C_G(z)$ has characteristic $2$ and therefore so does $C_{C_G(y)}(z)$ by \fref{lem:fund-charp}. In particular   $C_{O(C_G(y))}(z)=1$ and so $z$ inverts $O(C_G(y))$.

 Suppose that $K$ is a component of $C_G(y)$ which is normalized by $z$. If $z$ centralizes $K$, then $K$ is a component of $C_{C_G(y)}(z)$, a contradiction.  Hence $z$ acts non-trivially on $K$ and so $K=[K,z]$. Finally, if $z$ induces as an inner automorphism of $K$, then $K\langle z \rangle = K C_{K\langle z \rangle}(K) \le C_G(Z(K))$. As $z$ inverts $O(C_G(y))$, we infer that $Z(K)$ is a $2$-group.
\end{proof}

The next lemma is of fundamental importance.

\begin{lemma}\label{lem:sig} There exists $y \in Y_M^\#$ such that $E(C_G(y)) \ne 1$.
\end{lemma}

\begin{proof}   There exists $z \in C_{Y_M}(S)^\# \subseteq C_{Y_M}(Q)^\#$ and, for such $z$,  $C_G(z)$ has characteristic $2$ by \fref{lem:basic1} (ii). Furthermore, $|Y_M| \ge 16$ by \fref{lem:YM8}.

 Assume that for all $y \in Y_M^\#$,  $E(C_G(y)) = 1$. Then Lemmas \ref {lem:sig1} and \ref{lem:there exists y}   imply  that
 $$\Sigma=\langle O(C_G(b))\mid b \in Y_M^\#\rangle\ne 1$$   has odd order.

Because $M^\dagger$ permutes the elements of $Y_M^\#$,  $\Sigma$ is normalized by $M^\dagger$. Since $C_{M^\dagger}(\Sigma)$ is normal in $M^\dagger$ and $F^*(M^\dagger)$ is a $2$-group and is a maximal $2$-local subgroup of $G$ implies that $C_{M^\dagger}(\Sigma) =1$. In addition, as $O_2(C_G(z)) \le S \le M$,  $$[C_{\Sigma}( z),O_2(C_G(z))]\le \Sigma \cap O_2(C_G(z)) =1$$ and so $z$ inverts $\Sigma$. Hence $[z,M^\dagger] \le C_{M^\dagger}(\Sigma)=1$ and so $z \in Z(M^\dagger)$. But then $M^\dagger \le C_G(z) \le N_G(Q)$, a contradiction.
\end{proof}

From now on we focus our interest on the following subset of elements of $Y_M$:
$$\mathcal Y =\{y \in Y_M^\#\mid E(C_G(y)) \ne 1\},$$
which by \fref{lem:sig} is non-empty. We also put
$$\mathcal Y_S =\{y \in \mathcal Y \mid  C_S(y) \in \syl_2(C_G(y))\}.$$

From among all the components that appear in $C_G(y)$ for $y \in \mathcal Y$ select $C$ such that first $|C/Z(C)|$ is maximal and second that $|C|$  is maximal. Then for $y \in \mathcal Y $ set $$E_y= \langle J \mid  J \text{  is a component of } C_G(y), J/Z(J)\cong C/Z(C) \text{ and } |J|= |C| \rangle.$$
Let $$\mathcal Y^* = \left\{y\in \mathcal Y  \Bigm\vert \begin{array}{ll}(a)& \text{the number of components in }E_y \text{ is maximal}\\(b)&|E_y| \text{ maximal }\end{array}\right\},$$ and
$$\mathcal Y^*_S= \mathcal Y^* \cap \mathcal Y_S.$$

\begin{lemma}\label{lem:Cy}
For $y \in \mathcal Y$, there exists $g \in M$ such that  $y^g \in \mathcal Y_S$.
\end{lemma}

\begin{proof} As $O_2(M) \le C_G(Y_M) \le C_G(y)$, we may choose $R \in \syl_2(C_G(y))$ such that $O_2(M) \le R$. Then $R   \leq M^\dagger$ by \fref{lem:weakcl} (vii). Since $S \in \syl_2(M^\dagger)$, there exists $h \in M^\dagger$ such that $R^h\le S$.  Hence    $R^h = C_S(y^h) \in \syl_2(C_G(y^h))$.  As $M^\dagger = MC_{M^\dagger}(Y_M)$ we have $h = gh_1$ with $g \in M$ and $h_1 \in C_{M^\dagger}(Y_M)$. Now $y^h = y^g$. This proves the claim.
\end{proof}

 By \fref{lem:Cy} every member of $\mathcal Y^*$ is conjugate to an element of $\mathcal Y^*_S$, thus $\mathcal Y^*_S \not= \emptyset$.

\begin{lemma}\label{lem:E to E}
Suppose that $y \in \mathcal Y_S$ and $K$ is a  component  of $E(C_G(y))$. If $w \in C_{C_S(y)}(K)$ is an involution, then $K\le E(C_G(w))$. In particular, if $w \in C_{Y_M}(K)^\#$, then $w \in \mathcal Y$.
\end{lemma}

\begin{proof}  Set $X= C_G(w)$.  Then  $$ K \le E(C_X(y)) \le L_{2'}(C_X(y)) \le L_{2'}(X)$$ by $L_{2'}$-balance \cite[Theorem 5.17]{gls1}. By \fref{lem:char2},  $z$ inverts $O(L_{2'}(X))$ for $z \in \Omega_1(Z(S))^\#$.   Since $z$ centralizes $C_S(y)\in \syl_2(C_G(y))$, $z$ normalizes $K$ and so $[K,z] =K$ by \fref{lem:char2}.  Since $z$ inverts $O(L_{2'}(X))$, we also have $K=[K,z]\le C_{X}(O(L_{2'}(X)))$. Thus $$K \le C_{L_{2'}(X)}(O(L_{2'}(X)))
=E(X)Z(O(L_{2'}(X)))O_2(L_{2'}(X)).$$ by \fref{lem:C Core}. We conclude that $K \le E(X)$, as claimed.
\end{proof}

\begin{lemma}\label{lem:K into E_w}
Suppose that $y \in \mathcal Y_S$, $w \in \mathcal Y$ and $K$ is a component in $E_y$ which is centralized by $w$. Then  either $K$ is a component of $E_w$ or $K \le J_1J_2$ where $J_1$ and $J_2=J_1^y$ are components of $E_w$,  $J_1/Z(J_1) \cong K/Z(K)$ and $|J_1|=|K|$. In particular,  $K \le E_w$.
\end{lemma}

\begin{proof} By \fref{lem:E to E}, $K \le E(C_G(w))$.  Let $J=\langle K^{E(C_G(w))}\rangle$. Then $J$ is a product of components of $C_G(w)$.  By \cite[Theorem 5.24 (ii)]{gls1}, $\langle y \rangle$ acts transitively on the components of $C_G(w)$ in $J$. If $J$ is a component of $E(C_G(w))$, then the maximal selection of $K$ implies that $K= J$ and so $K \le E_w$. So suppose that $J= J_1^yJ_1$.  Then $K\cap J_1$ is centralized by $y$ and so $K \cap J_1 \le  J_1^y\cap J_1$. Thus  $$K/(K \cap J_1) \cong KJ_1/J_1 \le J_1^yJ_1/J_1 \cong J_1^y/(J_1^y\cap J_1).$$  In particular, the maximal choice of $|K/Z(K)|$ implies that $K/Z(K) \cong J_1/Z(J_1)$. Moreover, we calculate $$|K||J_1^y \cap J_1| \le |J_1^y||K \cap J_1|\le |J_1^y||J_1^y\cap J_1|$$ and so from the maximal choices of $|K|$ we deduce that $|K| = |J_1^y|=|J_1|$.  Thus, by definition, $K \le J \le E_w$, and this completes the proof.
\end{proof}

\begin{lemma}\label{lem:comps to max}
Suppose that  $y \in \mathcal Y_S^*$  and $w \in C_{Y_M}(E_y)^\#$.  Then    $E_y = E_w$. In particular, $w \in \mathcal Y^*$.
\end{lemma}

\begin{proof}  By \fref{lem:E to E},  $w \in \mathcal Y$ and then, by \fref{lem:K into E_w},  $E_y \leq E_w$.
The maximal choice of  $|E_y|$  shows $E_y = E_w$.
In particular $w \in \mathcal Y^*$.
\end{proof}

 For $y \in \mathcal Y^*_S$,
define $$S_y= C_S(y) \cap E_y \in \syl_2(E_y); \text{ and }$$ $$T_y= C_{C_S(y)}(E_y).$$
 Observe that \fref{lem:comps to max} implies that $(Y_M \cap T_y)^\# \subseteq \mathcal Y^*$.

\begin{lemma}\label{lem:Ty cap ZS}  If $y \in \mathcal Y_S$ and $F \le E(C_G(y))$ is a component of $C_G(y)$, then  $C_{C_S(y)}(F) \cap Z(Q)=1$. In particular, $Z(Q) \cap T_y =1$.\end{lemma}

\begin{proof}  This follows by \fref{lem:char2}.
\end{proof}

\begin{lemma}\label{lem:Csy=CSK}
 Suppose that $y \in \mathcal Y^*_S$ is chosen with $|C_S(y)|$ maximal. Then $C_S(y) \in \syl_2(N_G(E_y))$. In particular, $C_S(y) = N_S(E_y)$ and $T_y=C_{S}(E_y) $
\end{lemma}

\begin{proof} Plainly $C_S(y) \le N_G(E_y)$.  Assume that $R \in \syl_2(N_G(E_y))$ with $R > C_S(y)$ and pick  $t \in  N_{R}(C_S(y))\setminus C_S(y)$. As $t$ normalizes $C_S(y) \ge O_2(M)$, \fref{lem:weakcl} (iii) and (iv) imply that $t$ normalizes $Y_M$. Hence  $\langle t\rangle C_S(y)$ normalizes $Y_M \cap C_{C_S(y)}(E_y)\ge \langle y \rangle$. Thus there exists $w \in (Y_M \cap C_{C_S(y)}(E_y))^\#$ which is centralized by $\langle t\rangle C_S(y)$. \fref{lem:comps to max} implies $w \in \mathcal Y^*$ and then the maximal choice of $|C_S(y)|$ together with \fref{lem:Cy} provide a contradiction. Therefore $C_S(y) \in \syl_2(N_G(E_y))$ and this proves the main claim.   It follows at once that  $C_S(y) = N_S(E_y)$ and $C_{S}(E_y) = T_y$. 
\end{proof}

 \begin{lemma}\label{lem:NormalTy} Let $y \in \mathcal Y^\ast_S$ with $|C_S(y)|$ maximal. Then $N_S(T_y) = C_S(y)$.
\end{lemma}
\begin{proof} Assume the statement is false and choose $t \in N_S(C_S(y)) \setminus C_S(y)$ with $T_y^t = T_y$. Then $t$ normalizes $U = Z(C_S(y)) \cap T_y \cap Y_M$. As $y \in U$, $U \not= 1$. Hence there is $1 \not= w \in U$ such that $w^t = w$. Since $w \in T_y$, $E_y=E_w$ by \fref{lem:comps to max}.  But then, by \fref{lem:Csy=CSK}, $t \in N_S(E_w)= N _S(E_y)= C_S(y)$, a contradiction.\end{proof}

 Suppose that $W$ is a group.  A subgroup $H $ of  $W$   is called a \emph{trivial intersection subgroup} in $W$  provided that $H$ is not normal in $W$ and, for all $g \in W\setminus N_W(H)$, we have $H \cap H^g=1$. The following lemma will play an important role in the proof of  our theorem.

\begin{lemma}\label{lem:Ty TI}
 Suppose that $y \in \mathcal Y^*_S$ is chosen with $|C_S(y)|$ maximal. Then  $T_y$ is a trivial intersection subgroup in $S$ and $T_y \cap O_2(M)$ is a trivial intersection subgroup in $N_G(O_2(M))$.
\end{lemma}

\begin{proof} By \fref{lem:Ty cap ZS}, $Z(S) \cap T_y = 1$. Hence $T_y$ is not normal in $S$ and also $T_y \cap O_2(M)$ is not normal in $N_G(O_2(M)) \geq S$.
Suppose that $g \in N_G(O_2(M)) $ and assume $T_y \cap T_y^g \ne 1$.
 Since $O_2(M)= O_2(M)^g  $,  $O_2(M) $ normalizes $T_y \cap T_y^g$. Therefore \fref{lem:weakcl} (iv) implies  $Y_M\cap T_y \cap T_y^g \ne 1$. Pick $w \in  (Y_M \cap T_y \cap T_y^g)^\#$. Then,
 by \fref{lem:comps  to max},$$E_y = E_w.$$  As $y^g \in \mathcal Y^*_{S^g}$, and $w \in T_y^g$, we also obtain by \fref{lem:comps to max}  $$E_{y^g} = E_w$$ and therefore $$E_y^g= E_{y^g}= E_w= E_y.$$

 Hence, as $g \in N_G(O_2(M))$, using \fref{lem:Csy=CSK} for the first and last equality yields $$T_y^g\cap O_2(M)=C_{O_2(M)}(E_{y})^g =C_{O_2(M)}(E_{y^g})= C_{O_2(M)}(E_y)=T_y\cap O _2(M).$$ This proves the $T_y\cap O_2(M)$ is a trivial intersection subgroup in $N_G(O_2(M))$. If, in fact, $g \in S\le N_G(O_2(M))$, then, again using \fref{lem:Csy=CSK}, we have
 $$T_y^g =C_{S}(E_{y})^g =C_{S}(E_{y^g})= C_{S}(E_y)=T_y$$ which shows that $T_y$ is a trivial intersection subgroup in $S$.
\end{proof}

\begin{lemma}\label{lem:|Y_M| bound} Suppose that $y \in \mathcal Y^*_S$ is chosen with $|C_S(y)|$ maximal.
Assume that $X\le Y_M$ is normalized by $C_S(y)Q$.
Then $$ |X\cap T_y|^2 \le |X|\le |XT_y/T_y|^2.$$ In particular, these bounds hold for $X=[Q,y]$ and $X=Y_M$.
\end{lemma}

\begin{proof}  As, by \fref{lem:basic1}(ii), $y \not \in Z(Q)$, we can choose $t \in N_{Q}(C_S(y))\setminus C_S(y)$ with $t^2 \in C_S(y)$. If $T_y$ is normalized by $t$, then $T_y\cap Y_M \ge \langle y \rangle$ is normalized by $C_S(y)\langle t\rangle$ and so by \fref{lem:comps to max} there exists $w \in \mathcal Y^*$ with $|C_S(w)| \ge 2|C_S(y)|$, a contradiction. Hence $t \not \in N_S(T_y)$ and so $$(T_y\cap X) \cap (T_y \cap X)^t\le T_y \cap T_y^t =1$$ by \fref{lem:Ty TI}. Thus $|X \cap T_y|^2\le |X|$.
As $$|X \cap T_y|=|(X \cap T_y)^t| =|(X \cap T_y)^t(X \cap T_y)/(X \cap T_y)| \le |X T_y/T_y|,$$ we also obtain
$$|X|=|X T_y/T_y||(X \cap T_y)| \le |X T_y/T_y|^2.$$
Since $Y_M$ and $[Q,y]\le Y_M$ are both normalized by $QC_S(y)$, the displayed bounds apply to these subgroups.
\end{proof}

\begin{lemma}\label{lem:YM normalizes K} Assume that $y \in \mathcal Y^*_S$ and $K$ is a component of $E_y$. Suppose that $N_G(S_yT_y)$ has characteristic $2$.  Then $Y_M \le O_2(N_G(S_yT_y))$. In particular, if $N_{E_y}(S_y) > S_yZ(E_y)$, then $Y_M$ normalizes $K$.
\end{lemma}

\begin{proof} We have that $O_2(M)$ normalizes $T_yS_y$. Hence the first assertion follows from \fref{lem:YM in O2J}.

Let $K$ be a component of $E_y$ and $X= S_y \cap K$.  Then by hypothesis $N_K(X) >  XZ(K)$. Let $w \in  N_K(X) \setminus XZ(K)$ have odd order, then $w \in N_G(T_yS_y)$ and so, for $t \in Y_M$,  $[w,t] \in O_2(N_G(S_yT_y))$. However, if $K \ne K^t$, then $w$ and $w^t$ commutes and so $[w,t]=w^{-1}w^t$ has odd order.  We conclude that $Y_M$ normalizes $K$.
\end{proof}

\begin{definition}\label{def:2} Assume that $W$ is a normal subgroup of a group $X$.  Then $W$ has the \emph{ Sylow centralizer property in $X$} provided that for $T \in \syl_2(X)$ and $R =W\cap T  \in \Syl_2(W)$, $$C_{TC_X(W)/C_X(W)}(RC_X(W)/C_X(W)) \le WC_X(W)/C_X(W).$$
\end{definition}

\begin{lemma}\label{lem:N(SyTy) char 2} Assume that $y \in \mathcal Y^*_S$ and that  every component $K$ of $E_y$  has the  Sylow centralizer property  in $N_{C_G(y)}(K)$.  Then $\Omega_1(Z(S)) \le S_yT_y\in \syl_2(E_y T_y)$ and $N_G(S_yT_y)$ has characteristic $2$.
\end{lemma}

\begin{proof} Since $\Omega_1(Z(S))$ normalizes every component in $E_y$ and they each satisfy the Sylow centralizer property  in $C_G(y)$, we have  $\Omega_1(Z(S)) \le S_y T_y$.
The result now follows  from \fref{lem:fund-charp}.
\end{proof}

\begin{lemma}\label{lem:Tyabel}  Suppose that  $y \in \mathcal Y^*_S$ with $|C_S(y)|$ is maximal  and $E_y =K $ is quasisimple and satisfies the Sylow centralizer property  in $C_G(y)$. Assume that  $C_G(x)$ has  characteristic $2$  for all $x \in \Omega_1(Z(S_y)) \setminus Z(K)$. Then $T_y$ is isomorphic to a subgroup of $Z(S_y)/(S_y\cap Z(K))$.
\end{lemma}

\begin{proof}  \fref{lem:Ty cap ZS} implies that $S > N_S(T_y)$. Let $g \in N_{S}(N_S(T_y)) \setminus N_S(T_y)$ with $g^2 \in N_S(T_y)$. Then $T_yT_y^g \le N_S(T_y)=N_S(T_y^g)$ and, as $T_y \ne T_y^g$,  \fref{lem:Ty TI} implies  $$[T_y,T_y^g] \le T_y \cap T_y^g = 1.$$  In particular, as $y \in T_y$,  $T_y^g \leq C_S(y)$ and so  normalizes $E_y=K$ and thus also $S_y$.
Assume that $T_y^g \cap S_y \ne 1$.  Then, as $S_y$ normalizes $T_y^g$, $T_y^g \cap \Omega_1(Z(S_y)) \ne 1$. By \fref{lem:E to E} the centralizer of every involution in $T_y$ is not of characteristic $2$.  The hypothesis on elements of $\Omega_1(Z(S_y))$ implies that $T_y^g \cap \Omega_1(Z(S_y)) \le Z(K) \cap S_y \le C_{C_S(y)}(K) =T_y$. As $T_y \cap T_y^g=1$, we have a contradiction. Hence $T_y^g \cap S_y=1$. As $T_y^g$ is normalized by $N_S(T_y) \ge S_y$, we  have $[T_y^g ,S_y]\le T_y^g \cap S_y=1$. Thus  the Sylow centralizer property in $C_G(y)$   yields   $$T_y^g \leq T_yZ(S_y).$$ As $T_y \cap T_y^g=1$, we conclude that $T_y$ is  abelian and isomorphic to a subgroup of $Z(S_y)/ (S_y\cap Z(K))= Z(S_y)/(S_y\cap T_y)$.
\end{proof}

Next, for $y \in \mathcal Y$, we study the action of $O_2(M)$ and $Y_M$ on the components of $C_G(y)$.

\begin{lemma}\label{lem:YM normalizes}
Assume that $y \in \mathcal Y$ and $K$ is a component of $E(C_G(y))$.  If $Y_M$ does not normalize $K$, then  $K/Z(K)$ has elementary abelian Sylow $2$-subgroups.
\end{lemma}

\begin{proof} We may assume that $y\in \mathcal Y_S$. Then $Y_M $ is an abelian normal subgroup of $C_S(y)\in \syl_2(C_G(y))$ which does not normalize $K$. Hence \fref{lem:compsnotnormal} provides the result.
 \end{proof}

\begin{lemma}\label{lem:char2A}  Suppose that $y \in \mathcal Y$ and $K$ is a component of $C_G(y)$. If $Z(Q) \cap K\ne 1$, then $F^*(C_G(y))= KO_2(C_G(y))$ and $O(C_G(y))=1$.
\end{lemma}

\begin{proof} Since $Z(Q) \cap K \ne 1$, we can select $z \in (\Omega_1(Z(Q)) \cap K)^\#$. As $z \in K$, $z$ centralizes $O(C_G(y))$ as well as any component $J$ of $C_G(y)$ with $J\ne K$. Applying \fref{lem:char2} proves the claim.
\end{proof}

\begin{lemma}\label{lem:O2M and K}
Suppose that $y \in \mathcal Y_S$ and $K$ is a component of $E(C_G(y))$ which is normalized by $Y_M$. Assume that $S \cap Z(K)=1$. Then either $S\cap K  \le O_2(M)$ or $O_2(M)$ normalizes $K$.
\end{lemma}

\begin{proof} If $S\cap K $ centralizes $Y_M$, then $S\cap K \le C_S(Y_M)= O_2(M)$. Suppose that
$Y_M$ does not centralize $S \cap K$. Then $1\ne [Y_M,S\cap K] \le S\cap K$ and $ [Y_M,S\cap K]$ is centralized by $O_2(M)$. Thus, for $m \in O_2(M)$, $[Y_M,S\cap K]\le K \cap K^m$. If $K \ne K^m$   this yields $[Y_M,S\cap K]\le S \cap Z(K)=1$, a contradiction.  Thus $K$ is normalized by $O_2(M)$.
\end{proof}

\begin{lemma}\label{lem:compnormal}
Suppose that $y \in\mathcal Y_S$ and  $ z \in \Omega_1(Z(S))^\#$.  Assume that $K$ is a component of $C_G(y)$ and $C_{K}(z)$ is not a $2$-group. Then $C_Q(y)$ normalizes $K$. In particular, $Y_M$ normalizes $K$.
\end{lemma}

\begin{proof}  Assume that the lemma is false. As $z$ inverts $O(C_G(y))$ by \fref{lem:char2}, $z$ inverts $O(C_G(y)) \cap Z(K)$ and so, as $C_K(z)$ is not a $2$-group, there is an odd prime $r$ and $R \in\syl_r(C_K(z))$ with $R \not \le Z(K)$. Assume that $C_Q(y)$ does not normalize $K$.
 Then there exists  $b \in C_Q(y)$  such that $K^b\ne K$. Because $K$ is a component of $C_G(y)$,  $[K,K^b]=1$.
 Since $C_G(z)\le N_G(Q)$ and $b \in Q$, we have  $$C_{K}(z)Q= (C_{K}(z)Q)^b= C_{K^b}(z)Q.$$
  In particular, as $C_{K}(z) C_{K^b}(z)\le C_{K}(z)Q$,  $R \in \syl_r(C_{K}(z) C_{K^b}(z)) $ and so $RR^b=R \le K \cap K^b\le Z(K)$, a contradiction. Hence $C_Q(y)$ normalizes $K$. As $Y_M \le C_Q(y)$, $Y_M$ also normalizes $K$.
 \end{proof}

Next we show that in many situations $E(C_G(y))$ is quasisimple.

\begin{lemma}\label{lem:compnormal2}  Suppose that $y \in\mathcal Y_S$, $z \in \Omega_1(Z(S))^\#$ and $K$ is a component of $C_G(y)$.  Assume there is  a non-trivial subgroup $J \le C_K(z)$ such that
\begin{enumerate}
\item [(a)]$J=O^2(J)$ is normalized by $C_Q(y)$; and
 \item [(b)] $[Q,y]$ is centralized by $J$,
 \end{enumerate}
 Then
  \begin{enumerate}
  \item $Q$ normalizes $J$ and $1 \not= Z(Q) \cap [Q,J]   \leq K$; and
  \item $F^*(C_G(y)) = KO_2(C_G(y))$.
   \end{enumerate} In particular, assumption (b) holds if, for all $W\leq Y_M$ with $W$ normalized by $J$,  we have $[W,J] = 1$.
\end{lemma}

\begin{proof}    By (a) and \fref{lem:compnormal}, $C_Q(y)$ normalizes $K$ and, as $z \in Z(Q)$ and $Q$ is large,  $J =O^2(J)\leq N_G(Q)$ and $[Q,J] \ne 1$.  Set $W=[Q, y ]$.  Then (b) implies  $[W,J] = 1$ and, as $[J,y]=1$, we have $$[Q,J,y] = 1$$ by the Three Subgroups Lemma. In particular, as $J= O^2(J)$  and $C_Q(y)$ normalizes $C_K(z)$ by (a),
$$[Q,J]= [Q,J,J] \le [C_Q(y),J] \le  J.$$   Because $[Q,J] \not= 1$ and $[Q,J]$ is normalized by $Q$, we have that $$1 \not= Z(Q) \cap [Q,J]\leq Z(Q) \cap J    \leq K.$$ Thus $Z(Q) \cap K \not= 1$ and so    $F^*(C_G(y)) = KO_2(C_G(y))$ follows from \fref{lem:char2A}. This proves (i) and (ii).

Now suppose for all $W\leq Y_M$ with $W$ normalized by $J$,  we have $[W,J] = 1$. Then as $[Q,y]\le Y_M$ and is normalized by $J$, (b) holds.
\end{proof}

\begin{lemma}\label{lem:compnormal1} Suppose that $y \in\mathcal Y_S$ and $z \in \Omega_1(Z(S))^\#$. Let $K$ be a component of $C_G(y)$ and set $L_K= C_K(z)$.  Assume that $L_K$ is not a $2$-group.  Then  $C_{G}(Y_M)C_Q(y)$ normalizes $K$. Furthermore, if $Y_M \cap K \le C_{K}(O^2(L_K))$, then  $F^*(C_G(y))=K O_2(C_G(y))$.
\end{lemma}

\begin{proof}  By \fref{lem:compnormal}, $Y_M$ normalizes $K$.
Assume that $Y_M \cap K \not\leq Z(K)$. Then, for  $m \in C_G(Y_M)$, $K^m \cap K \geq Y_M \cap K$. Hence  $K^m = K$, and so $C_G(Y_M)$ normalizes $K$. Thus the main assertion holds in this case.

If $Y_M\cap K \le Z(K)$, then   $Y_M \cap K \le C_{K}(O^2(L_K))$. Hence suppose that  $Y_M \cap K \le C_{K}(O^2(L_K))$ and set $L_1=O^2(L_K)$.   If $W\leq Y_M$   is normalized by $L_1$, then, as $W$ normalizes $K$, $$[W,L_1] \leq W \cap K \leq Y_M\cap K \le C_{K}(L_1).$$ Thus $[W,L_1] = [W,L_1,L_1] = 1$.   \fref{lem:compnormal} now provides the hypothesis for   \fref{lem:compnormal2} which  in turn yields  $F^*(C_G(y))=KO_2(C_G(y))$. In particular, $C_G(Y_M)$ normalizes $K$. This completes the proof.
\end{proof}

\begin{lemma}\label{lem:2group} Suppose that $z \in \Omega_1(Z(S))^\#$,  $y \in\mathcal Y_S$ and  $K$ is a component of $C_G(y)$.   Assume that $J \le C_K(Y_M)$, $O_2(M)$ normalizes $J$ and $J$  is not a $2$-group. Set $\wt J= JC_{C_S(y)}(K)/C_{C_S(y)}(K)$.  Then the following hold
\begin{enumerate}
\item  $M^\circ \le N_G(O^2(J))$.
\item There exist distinct non-central $O^2(J)$-chief factors in $$O_2(O^2(J))/\Phi(O_2(O^2(J)))$$ which are isomorphic as $O^2(J)$-modules. In particular, $O_2(J)$ has at least two  non-central $O^2(J)$-chief factors.
\item $F^*(C_G(y))= KO_2(C_G(y))$ and $O(C_G(y))=1$.
\item  $|Z(\wt{O^2(J)})|\ne  2$.

\end{enumerate}
\end{lemma}

\begin{proof}
\fref{lem:weakcl} states $[C_G(Y_M), M^\circ] \le O_2(M)$ and so  $JO_2(M)$ is normalized $M^\circ$.  Hence, as $J$ is normalized by $O_2(M)$,  $O^2(JO_2(M))= O^2(J)$  is normalized by $M^\circ$. This is (i).

  We have that $O^2(J) \not= 1$ by hypothesis. Further, by (i), $Q$ normalizes $O^2(J)$. As $O^2(J)$ normalizes $Q$, we have $[Q,O^2(J)] \leq Q \cap O^2(J)$. As $[Q,O^2(J)] \not= 1$, we have that $Q \cap O^2(J) \not= 1$ and so $O_2(O^2(J)) \not= 1$.

 Assume that (ii) is false. Then the non-central  $O^2(J)$-chief factors in $O_2(O^2(J))/\Phi(O_2(O^2(J)))$ are pairwise non-isomorphic. Since $Q\le M^\circ$ normalizes $O^2(J)$ and  $O^2(J)\le N_G(Q)$,  \fref{lem:non-iso chief factors} shows that $$[O_2(J),O^2(J)]/\Phi([O_2(J),O^2(J)])$$ is centralized by $Q$. Since $M^\circ$ operates on this factor, we conclude  from Burnside's Lemma   that $O_2(O^2(J))$ is centralized by $O^2(M^\circ)$. This contradicts \fref{lem:McircZ(Q)} and  completes the proof of (ii).

We have that  $[Q,O^2(J)]$ is a non-trivial normal subgroup of $Q$ contained in $K$. It follows that $Z(Q) \cap K \ne 1$.
 Part (iii) follows from  \fref{lem:char2A}.

 Suppose that $|Z(\wt{O^2(J)})| =2$. Observe that $Z(K)$ is a 2-group by (iii). Then, as $O^2(J)$ centralizes $Z(K)$, $$|Z(O^2(J))Z(K)/Z(K)|= |Z(O^2(J)):Z(O^2(J))\cap Z(K)|\le 2.$$  By (i), $M^\circ $ normalizes $Z(O^2(J))$. If $Z(K) \cap Z(O^2(J)) =1$, then $M^\circ$ centralizes $Z(O^2(J)) $ and this contradicts   \fref{lem:McircZ(Q)}. So assume that $Z(K)  \cap Z(O_2(J))\ne 1$.

 As $Z(K) \leq T_y$, \fref{lem:Ty cap ZS} implies $Z(K) \cap Z(Q) = 1$. In particular $Z(K) \cap Z(O^2(J))$ is not normalized by $Q$ and so there exists $x \in Q$ such that $$(Z(K) \cap O^2(Z(J)))(Z(K) \cap Z(O^2(J)))^x= Z(O^2(J)).$$ Since $T_y$ is a trivial intersection subgroup in $S$ by  \fref{lem:Ty TI}, we conclude that $Z(O^2(J))$ has order $4$. As $Z(O^2(J))$ is normalized by $Q$, we have that  $Z(O^2(J))$ contains elements in $Z(K)^\#$ and elements in $Z(Q)^\#$. By \fref{lem:E to E},  these elements are not conjugate in $G$, hence  $O^2(M^\circ)$ centralizes $Z(O^2(J))$ and again we have a contradiction to  \fref{lem:McircZ(Q)}.
  This proves (iv).
\end{proof}

  The next lemma will be used when $K$ is a  group of Lie type in characteristic 2 and also for some situations when $K$ is a sporadic simple group. Recall that $U_Q$ is defined by $U_Q= \langle Y_M^{N_G(Q)}\rangle$, $U_Q$ is elementary abelian  and $U_Q \le C_Q(y)$ for all $y \in Y_M$ by \fref{lem:NQclos}.

\begin{lemma}\label{lem:special} Suppose that $z \in \Omega_1(Z(S))^\#$,  $y \in\mathcal Y_S$ and  $K$ is a component of $C_G(y)$.
 Set $L_K= C_K(z)$ and $J_K= C_{O^2(L_K)}(Z(O_2(O^2(L_K))))$.
If  $ O_2(O^2(L_K))$ is non-abelian,  then $O^2(J_K) $ does not act irreducibly on $O_2(O^2(L_K))/Z(O_2(O^2(L_K)))$.
\end{lemma}

\begin{proof} Set $Z=Z(O_2(O^2(L_K)))$.  Then $ J_K $ centralizes $Z$.  By \fref{lem:charpy} (v), $F^*(L_K) = O_2(L_K)$. Suppose that $$O^2(J_K)\text{ acts   irreducibly on }O_2(O^2(L_K))/Z.$$ Since $O_2(O^2(L_K))$ is non-abelian, $O_2(O^2(L_K))/Z$ is not cyclic and so $O^2(J_K)\ne 1$. In particular,  $O^2(L_K)\ne 1$ and so  \fref{lem:compnormal1}  yields
$$C_G(Y_M)C_Q(y) \mbox{ normalizes }K.$$
As $O^2(J_K)  \le L_K\le C_G(z) \leq N_G(Q)$, $O^2(J_K)$ normalizes $U_Q=\langle Y_M^{N_G(Q)}\rangle$.  Using  $U_Q \leq C_Q(y)$ and $C_Q(y)$ normalizes $K$, yields $U_Q$ normalizes $K$. Furthermore, $U_Q$ normalizes $L_K=C_K(z)$  and therefore  also $O^2(J_K)$ and so does $C_Q(y)$.  It follows that $$[U_Q,O^2(J_K)]\le U_Q \cap O^2(J_K)  \le O_2(O^2(J_K)) \le O_2(O^2(L_K)).$$
Since $[U_Q,O^2(J_K)]$ is normalized by $O^2(J_K)$, $O^2(J_K)$ acts irreducibly on $O_2(O^2(L_K))/Z$ and  $O_2(O^2(L_K))$ is non-abelian, but $U_Q$ is abelian, we get that $ [ U_Q,O^2(J_K)]\le Z$.
  Therefore, $$[ U_Q,O^2(J_K), O^2(J_K)] \le [Z,O^2(J_K)]=1.$$
Hence $U_Q$ is centralized by $O^2(J_K)$ and thus
$$  O^2(J_K) \leq C_G(Y_M).$$

 Since $C_G(Y_M)C_Q(y)$ normalizes $K$, $O_2(M)$ normalizes $J_K$. As $O_2(J_K)$ has exactly one non-central $O^2(J_K)$-chief factor, \fref{lem:2group} (ii) provides the final contradiction.
 \end{proof}

\begin{lemma}\label{lem:O2L abelian} Suppose that $z \in \Omega_1(Z(S))^\#$,  $y \in\mathcal Y_S$ and  $K$ is a component of $C_G(y)$. Set $L_K= C_K(z)$. Assume that $L_K$ is not a $2$-group, $O_2(O^2(L_K))$ is elementary abelian and contains exactly one non-central $O^2(L_K)$-chief factor.  If $Q$ normalizes $O^2(L_K)$, then $[O_2(L_K),O^2(L_K)] \le Z(Q)$.
\end{lemma}

\begin{proof} Since $Q$ normalizes $O^2(L_K)$ and $L_K$ normalizes $Q$,  $[Q,O^2(L_K)]\le O_2(O^2(L_K))$. The result follows from \fref{lem:non-iso chief factors}.
\end{proof}

\begin{lemma}\label{lem:project at least 4} Assume that $z \in \Omega_1(Z(S))^\#$, $y \in \mathcal Y$ and $K$ is a component of $C_G(y)$ which is normalized by $Y_M$. Then, setting $\wt {KN_{C_S(y)}(K)}= KN_{C_S(y)}(K)/C_{C_S(y)}(K) $, we have $$|\wt{Y_M}|\ge |\wt{\langle z^M\rangle }| > 2.$$
\end{lemma}

\begin{proof} Let $K$ be a component of $C_G(y)$   and assume that $$|\wt{\langle z^M\rangle } |\le 2.$$
 Then, for all $m \in M$,  $z^m \in Y_M^\#$ and  \fref{lem:char2} implies $z^m$ acts non-trivially on $K$. Therefore
 $$\wt z= \wt{z^m}$$ for all $m \in M$.  Hence $$[z,m]= z^{-1}z^m \in C_{C_S(y)}(K)$$ for all $m \in M$. Therefore $[\langle z\rangle,M] \le C_{C_S(y)}(K)$.  If $[\langle z\rangle, M]  \ne 1$, then $K$ is a component of $N_{C_G(y)}([\langle z\rangle, M]) $ and this contradicts \fref{lem:charpy}(iii). Hence $M \le C_G(z) \le N_G(Q)$ and this is a contradiction to our basic assumption.
\end{proof}

\begin{lemma}\label{lem:YM projects 8} Suppose that $  z \in \Omega_1(Z(S))^\#$  and $y \in\mathcal Y^\ast_S$. Let   $K$ be a component of $C_G(y)$ and $L_K = C_K(z)$. Assume $L_K/Z(L_K)$ is not a $2$-group and set $\wt {KN_{C_S(y)}(K)}= KN_{C_S(y)}(K)/C_{C_S(y)}(K)$.
\begin{enumerate}
\item If $|\wt{Y_M}| \le 4$, then $F^*(C_G(y)) =KO_2(C_G(y))$ and $O^2(L_K)$ is normalized by $Q$.
\item If  $|\Omega_1(C_{\wt{KN_{C_S(y)}(K)}}(O^2(\wt {L_K})))\cap \wt{Y_M}|=2$, then $|\wt{Y_M}| \ge 8$.
    \end{enumerate}
\end{lemma}

\begin{proof} Suppose that $|\wt{Y_M}| \le  4$. Assume that $W \le Y_M$ is normalized by $O^2(L_K)$. If $|\wt{W}|=2$, then $O^2(L_K)$ centralizes $W$ and, if $|\wt{W}|=4$, then, as $O^2(L_K)$ centralizes $\wt{\langle z\rangle}$, again $O^2(L_K)$ centralizes $W$.   \fref{lem:compnormal} implies  \fref{lem:compnormal2}(a) holds. Hence \fref{lem:compnormal2} yields $K= E(C_G(y))$ and   $O^2(L_K)$ is normalized by $Q$. In particular (i) holds.

To prove (ii), assume that $|\wt{Y_M}| \le 4$ and set  $X= C_{Y_M}(O^2(L_K))$. Then, by (i), $K= E(C_G(y))$,  $X$ is normalized by $Q$ and it is also normalized by $C_S(y)$. In particular, $C_{C_S(y)}(K)= T_y$.   Since $X$ is elementary abelian, $|XT_y/T_y|\le  2$ holds  by assumption.  Using \fref{lem:|Y_M| bound}  yields  $|X|\le 4$. As $\wt z \in \wt X$ and $y \in C_{Y_M}(K)\le X$, we deduce that $|X|=4$  and  $|C_{Y_M}(K)|= 2$. Hence    $|Y_M| = |\wt {Y_M}||C_{Y_M}(K)|\le 8$ and this contradicts \fref{lem:YM8}. Therefore  (ii) holds.
\end{proof}

 \begin{lemma}\label{lem:constrK} Suppose that $y \in Y_M^\#$ and $K$ is a component of $C_G(y)$. Let $P$ be a $2$-local subgroup of $K$,  and assume that both $K$ and $P$ are normalized by $O_2(M)$. Set $\wt {KN_{C_S(y)}(K)}= KN_{C_S(y)}(K)/C_{C_S(y)}(K)$. If $P$ is of characteristic $2$, then   $\wt{Y_M} \le O_2(\wt{PO_2(M)})$.
\end{lemma}

\begin{proof} Set $H =PO_2(M)C_{C_S(y)}(K) \le C_G(y)$. Then $O^2(C_H(O_2(H))) \leq P$ and so $O^2(C_H(O_2(H)))\leq C_P(O_2(P)) \leq O_2(P)$, as $P$ has  characteristic $2$. Hence  $H$ has characteristic $2$ and so  \fref{lem:YM in O2J} gives $Y_M \leq O_2(H)$. Therefore  $\wt{Y_M} \le O_2(\wt{H})$.
\end{proof}

\section{The standard setup and consolidation of notation}

Throughout the remainder of this paper \fref{hyp:MainHyp} holds. We pick and fix
$y \in \mathcal Y_S^*$  with $|C_S(y)|$ maximal. We continue the notation  $$S_y = C_S(y) \cap E_y\text{ and } T_y= C_{C_S(y)}(E_y)$$ where $E_y  $ is as defined just before \fref{lem:E to E}.  Recall that $C_S(y)$ is a Sylow 2-subgroup of $C_G(y)$ by the definition of $\mathcal Y_S$ and so $S_y$ is a Sylow 2-subgroup of $E_y$. Furthermore by \fref{lem:Csy=CSK} we have that $C_S(E_y) = C_{C_S(y)}(E_y)$. The subgroup  $K$ represents an arbitrary  component of $E_y$.
We denote by $\wt{\;\;\;}$   the  projection   $$\wt{\;\;\;}:KN_{C_S(y)}(K)\rightarrow KN_{C_S(y)}(K)/C_{C_S(y)}(K).$$ Thus $\wt K =  KC_{C_S(y)}(K)/C_{C_S(y)}(K) \cong K/Z(K)$.  By \fref{lem:YMnotmaxabelian},  $O_2(M)^\prime \not= 1$ and, by \fref{lem:weakcl}(iv), $Y_M = \Omega_1(Z(O_2(M)))$. Hence we will  fix an involution $$z\in \Omega_1(Z(S))\cap Y_M\cap  O_2(M)'\le Z(Q).$$  Since $z$ centralizes $C_S(y)$ and so $S_y$, $z$ normalizes $K$. We know from \fref{lem:char2}  that $K= [K,z]$ and $z$ inverts $O(C_G(y))$.  We set $$L_K= C_K(z).$$
Obviously, $L_K \le C_G(z)\le N_G(Q)$ and so $[Q,y]\le Y_M$ is normalized by $L_K$.  Furthermore, if $L_K$ is not a $2$-group, \fref{lem:compnormal1} implies that $C_G(Y_M)C_Q(y)$ normalizes $K$ and, in particular, $O_2(M)$ normalizes $L_K$.
We will often require the subgroup $$U_Q= \langle Y_M^{N_G(Q)}\rangle$$ which is elementary abelian and contained in $Q \cap O_2(M)$ by \fref{lem:NQclos}.

The next five sections investigate the various possibilities for the isomorphism type of $K/Z(K)$.

\section{Sporadic groups as components}\label{sec:spor}

 The aim of this section is to show that $K/Z(K)$ cannot be a sporadic simple group or  the Tits group ${}^2\F_4(2)^\prime$. We begin with $\Ru$ and ${}^2\F_4(2)'$.

\begin{lemma}\label{lem:tits} $K/Z(K) \not\cong {}^2\F_4(2)^\prime$ or $\Ru$.
\end{lemma}

\begin{proof} We first provide some structural detail about the groups $X^* \cong {}^2\F_4(2)'$, ${}^2\F_4(2)$ and $\Ru$. Suppose that $x$ is a $2$-central involution in $X^*$. Then by \fref{lem:F4twiststruk} and \cite[page 65]{RS} the centralizer $X=C_{X^*}(x)$ has the following normal subgroup structure:
$$1 \leq X_1 < X_2 \leq X_3 < X_4=O_2(X) < X,$$
where $|X_1| = 2$, $X_1=Z(X_4)$, $X_2$ is elementary abelian of order 32, $X_3 = C_X(X_2)$, $X_2=\Omega_1(X_3)$ and $O^2(X)$ acts irreducibly on $X_2/X_1$ and on $X_4/X_3$ each of order $16$. Furthermore,
$$X_3 \cong \begin{cases} X_2& X^*\cong {}^2\F_4(2)'\\4\times 2^4&X^*\cong {}^2\F_4(2)\\\Q_8\times 2^4 & X^* \cong \Ru.\end{cases}$$   Finally, if $X^* \cong {}^2\F_4(2)$ or ${}^2\F_4(2)'$, then  $X/X_4 \cong {}^2\B_2(2)\cong  \Frob(20)$, while, if $X^* \cong \Ru$, then $X/X_4\cong \Sym(5)$.

Recall  the  Sylow centralizer property from  \fref{def:2}. By \cite[Table 5.3r]{gls2},  $\Aut(\Ru) = \Ru$ and so, when $K/Z(K) \cong \Ru$, the Sylow centralizer property  holds for  $K$   in $N_{C_G(y)}(K)$. We read from \cite[Theorem 2.5.12 and Theorem 2.5.15]{gls2} that $\Aut({}^2\F_4(2)^\prime) = {}^2\F_4(2) = \Aut({}^2\F_4(2))$. As presented  above for $X^\ast \cong {}^2\F_4(2)$, we have that $X_1 = Z(X_4) = Z(S)$. Thus  the Sylow centralizer property  also holds when $K/Z(K)  \cong {}^2\F_4(2)^\prime$.

  Suppose $K/Z(K) \cong {}^2\F_4(2)'$ or $\Ru$. Notice that either $Z(K)=1$ or $K \cong 2\udot\Ru$. As $L_K \ge S_y$, $L_K$ projects mod $Z(K)$  onto $X$ as  described above (in the cases $X^* \cong \Ru$ and $X^*\cong {}^2\F_4(2)'$). For $1\le i \le 4$, we define  $B_i\le L_K$ to be the preimage of the subgroup $X_i$.
  Since $L_K$ is not a $2$-group,   \fref{lem:compnormal1}  implies that $K$ is normalized by
$C_G(Y_M)C_Q(y)$.    We have $$\wt {KO_2(M)} \cong \Ru, {}^2\F_4(2)'\text{ or }{}^2\F_4(2)$$
 and, as $C_{\wt{KN_{C_S(y)}(K)}}(\wt{O^2(L_K)}) = \wt{B_1}$ has order $2$,    \fref{lem:YM projects 8} (ii) implies that
$$ |\wt {Y_M} |\ge 8.$$

Suppose that $W \le Y_M$ is normalized but not centralized by $O^2(L_K)$. Then $\wt {B_2}\le \wt W\le \wt {Y_M}$ and so $|\wt {Y_M} | \ge 2^5$.  Hence, as $X_2=\Omega_1(C_X(X_2))$,   $$  \wt {Y_M}= \wt B_2 .$$
Now we have $$\Omega_1(C_{\wt{O_2(M)(S\cap K)}}(\wt{Y_M})) = \Omega_1(\wt B_3) = \wt B_2 = \wt {Y_M}.$$
It follows that $\Omega_1(O_2(M))C_{C_S(y)}(K)= Y_MC_{C_S(y)}(K)$, which means that $[\Omega_1(O_2(M)),O_2(M)] \le C_{C_S(y)}(K)$. Since $K$ does not centralize any element of $Z(Q)$ by \fref{lem:Ty cap ZS}, we have $\Omega_1(O_2(M))= \Omega_1(Z(O_2(M)))= Y_M$ and this contradicts \fref{lem:YMnotmaxabelian}.
 Therefore \begin{center}$O^2(L_K)$   centralizes every subgroup of $Y_M$ which it normalizes.\end{center}
 By   \fref{lem:compnormal2} $$  F^*(C_G(y))=KO_2(C_G(y)) \text{ and }Q \text{ normalizes } O^2(L_K).$$
Select $g \in N_Q(C_S(y))\setminus C_S(y)$. We have $T_yT_y^g $ is centralized by $K \cap K^g$ and $K \cap K^g \ge O^2(L_K)$  as $Q$ normalizes $O^2(L_K)$. Hence $$\wt{T_y^g} \le C_{\wt {KC_S(y)}}(O^2(\wt{L_K}))= \wt{B_1}$$ which has order $2$.  As $T_y^g \cap T_y=1$ by \fref{lem:Ty TI} and \fref{lem:NormalTy}, we conclude that $|T_y|=2$.

Suppose that  $K/Z(K) \cong \Ru$.  Then by \cite[Table 5.3r]{gls2} there is a $2$-local subgroup $J$ of $K$ containing $S_y$ with  $$J/Z(K) \sim 2^{3+8}.\SL_3(2)$$ and $J$ is normalized by $O_2(M)$. Hence \fref{lem:YM in O2J} implies that $Y_M \le O_2(JY_M)$  and  $\langle Y_M^J \rangle  $ is an elementary abelian.  Now the structure of $J$ and the fact that $O_2(J/Z(K))$ is non-abelian implies that $ |\wt{ \langle Y_M^J \rangle} |=2^3$. Hence $\wt{ \langle Y_M^J \rangle} = \wt{Y_M}$ and $\wt {O_2(M)} \le O_2(\wt J)$. Thus  $O_2(M) \le O_2(O_2(M)J)$ and \fref{lem:weakcl} implies $J$ normalizes $O_2(M)$ and so also $Y_M$. \fref{lem:charO2M} yields $J \le M^\dagger$  and  $J$ induces $\SL_3(2)$ on $Y_M/\langle y \rangle$. As $M^\dagger$ does not act transitively on $Y_M^\#$ and $O_2(M^\dagger/C_M) = 1$,   the subgroup structure of $\SL_4(2)$ yields  $M^\dagger= JC_M$. But then $C_{Y_M}(M^\dagger) = \langle y\rangle$, a contradiction as $y \not \in Z(S)$. Hence $K/Z(K) \not \cong \Ru$.

Suppose that $K \cong {}^2\F_4(2)^\prime$. As $T_y=\langle y \rangle$, $Y_MK= T_yK= \langle y\rangle \times K$ and so $Y_M\cap K$ has index $2$ in $Y_M$. Since $F^*(C_G(y))= KO_2(C_G(y))= K\langle y\rangle$, we have $C_G(y)= KC_S(y)$.  Because $L_K$ normalizes $Q$, $L_K$ normalizes $U_Q=\langle Y_M^{N_G(Q)}\rangle$ which is elementary abelian. Again we have $U_QK= \langle y\rangle \times K$ and so $U_Q \cap K$ is an elementary abelian subgroup of $K$ normalized by $L_K$. We deduce that $K \cap U_Q= B_2 \ge Y_M \cap K$. Since $Y_M \cap K = C_{B_2}(O_2(M))$,   $O_2(M)L_K/O_2(O_2(M)L_K)\cong \Frob(20)$ and $B_2/Z(O_2(L_K))$ is an irreducible $4$-dimensional $L_K$-module, either  $O_2(M) \le O_2(O_2(M)L_K)$ or $|Y_M\cap K | =8$.  In the former case $L_K \le M^\dagger$ as $O_2(M)$ is weakly closed in $S$. But then $L_K$ normalizes $Y_M$ and this contradicts $O^2(L_K)$ centralizing every subgroup of $Y_M$ that it normalizes. Hence
$$|Y_M\cap K | = 8\text{ and }|Y_M|=16.$$

By \fref{lem:F4twiststruk} (iv), $|Z_2(S_y)|=4$ and so  $Z_2(S_y) \le Y_M$. In particular, $P_1=N_K(Z_2(S_y))$ normalizes $C_{O_2(M)K}(Z_2(S_y)) \ge O_2(M)$ and so $O_2(M)$ and   $Y_M$ are normalized by $P_1$.

Since $O_2(M)'\ne 1$, we have $O_2(M)'\cap Y_M \ne 1$. As $O_2(M)' \le K$, we have $Y_M\cap O_2(M)'$ is either $Y_M\cap K$ or $Z_2(S_y)$.  If $Y_M\cap O_2(M)'= Z_2(S_y)$, then $MC_M/C_M$ embeds into the stabilizer of a $2$-space in $\SL_4(2)$, which is isomorphic to $2^4.(\Sym(3) \times \Sym(3))$. Since $P_1/C_{P_1}(Y_M\cap K) \cong \Sym(4)$, this means that $O_2(MC_M/C_M)\ne 1$, a contradiction.  Therefore $M$ normalizes $Y_M\cap K = Y_M \cap O_2(M)^\prime$ and, as $O_2(MC_M/C_M)=1$ and $MC_M/C_M$ is isomorphic to a subgroup  $\SL_3(2)$, again using $P_1/C_{P_1}(Y_M\cap K) \cong \Sym(4)$ yields $ MC_M/C_M \cong \SL_3(2)$. Now $y^M$ has size $1$, $7$ or $8$. In the first two cases $y$ is centralized by a conjugate of $S$, a contradiction. In the latter case, $y$ is centralized by an element of order $7$ in $M$ and this contradicts the fact that $|K|$  is  coprime to $7$.  Hence $K \not \cong {}^2\F_4(2)'$.
 \end{proof}

\begin{proposition}\label{prop:sporadic} $K/Z(K)$ is not a sporadic simple group.
\end{proposition}

\begin{proof}   We   use the information from \cite[Table 5.3]{gls2}  to see that $K$ satisfies the Sylow centralizer property \fref{def:2} in $N_{C_G(y)}(K)$.
 Hence $z$ induces a $2$-central involution on $K$.
 By \fref{lem:charpy} (v), $F^*(L_K)$ is a $2$-group and, in particular,  $L_K$ does not have a component.
It follows that $K/Z(K)$ is not  $\J_1$, $\Co_3$, $\McL$, $\Ly$, $\ON$ or  $\M(23)$. By \fref{lem:special}, if $O_2(O^2(L_K/Z(K)))$ has derived group and Frattini subgroup of order $2$, then $O^2(L_K)$ does not act irreducibly on $O_2(O^2(L_K))/Z(O_2(O^2(L_K)))$. Using \cite[Table 5.3]{gls2} shows that $K/Z(K) \not\cong \Mat(11)$, $\J_2$, $\J_3$, $\J_4$, $\Co_1$, $\Co_2$, $\Suz$, $\M(22)$, $\M(24)^\prime$, $\F_1$, $\F_2$, $\F_3$, or $\F_5$. Because of \fref{lem:tits} the groups which remain to be considered are
$$K/Z(K) \cong \Mat(12), \Mat(22), \Mat(23), \Mat(24), \HS, \text{ and } \He.$$

Using  \cite[Table 5.3]{gls2}    we observe that $L_K$ is not a $2$-group. In particular, $C_G(Y_M)C_Q(y)$ normalizes  $K$ by \fref{lem:compnormal1}.

\begin{claim}\label{claim:6.21} Either $|\wt{Y_M}|\ge 8$ or $K/Z(K) \cong \Mat(22)$ and $\wt {Y_M} \not \le \wt {K}$.\\\end{claim}

\medskip

Using  \cite[Table 5.3]{gls2} we see that $C_{\wt {KN_{C_S(y)}(K)}}(\wt {O^2(L_K)})$ has order $2$ unless $K/Z(K) \cong \Mat(22)$ in which case  it has order $4$ and is not contained in $\wt K$. Hence \fref{lem:YM projects 8} gives the result.\qedc

Suppose that $K/Z(K) \cong \HS$. Then $L_K/Z(K)$ has shape $(2^{1+4}_+\circ 4). \Sym(5)$.
 As $L_K \leq N_G(Q)$ and $\langle Y_M^{L_K}\rangle \le U_Q \cap O_2(L_K)Y_M$ is elementary abelian, we obtain from \cite[Table 5.3m]{gls2} that $Y_M$ projects into $\Omega_1(Z(L_K/Z(K)))$. Thus $|Y_M/C_{Y_M}(K)|=2$,  contrary to \ref{claim:6.21}

Assume that $K/Z(K)$ is one of $\Mat(22)$, $\Mat(23)$, $\Mat(24)$ or $\He$.
Let  $J\in \mathcal L_K(S\cap K)$  be normalized by $O_2(M)$.  Then $Y_M \le O_2(JY_M)$ and $\langle Y_M^J\rangle$ is elementary abelian by \fref{lem:YM in O2J}.
 Hence

\begin{claim}\label{claim:6.22}$Y_M \le \bigcap_{ {J\in \mathcal L_K(S\cap K);}\atop{ O_2(M)\le N_G(J)}}O_2(JY_M).$\\\end{claim}

Assume that $K \cong \He$ or $\Mat(24)$. Then $S \cap K$ is isomorphic to a Sylow $2$-subgroup of $\SL_5(2)$. Hence $S\cap K$ has exactly  two elementary abelian subgroups $E_1$, $E_2$ of order $64$ and they intersect in a group of order $2^4$. Also note that $O_2(L_K)$ is the unique extraspecial subgroup of order $2^7$ in $S\cap K$.  For $i=1,2$, set $J_i= N_K(E_i)$.  If $J_1$ and  $J_2$ are conjugate in $KO_2(M)$, then $\wt{KO_2(M)}\cong \Aut(\He) $ and $L_KO_2(M)$ acts irreducibly on $O_2(L_K)/Z(O_2(L_K))$.  Hence $O_2(L_K)\le  \langle Y_M^{L_KO_2(M)}\rangle \le U_Q$ which is a contradiction as $U_Q$ is abelian. Therefore $O_2(M)$ normalizes $J_1$ and $J_2$, and, as $J_1$ and $J_2$  have characteristic $2$, we get by \ref{claim:6.22}  $$\wt{Y_M}\le \wt{E_1} \cap \wt{E_2} \cap \wt{O_2(L_K)}= \wt{Z_2(S\cap K)}.$$
Since $|\langle Z_2(S\cap K)\rangle|=8$,  \ref{claim:6.21} gives $\wt {Y_M}= \wt{Z_2(S\cap K)}.$  However, $\langle \wt{Z_2(S\cap K)} ^{\wt L_K}\rangle = \wt{O_2(L_K)}$ which is not abelian whereas $\langle \wt {Y_M} ^{\wt L_K}\rangle \le \wt U_Q$ which is abelian. As this is impossible, we conclude  $K/Z(K) \not \cong \Mat(24)$ or $\He$.

Assume next that $K/Z(K) \cong \Mat(22)$ or $\Mat(23)$. Then from \cite[Table 5.3c and 5.3d]{gls2}, $(S\cap K)/Z(K)$ has  two elementary abelian subgroups $E_1/Z(K)$, $E_2/Z(K)$  of order 16 with  normalizers in $K$ that  are of characteristic $2$, where $N_K(E_1/Z(K)) \cong 2^4 .\Sym(5)$ and $N_{K}(E_2/Z(K))\cong 2^4. \Alt(6)$, $2^4. \Alt(7)$, respectively. Furthermore, they are normalized by $O_2(M)$. We have $\wt{O_2(N_{KO_2(M)}(E_2))} \le C_{\wt{KO_2(M)}}(\wt{E_2})=\wt{E_2} \le \wt K $ and thus by \ref{claim:6.22}
$$\wt{Y_M}\le \wt{E_1}\cap \wt{E_2} \le \wt K .$$
Since $(E_1 \cap E_2)/Z(K)$ has order $4$, we have a contradiction to \ref{claim:6.21} in this case as well.
Hence $K/Z(K) \not \cong \Mat(22)$ or $\Mat(23)$.

Assume    that $K/Z(K) \cong \Mat(12)$. In this case $$L_K/Z(K) \sim 2^{1+4}_+.\Sym(3)$$ and by \cite[Table 5.3 b, notes 2]{gls2} an element  $\tau$ of order $3$ in $L_K$  acts fixed point freely on  $O_2(L_K/Z(K))/Z(O_2(L_K/Z(K)))$.

Set $U_1=\langle Y_M^{L_K}\rangle \leq U_Q$. Then $U_1$ is elementary abelian. If some involution $u$ of $U_1$ induces an outer automorphism of $K$, then so does some involution of $C_{U_1}(\tau)$; however, $\tau$ is in the $K$-conjugacy class $3A$   whereas the elements of order $3$ in $C_K(u)$ are in the class $3B$ (see \cite[Table 5.3 b, notes 3]{gls2}). Therefore $\wt{U_1}\le \wt K$. The action of $\tau$ now shows that $|\wt {U_1}| =8$. Hence by \ref{claim:6.21}  $$\wt{Y_M}= \wt{U_1}.$$
We have $$\wt{O_2(M)} \le C_{\wt{KO_2(M)}}(\wt{U_1})$$
which has order at most $2^4$, as $m_2(\Aut(\Mat(12))) \le 4$ by \cite[Table 5.6.1]{gls2}. But then $\wt {O_{2}(M)'} = 1$ whereas we know it contains $\wt z$, a contradiction. Hence $K/Z(K) \not \cong \Mat(12)$.
\end{proof}

\section{Groups of Lie type in odd characteristic as components}\label{sec:LieOdd}

The aim of this section  is to show that if $K/Z(K)$   is a group of Lie type defined in odd characteristic, then $K \cong {}^2\G_2(3)' \cong \SL_2(8)$.

\begin{lemma}\label{lem:L2q}
The following statements hold.
\begin{enumerate}
\item $K/Z(K) \not  \cong  \PSL_2(p)$ with $p\ge 7$ a Fermat or Mersenne prime.
\item If  $K/Z(K) \cong \PSL_2(9)$, then  $|Z(K)|$ is odd and  $$\wt {KN_{C_{S}(y)}(K)}  \cong \Sym(6)\mbox{ or }\Aut(\PSL_2(9)).$$
\item If $K/Z(K) \cong \PSL_2(5)$, then $Z(K)=1$, $Y_M \le KC_{C_S(y)}(K)$ and $\wt {Y_M} = \wt {S\cap K}$.
\end{enumerate}
\end{lemma}

\begin{proof} Suppose that $K$ is one of the groups itemised in the lemma with $\wt{KN_{C_{S}(y)}(K)} \not \cong \Sym(6)$ or $\Aut(\PSL_2(9)).$  Thus, if $K \cong \PSL_2(p)$, then $\wt {KN_{C_S(y)}(K)} \cong \PSL_2(p)$ or $\PGL_2(p)$ and, if $K \cong \PSL_2(9)$, we have $$\wt {KN_{C_{S}(y)}(K)}  \cong X \in\{\Alt(6), \PGL_2(9), \Mat(10)\}.$$Assume that $z$ induces an outer automorphism on $K$. Then, as $\Mat(10)$ has semidihedral Sylow $2$-subgroups, we have $K\langle z \rangle /Z(K) \cong \PGL_2(p)$ or $\PGL_2(9)$ and, in particular, the Sylow $2$-subgroups of $K\langle z \rangle /Z(K)$ are dihedral groups of order at least $8$. Since  $zZ(K)$  centralizes $(S\cap K)Z(K)/Z(K)$, this is impossible. Hence $z$ induces an inner automorphism on $K$. In particular,  \fref{lem:char2} yields  $Z(K)$ is a $2$-group.

 Assume that $ Z(K) \ne 1$.  Then $S\cap K$ is a quaternion group.  Since $K=[K,z]$, we have $z= ws$ for some $w \in  C_{\langle z \rangle K}(K)$ and $s \in (S\cap K)\setminus Z(K)$.  As $[S\cap K,z]=1$, we have $[S \cap K, s]=1$, a contradiction. Hence $$Z(K)=1.$$

We first prove parts (i) and (ii).
By \fref{lem:YM normalizes} as $S \cap K$ is not elementary abelian, $Y_M$ normalizes all the components   of $E_y$ and by \fref{lem:project at least 4} we have
 $$|\wt{Y_M}|>2.$$
If $\wt K \not \cong \PSL_2(7)$ or $\PSL_2(9)$, then, as $\wt {Y_M} $ is normalized by $\wt{N_{C_S(y)}(K)}$ the structure of the Sylow $2$-subgroup of $\wt K$ shows that the only normal elementary abelian $2$-subgroup has order $2$ and so $|\wt{Y_M}|\le 2$, which  is not the case.

  Hence  $\wt K   \cong \PSL_2(7)$ or $\PSL_2(9)$,  $\wt{S\cap K}\cong \Dih(8)$ and $|\wt {Y_M}|= 4$. Thus $[S\cap K, Y_M]= Z(S\cap K)$ and so $O_2(M)$ normalizes $K$. Since $\wt{O_2(M)}$ centralizes $\wt{Y_M}$, $\wt{O_2(M)}=\wt{Y_M}$. But as $z \in O_2(M)^\prime$, we get $z \in C_S(K)$ which is a contradiction to \fref{lem:charpy}. Thus (i) holds and to complete the proof of (ii) we just have to establish that, if
  $\wt {KN_{C_{S}(y)}(K)}  \cong \Sym(6)$ or $\Aut(\PSL_2(9))$, then $|Z(K)|$ is odd. Since $z$ centralizes $S\cap K$, we have that $K\langle z \rangle/C_{K\langle z \rangle}(K) \cong \PSL_2(9)$ or $\Sym(6)$.
That $|Z(K)|$ is odd follows from these observations and  \cite[Proposition 5.2.8 (b)]{gls2}.

For the proof of (iii), we have already shown that $Z(K)=1$ and so $K \cong \PSL_2(5)$. Thus $\wt{N_{C_G(y)}(S \cap K)} \cong \Alt(4)$ or $\Sym(4)$. Lemmas \ref{lem:YM normalizes K} and \ref{lem:N(SyTy) char 2}  imply that $Y_M$ normalizes $K$ and $Y_M \leq O_2(N_G(S_yT_y))$. Hence (iii) holds. It follows from  \fref{lem:project at least 4} that $\wt{Y_M} = \wt{S \cap K}$.
\end{proof}

\begin{lemma}\label{lem:L24done}   We have  $K/Z(K) \not\cong \PSL_2(5)$.
\end{lemma}

\begin{proof} Assume $K/Z(K) \cong \PSL_2(5)$. By \fref{lem:L2q} we have that $K \cong \PSL_2(5)$. Furthermore $\wt {Y_M} = \wt {S\cap K}$ and this is true for all components $K$ of $E_y$.  In particular, $[S_y,Y_M]\le C_K(E_y) \cap E_y=1$ and so  $$S_y \leq C_S(Y_M)=O_2(M).$$
Set $F_y = E_yO_2(M)$. Then $O_2(M)$ is a Sylow 2-subgroup of $F_y$. As $O(F_y) = 1$ we have by  \fref{prop:JS normalizes K}  that $J(O_2(M))$ normalizes every  component of $E_y$.  Since $J(O_2(M))$ centralizes $Y_M$, for any fixed component $K$ we have $$[\wt{S\cap K},\wt{ J(O_2(M))}] =[\wt{Y_M},\wt{ J(O_2(M))}]=1$$ and so $\wt{ J(O_2(M))}=\wt{Y_M}$.  Therefore $\Phi(J(O_2(M))) \leq C_{C_{S}(y)}(K)$. \fref{lem:charpy} implies  $J(O_2(M))$ is elementary abelian. Therefore
 $$J(O_2(M)) = S_y \times J(O_2(M) \cap T_y)$$ and so $N_{E_y}(S_y) \le N_G(J(O_2(M)))\leq M^\dagger$. Now the action of $N_{E_y}(S_y)$ on $S_y$ yields  $Y_M \cap E_y = S_y$ and $O_2(M) = C_{O_2(M)}(K) \times S_y$. In particular,  $\wt{O_2(M)}$ is abelian. Then $z \in O_2(M)^\prime$ is contained in $C_S(K)$, which contradicts \fref{lem:charpy}. Hence $K/Z(K) \not \cong \PSL_2(5)$.
\end{proof}

 \begin{lemma}\label{lem:notA6} We cannot have  $K/Z(K) \cong \PSL_2(9)$.
  \end{lemma}

\begin{proof} Assume $K/Z(K) \cong \PSL_2(9)$. By \fref{lem:YM normalizes},  $K$ is normalized by $Y_M$ and, by \fref{lem:L2q},    $\wt {KN_{C_{S}(y)}(K)} \cong \Sym(6)$ or $\Aut(K)$ with
$$|Z(K)| \text{ is odd.}$$
Furthermore, by \fref{lem:project at least 4} we have $|\wt {Y_M}|\ge 4$.

Assume that  $ [\wt {S\cap K}, \wt{Y_M}]\ne 1 $. Then  $K \ge [S\cap K,Y_M] \not= 1$ and so  $O_2(M)$ normalizes $K$ by \fref{lem:O2M and K}.  Since $z \in O_2(M)'$ and $\wt{O_2(M)'} \le \wt K$, we have $\wt z \in \wt K$. Now \fref{lem:project at least 4} implies that $\wt{O_2(M)'} \cap\wt{ Y_M}$ has order $4$. But then, as $\wt { O_2(M)}$ centralizes $\wt {Y_M}$, we have  $\wt {O_2(M)}$ is abelian. As $z \in O_2(M)^\prime$, we then get that $z \in C_S(K)$, contradicting \fref{lem:charpy}.
Hence $$[\wt {S\cap K}, \wt{Y_M}]= 1 .$$

As $|\wt {Y_M}| \ge 4$ by \fref{lem:project at least 4} and $[\wt {S\cap K}, \wt{Y_M}]=1$,  we have $|\wt {Y_M}|=4 $ and $\wt {Y_M}$ maps to the centre of a Sylow $2$-subgroup of $\Sym(6)$.
 In particular,   $S \cap KY_M$ is contained in $O_2(M)$. This applies to every component of $E_y$.
 Especially
\begin{equation}\label{eq:8.30}S_y \le O_2(M).\tag{1}\end{equation}
   If $z$ does not induce an inner automorphism on $K$, then $O^2(L_K) \cong \Alt(4)$. By \fref{lem:compnormal1} we have that $O_2(M)$ normalizes $K$, which contradicts $z \in O_2(M)^\prime$. Thus $z$ induces an inner automorphism and so by \fref{lem:char2}  $O(K) = 1$. Now by \cite[Remark following Proposition 8.5]{gls1} the assumptions of
\fref{prop:JS normalizes K} are satisfied, which yields  that $J(O_2(M))$ normalizes every component of $E_y$. Hence $\wt {J(O_2(M))} \le J(\wt{N_{C_S(y)}(K)}) \cong \Dih(8) \times 2$. Thus
\begin{equation}\label{eq:8.31}|\Phi(\wt{J(O_2(M))}) |\le 2.\tag{2}\end{equation}

 Let $A$ be a maximal elementary abelian subgroup of $O_2(M)$. Then $A$ normalizes $K$ and $$m_2(A)= m_2(C_{AK}(K) )+ m_2(AK/C_{AK}(K))$$. Combining this with  \fref{eq:8.30} we conclude that $J(AK) = A(S\cap K)$.  In particular, $J(O_2(M))$ is not abelian.

 As  $\Phi(J(O_2(M))) \ne 1$,  we may select $z_*\in C_{Y_M\cap \Phi(J(O_2(M)))}(S)^\#$, and obtain $$\wt {\langle {z_*}^M\rangle}\le\Phi(\wt{J(O_2(M))})$$ contrary to \fref{eq:8.31} and  \fref{lem:project at least 4}. Hence $K/Z(K) \not \cong \PSL_2(9)$.
\end{proof}

\begin{lemma}\label{lem:notL2 or 2G2}  We cannot have  $K/Z(K) \cong \PSL_2(p^a)$ with $p$ an odd prime.

\end{lemma}

\begin{proof}
Suppose that  $K/Z(K) \cong \PSL_2(p^a)$. By Lemmas \ref{lem:L2q}, \ref{lem:L24done} and \ref{lem:notA6}, $p^a $ is not a Mersenne or Fermat prime and $p^a \ne 9$. If $z$ induces an inner automorphism on $K$, then $L_K$ has a normal 2-complement.  Application of \fref{lem:charpy} (v) yields that $L_K$ is a 2-group. Now \cite[Hauptsatz 8.27]{Hu} implies that $p^a$ is a Fermat or Mersenne prime or $p^a = 9$, a contradiction.

  Hence $z$ induces an outer automorphism on $K$. If $z$ induces an inner-diagonal automorphism, then $\langle z \rangle K/Z(K)$ has non-abelian dihedral Sylow $2$-subgroups. Since $z$ induces an outer automorphism which centralizes $S \cap K$ this is impossible.

 Hence $z$ is in the coset of the field automorphism (mod $\PGL_2(p^a)$) and hence is a field automorphism by \cite[Proposition 4.9.1]{gls2}. Thus, as $p^a \ne 9$,  $F^*(L_KZ(K)/Z(K)) \cong \PSL_2(p^{a/2})$ and this contradicts \fref{lem:charpy}(i). Hence $K/Z(K) \not\cong \PSL_2(p^a)$.
  \end{proof}

\begin{proposition}\label{prop:Lieodd}  If $K/Z(K)$ is a group of Lie type in odd characteristic, then $K/Z(K) \cong {}^2\G_2(3)^\prime \cong \PSL_2(8)$.
\end{proposition}

\begin{proof} By \fref{lem:notL2 or 2G2}  we may assume that $K/Z(K) \not\cong \PSL_2(p^a)$ and we also suppose that $K/Z(K) \not \cong {}^2\G_2(3)^\prime$.  We know that $F^*(L_K) $ is a $2$-group by \fref{lem:charpy} (v). Using \fref{lem:Lie odd invs} yields  $K/Z(K)$ is one of the following groups.
$$\PSL_3(3), \PSU_3(3), \PSL_4(3), \PSU_4(3), \PSp_4(3),\POmega_7(3), \POmega^+_8(3),   \G_2(3).$$
Furthermore, in each case the conjugacy class of $\wt z$ is uniquely determined and is contained in $\wt{K}$. Using \cite[Table 4.5.1]{gls2} with \fref{lem:Lie odd invs} we have
$$\wt{O_2(L_K)}= \begin{cases}\Q_8&K/Z(K)\cong \PSL_3(3)\\\Q_8\circ 4&K/Z(K)\cong \PSU_3(3)\\2^{1+4}_+&K/Z(K)\cong \PSL_4(3), \PSU_4(3), \PSp_4(3), \G_2(3)\\
2^{1+4}_+\times 2^2 &K/Z(K)\cong \POmega_7(3)\\2^{1+8}_+&K/Z(K) \cong \POmega_8^+(3).\end{cases}$$
Moreover, other than for $K/Z(K) \cong \POmega_7(3)$, $\wt {L_K}$ does not normalize any elementary
abelian subgroup of $\wt{O_2(L_K)} $ of order greater that $2$.

Suppose that $K \not \cong \POmega_7(3)$. Then $$\wt U_Q \cap \wt{O_2( L_K)} = \Omega_1(Z(\wt{O_2(L_K)}))= \Omega_1(Z(\wt {L_K}))$$ which has order $2$. Hence $O^2(L_K)$ centralizes $U_Q$ and so also $Y_M$.  Applying \fref{lem:2group} (iv) provides a contradiction.

Therefore $\wt K \cong \POmega_7(3)$. Then $O^2(\wt {L_K}) \cong  \Alt(4) \times (\SL_2(3)\circ \SL_2(3))$. Set $J = O^2(C_{O^2(\wt{L_K})}(Z(O_2(O^2(\wt{L_K})))))$. Then $J \cong \SL_2(3)\circ \SL_2(3)$ and $O^2(J)=J$ centralizes every abelian subgroup of $O_2(O^2(\wt{L_K}))$ which it normalizes. In particular, $J$ centralizes $U_Q\ge Y_M$. Thus  \fref{lem:2group} (iv) provides a contradiction. This completes the proof of the proposition.
\end{proof}

The group ${}^2\G_2(3)^\prime$ will be handled as $\PSL_2(8)$ in \fref{sec:LieChar2}.

 \section{Alternating groups as  components}\label{sec:Alt}

In this section we will show that $K/Z(K)$ is not an alternating group $\Alt(n)$, $n \ge 5$. The cases $n = 5,6$ have been discussed in \fref{lem:L24done} and \fref{lem:notA6}. Thus we may assume that $n \geq 7$. Therefore $\wt{KN_{C_S(y)}(K)} $ is isomorphic to either $\Alt(n)$ or $\Sym(n)$.

\begin{lemma}\label{lem:O2normal} We have   $C_G(Y_M)C_Q(y)$ normalizes $K$.
\end{lemma}

\begin{proof} We consider $X \cong \Sym(n)$. Then, as $n \ge 7$, every involution in $X$ either centralizes an element of cycle shape $3$ or $3^2$. Hence $L_K$ is not a $2$-group. \fref{lem:compnormal1} gives the result.
\end{proof}

Because  $O_2(M)$ normalizes $K$ by \fref{lem:O2normal} and $\wt z \in O_2(M)'$,  $\wt{K\langle z\rangle}=\wt{K}$ is isomorphic to  $\Alt(n)$. Under this isomorphism, we get $\wt z$ is even and we let $\mathrm {supp}(z)$ be the set of elements of $\{1,\dots,n\}$ moved by the image of $z$. For a subgroup $H$ of $\Sym(n)$, we use $H^e$ to denote the subgroup of even elements of $H$.
We set notation so that $|\mathrm{supp}(z)| = 2m$.

\begin{lemma}\label{lem:ZK=1} We have $n > 7$ and $Z(K)=1$.
\end{lemma}

\begin{proof}    If $n = 7$, then as $\wt z$ is even, we get $m=2$ and then $O_3(C_K(z)) \not=~1$, which contradicts \fref{lem:charpy}. Thus $n > 7$.

We have  $K=[K,z]$ by \fref{lem:char2} and so $z$ induces a non-trivial automorphism of $K$ of order $2$ and $z$ centralizes $S\cap K \in \syl_2(K)$.  Application of \cite[Proposition 5.2.8 (b)]{gls2} implies that $Z(K)=1$.
\end{proof}

\begin{lemma}\label{lem:alt1}  We have $n-2m \leq 2$ and, if $2m = n-2$, then $n \equiv 2 \pmod 4$.
Furthermore    either $O_2(L_K)/Z(K)$ is elementary abelian and involves exactly one non-trivial irreducible $O^2(L_K)$-module  or $n \in \{8,9,10\}$ and $|\mathrm{supp}(z)| = 8$.
\end{lemma}

\begin{proof} By \fref{lem:ZK=1} $Z(K) = 1$, so $K \cong \Alt(n)$. If $2m\leq n-4$, then $O^2(F^*(L_K))$ contains $\Alt(n-2m)$ and this contradicts  \fref{lem:charpy} (v), other than if $2m = n-4$.

Suppose that  $2m = n-4$.  We may assume that $z = (12)(34) \ldots (n-5, n-4)$. Then $L_K \cong (2 \wr \Sym(m)\times \Sym(4)  )^e$, which contains a Sylow $2$-subgroup of $K$ only if $n \equiv 4 \pmod 8$. By \fref{lem:O2normal} we have that $O_2(M)$ normalizes $K$. We consider $H, J \in \mathcal{L}_K (S \cap K)$ with $H$ stabilizing the partition $\{\{1,2\}, \ldots \{n-1,n\}\}$ and $J$ stabilizing $\{\{1,2,3,4\}, \ldots , \{n-3,n-2,n-1,n\}\}$. Then $H$ and $J$ are normalized by $O_2(M)$.    By \fref{lem:YM in O2J}, $Y_M \leq O_2(HO_2(M)) \cap O_2(JO_2(M))$. This shows that $\widetilde {Y_M}$ is contained in the subgroup$$\langle (12)(34), \ldots , (n-3,n-2)(n-1,n)\rangle.$$ In particular any subgroup of $Y_M$, which is normalized by $O^2(L_K)$ is centralized by $O^2(L_K)$. By \fref{lem:compnormal2}, $Q$ normalizes $O^2(L_K)$. But then it also normalizes the fours-group $$ \langle (n-3,n-2)(n-1,n), (n-3,n-1)(n-2,n) \rangle,$$ as this subgroup is  obviously characteristic in $O^2(L_K)$. This is trivial to observe  if $2m > 8$. In the case $2m = 8$, it is $[Z(O_2(O^2(L_K))), O^2(L_K)]$ which is also characteristic. Therefore there exists $z_1 \in Z(Q)^\#$ with $|\mathrm{supp}(z_1)| = 4$, a contradiction to \fref{lem:charpy}(v).

Therefore $|\mathrm{supp}(z)| \geq n-3$. If $|\mathrm{supp}(z)|= n-3$, then $O(L_K) \ne 1$, and we have a contradiction to \fref{lem:charpy}.  Hence  $|\mathrm{supp}(z)| \geq n-2$.  If $2m \not= n-2$, we have that $C_{K}(z) \cong (2 \wr \Sym(m))^e $. If $2m= n-2$, then $C_{K}(z) \cong (2 \times 2 \wr \Sym(m))^e$ which contains a Sylow $2$-subgroup of $K$ only if $n \equiv 2 \pmod 4$. As $n \geq 7$, we have $m \geq 3$. Now $\Sym(m)$ has a non-trivial normal $2$-subgroup if and only if  $m = 4$. Thus so long as $m \ne 4$,  we have that $O_2(L_K)$ is elementary abelian and $L_K/O_2(L_K)\cong \Sym(m)$ induces the non-trivial irreducible part of the natural permutation module on the unique non-central chief factor in $O_2(L_K)$. Finally we note that we have $m=4$ only when  $n \in \{8,9,10\}$.
\end{proof}

We now deal with the three exceptional cases in \fref{lem:alt1}.

\begin{lemma}\label{lem:alt3} We have  $|\mathrm{supp}(z)| \not=8$. In particular $n > 10$.
\end{lemma}

\begin{proof} Suppose that  $|\mathrm{supp}(z)| =8$. Then by \fref{lem:alt1}, $n \in \{8,9,10\}$.  By \fref{lem:ZK=1} $Z(K)=1$.   We may suppose that $z$ corresponds to the permutation $ (12)(34)(56)(78)$.  By \fref{lem:O2normal},  $O_2(M)$ normalizes $K$.

To start assume that $\wt K \cong \Alt(8)$ or $\Alt(9)$. Then there exist $J_1,J_2 \in \mathcal L_K(S\cap K)$ with $J_1  \cong J_2  \cong 2^3{:}\SL_3(2)$ and $J_1\ne J_2$. Both these subgroups are normalized by $O_2(M)$ and hence $$Y_M \le O_2(Y_M J_1) \cap O_2(Y_MJ_2)= C_{Y_MK}(K)(O_2(J_1) \cap O_2(J_2))$$ by \fref{lem:YM in O2J}. Since $|O_2(J_1) \cap O_2(J_2)|=2$, this contradicts \fref{lem:project at least 4}.
 Hence $$\wt{KO_2(M)} \cong \Sym(8), \Sym(9), \Alt(10),\text{ or } \Sym(10). $$Notice that in $\Alt(10)$, $(\Sym(8) \times \Sym(2))^e \cong \Sym(8)$.

We consider  $J \in \mathcal L_K(S\cap K)$  stabilizing the partition $$\{\{1,2,3,4\},\{5,6,7,8\}, \Omega_0\}$$ where $|\Omega_0|\in\{0,1,2\}$. Then
$$J  \cong \begin{cases}(\Sym(4)\wr 2)^e &n\in \{8,9\}\\(\Sym(4)\wr 2 \times 2)^e&n=10 \end{cases}.$$
Notice that  $J$ has characteristic $2$ and is normalized by $O_2(M)$. Setting $J_1= O_2(JO_2(M))$, we have  $Y_M \le J_1$ and $\langle Y_M^{J}\rangle$ is elementary abelian by \fref{lem:YM in O2J}.
We calculate  $$\wt {J_1} = \begin{cases} \left\langle {(12)(34), (13)(24),}\atop{ (56)(78), (57)(68)} \right\rangle&n=8,9 \text { or }  \wt{KO_2(M)} \cong \Alt(10)\\  \left\langle {(12)(34), (13)(24),}\atop {(56)(78), (57)(68), (9,10)} \right\rangle   & \wt{KO_2(M)}\cong \Sym(10)  \end{cases}$$ and $$\wt{J_1} \cap \wt{O_2(L_K)} =\begin{cases} \left\langle {(12)(34), (56)(78),}\atop{ (13)(24)(57)(68) }\right \rangle&n\in\{8,9 \} \text { or }  \wt{KO_2(M)}\cong \Alt(10)\\ \left\langle {(12)(34), (56)(78),}\atop{ (13)(24)(57)(68),(9,10) }\right \rangle&\wt{KO_2(M)}\cong \Sym(10)\\
\end{cases}$$  which has order $8$ in the first cases and $16$ in the second.  As $L_K \leq N_G(Q)$,   the projection  $\wt{Y_M}$    is contained in $\wt{J_1} \cap \wt{O_2(L_K)} $.

Suppose that $n \in \{8,9\}$ or $\wt{KO_2(M)} \cong \Alt(10)$. Then we have $|C_{\wt K}(O^2(\wt {L_K}))| =2$ and so, as $\wt {Y_M}
\le \wt K$,  \fref{lem:YM projects 8} (ii) applies to give $$\wt {Y_M}= \wt{J_1} \cap \wt{O_2(L_K)} .$$
 Pick $\rho \in L_K$ corresponding to  $(1,3,5)(2,4,6)$. As $(13)(24)(57)(68)  \in \wt{Y_M}$, $$(3,5)(4,6)(1,7)(2,8)= ((13)(24)(57)(68))^\rho \in \wt{\langle Y_M^{L_K}\rangle}.$$
Since $(12)(34)$ and $(3,5)(4,6)(1,7)(2,8)$ do not commute, we have a contradiction to $\wt{\langle Y_M^{L_K}\rangle}\le \wt {U_Q}$ being abelian. Hence $\wt{KO_2(M)}\cong \Sym(10)$.

Let $\wt H \le \wt {KO_2(M)}$ be the subgroup that preserves the partition
$$\{\{1,2\},\{3,4\}, \{5,6\},\{7,8\},\{9,10\}\}.$$
Then $\wt H \cong 2\wr \Sym(5)$, and $\wt {Y_M} \le \wt {O_2(H)}$ and we have
$$\wt {Y_M} \le \wt{J_1} \cap \wt{O_2(H)}\le  \left\langle {(12)(34), (56)(78), (9,10)}\right\rangle $$
which has order $8$.
Notice that $ \wt z$ is the only $\wt {KO_2(M)}$-conjugate of $\wt z$ in $\wt {Y_M}$.

By the choice of $z$ we have  $z \in Y_M \cap O_2(M)'$. By \fref{lem:project at least 4}, there exists $m \in M$ such that $\wt {z^m}\ne \wt z$. Obviously $z^m \in Y_M\cap O_2(M)'$ and so $$\wt {z} \in \wt{O_2(M)'\cap Y_M}\le \wt K\cong \Alt(10).$$
Hence $\wt {z^m}$ corresponds to an element of cycle type $2^2$. However this means  $C_{K}(z^m)$ contains a component isomorphic to $\Alt(6)$ and this  contradicts \fref{lem:charpy} (v).
\\

Assume now $n = 10$. Then $|\mathrm{supp}(z)| \not= 8$. By \fref{lem:alt1} this gives $|\mathrm{supp}(z)| = 10$, which contradicts $\wt z \in \wt K$.
\end{proof}

\begin{proposition}\label{prop:altdone} We have $K/Z(K)$ is not an alternating group.
\end{proposition}

\begin{proof}  By \fref{lem:alt3}  we have $n > 10$. Further  \fref{lem:ZK=1} gives us $Z(K)=1$. By the choice of $z$ we have $\wt{K\langle z \rangle} = \wt K$.

Assume that $\wt{Y_M}$  covers the unique non-trivial irreducible $O^2(L_K)$-module in $\wt{O_2(L_K)}$.  Then $C_{\wt K}(\wt{Y_M})= O_2(\wt {L_K})$. By \fref{lem:O2normal} we have that $O_2(M)$ normalizes $K$ and so $\wt {O_2(M)} \le  C_{\wt K}(\wt{Y_M})$ is elementary abelian.  Therefore $z\in O_2(M)' \le C_S(K)$, which is impossible.

 Hence $\wt{Y_M}$  does not cover the non-trivial irreducible $O^2(L_K)$-module in $\wt{O_2(L_K)}$ and so any $O^2(L_K)$-invariant subgroup $W$ of $Y_M$ is centralized by $O^2(L_K)$.  Therefore   \fref{lem:compnormal2} yields  $O^2(L_K)$ is normalized by $Q$. Since $O_2(O^2(L_K))$ is elementary abelian and contains exactly one non-central $O^2(L_K)$-chief factor, \fref{lem:O2L abelian} yields $$[O_2(O^2(L_K)),O^2(L_K)]\le Z(Q).$$  We now notice    $$[O_2(O^2(L_K)),O^2(L_K)]= O_2(O^2(L_K))$$  contains an element $w$ which is $K$-conjugate to the  permutation $(12)(34)$. As $w \in Z(Q)$, $C_G(w)$ has characteristic $2$, hence $C_{C_G(y)}(w)$  has characteristic $2$ by \fref{lem:charpy} (vi)  which it plainly does not. This contradiction shows that $K/Z(K)$ is not an alternating group.
\end{proof}

\section{ Groups of Lie type in characteristic 2 as components}\label{sec:LieChar2}

 In this section we tackle   the possibility that $K/Z(K)$ is a group of Lie type in characteristic two. Some of these groups have been considered before under  different names. For example $\L_2(4) \cong \L_2(5)\cong \Alt(5)$, $\PSp_4(2)^\prime \cong \Alt(6)$, $\L_3(2) \cong \L_2(7)$,  $\G_2(2)^\prime \cong \U_3(3)$, $\L_4(2) \cong \Omega^+_6(2) \cong \Alt(8)$ and $\Omega^-_6(2) \cong \U_4(2) \cong \PSp_4(3)$.

We will start with the groups $\SL_2(2^a)$ and ${}^2\B_2(2^a)$, which then also handles the   case of ${}^2\G_2(3)^\prime$ which was left open in \fref{prop:Lieodd}.

\begin{lemma}\label{lem:rank1-1}
Suppose that $K/Z(K) \cong \SL_2(2^a)$ or ${}^2\B_2(2^a)$, $a \geq 3$. Then
\begin{enumerate}
\item $Y_M \le \Omega_1(S_y)T_y$;
\item $K$ is simple; and
\item $S_y \le O_2(M)$.\end{enumerate}
\end{lemma}

\begin{proof} By Lemmas \ref{lem:centsylow} and \ref{lem:N(SyTy) char 2}, $N_G(S_yT_y)$ has characteristic~$2$. \fref{lem:YM normalizes K} yields $Y_M \le O_2(N_G(S_yT_y))$ and $Y_M$ normalizes every component of $E_y$. In particular, $Y_M$ induces inner automorphisms on each of such component. It follows that $Y_M \le \Omega_1(S_y)T_y$. This proves (i).

Assume that $Z(K) \not= 1$. By \cite[Table 6.1.3]{gls2}, $K/Z(K) \cong {}^2\B_2(8)$, as $a \geq 3$.  Furthermore by \fref{lem:Sz8schur} $Z(S_y \cap K) = Z(K)$.

Hence $K\langle z \rangle = KC_{K\langle z \rangle}(K)$,  $z \not \in K$ and $z \not \in C_{K\langle z \rangle}(K)$. Thus $z= ab$ where $a \in S \cap K$ and $b\in C_{K\langle z \rangle}(K)$. As $z $ and $C_{K\langle z \rangle}(K)$ centralize $S\cap K$, so does $a$. Therefore $a \in Z(S\cap K)=Z(K)$ and we conclude $z\in C_{K\langle z \rangle}(K)$, a contradiction.  This proves (ii).

By (ii) $K$ is simple. Since $S_y$ centralizes $\Omega_1(S_y)T_y \ge Y_M$, we have $S_y \le S \cap C_M(Y_M)=O_2(M)$.  This is (iii).
\end{proof}

\begin{lemma}\label{lem:rank1-done}
We have that $K/Z(K) \not \cong \SL_2(2^a)$ or ${}^2\B_2(2^a)$ with $a\ge 3$.
\end{lemma}

 \begin{proof} Assume that    $K/Z(K)  \cong \SL_2(2^a)$ or ${}^2\B_2(2^a)$ with $a\ge 3$. By \fref{lem:rank1-1} (ii), $K$ is simple.

  We first prove that $J(O_2(M)) = \Omega_1(S_y)\times J(T_y)$.  By \fref{lem:rank1-1} (iii), $S_y \le O_2(M)$ and so we consider
 $X=E_yO_2(M)$. We have $O_2(M) \in \syl_2(X)$.  Using Lemmas \ref{lem:J structure} and \ref {lem:rank1-Thompson},  the fact that $a>2$ yields  $$J(O_2(M)) = J(C_{O_2(M)}(E_y)) \times J(S_y).$$
 Using \fref {lem:rank1-Thompson} again gives $J(S_y)= \Omega_1(S_y)= Z(S_y)$.  From the structure of $J(O_2(M))$, we see that  $N_{E_y}(S_y)$ normalizes $J(O_2(M))$ and therefore $N_E(S_y) \le M^\dagger$ by \fref{lem:charO2M}.  Since $Y_M$ is normalized by $M^\dagger$, it is also normalized by $N_{E_y}(S_y)$. Therefore $[Y_M, N_{E_y}(S_y)] \le Y_M\cap E_y$. Since $Y_M$ does not centralize $N_{E_y}(S_y)$, we deduce that $\Omega_1(S_y)=[Y_M, N_{E_y}(S_y)] <Y_M$. Now we have $J(O_2(M))= J(C_S(E_y))Y_M$. Thus \begin{eqnarray*}[J(O_2(M)),O_2(M)]&=&[J(C_{O_2(M)}(E_y))Y_M,O_2(M)]\\&=&[J(C_{O_2(M)}(E_y)), O_2(M)] \le T_y.\end{eqnarray*}
 As $[J(O_2(M)),O_2(M)] \le T_y$, \fref{lem:Ty cap ZS} implies $$[J(O_2(M)),O_2(M)]=1.$$
  Hence $$ Y_M \leq J(O_2(M)) \leq \Omega_1(Z(O_2(M))) = Y_M.$$ But then $Y_M= J(O_2(M))$ and this contradicts  \fref{lem:YMnotmaxabelian}. The lemma is proved.
 \end{proof}

\begin{lemma}\label{lem:U3q} We have $K/Z(K) \not \cong \PSU_3(q)$ for  $q=2^a\ge 4$.
\end{lemma}

\begin{proof}   Suppose that $K/Z(K) \cong \PSU_3(q)$ with $q \ge 4$.
We have  $Z(K)=1$  by \cite[Table 6.1.3]{gls2} and \fref{lem:char2}. We take facts about $K$ from \cite[5.4]{DeSte} and \cite[II.10.12 Satz]{Hu}. An important point is that $\Omega_1(S\cap K)= Z(S\cap K)$.

  We have  $|L_K|=q^3(q+1)/(q+1,3)$ and so $L_K$ is not a $2$-group. Therefore $C_G(Y_M)C_Q(y)$ normalizes $K$ by \fref{lem:compnormal1}.

  We start with the following statement.

  \begin{claim}\label{clm:abelian in Z} Suppose that $A$ is an elementary abelian normal subgroup of $O_2(M)C_Q(y)$. Then $A \le  \Omega_1(S\cap K) C_{C_S(y)}(K)= Z(S\cap K) C_{C_S(y)}(K)$. \end{claim}

  \medskip

Since $A$ normalizes $K$, is elementary abelian, and   no outer automorphism of $K$ centralizes $(S\cap K)/Z(S\cap K)$,  $A\le (S\cap K) C_{C_S(y)}(K)$. Therefore $$A\le \Omega_1(S\cap K) C_{C_S(y)}(K)= Z(S\cap K) C_{C_S(y)}(K).$$ \qedc

By \fref{clm:abelian in Z}, $Y_M \le  Z(S\cap K) C_{C_S(y)}(K)$.
As $Z(S\cap K)$ is centralized by $L_K$, we obtain $$ O^2(L_K) \le C_G(Y_M).$$

Now, as $O_2(M)$ normalizes $L_K$,  \fref{lem:2group} (iii) yields $F^*(C_G(y)) =KO_2(C_G(y))$ (and part (ii) leads to   $q=8$, but we shall not use this).

Furthermore, \fref{clm:abelian in Z} implies that $[S\cap K, Y_M] \le S\cap K \cap C_S(K)=1$,  $$S\cap K \le C_S(Y_M)=O_2(M).$$

 Assume  $w \in O_2(M)$ has order $2$   and induces an outer automorphism on $K$.  Then, as $O_2(M)$ normalizes $K$ and is contained in $S$, $w$ acts on $N_K(S\cap K)/(S\cap K)$. By \cite[Proposition 4.9.2 (b)(2) and (g)] {gls2}, $w$ is conjugate in $\Aut(K)$ to a standard graph automorphism and so
$w$ centralizes $Z(S\cap K)$ (see \cite[Theorem 2.5.1 d]{gls2}) and $ C_K(w) \cong \SL_2(q)$.  Hence  there is an element of $\nu\in N_{C_K(w)}(Z(S\cap K))$ of order $q-1$.  Since the Sylow $2$-subgroups of $K$ are trivial intersection subgroups in $K$, $\nu$ normalizes $S\cap K$. As $N_K(S\cap K)/(S\cap K)$ is cyclic, $\langle \nu \rangle (S\cap K)$ is uniquely determined by its order in $N_K(S\cap K)/(S\cap K)$, and so $\nu$ normalizes $\Omega_1(O_2(M))(S\cap K)$.  Therefore $\nu$ normalizes $\Omega_1(O_2(M))$ as $S \cap K \le O_2(M)$.  It follows that $\nu \in N_G(\Omega_1(O_2(M))) \le M^\dagger$. Therefore $\nu $ normalizes $Y_{M^\dagger}= Y_M$. As $\langle \nu
\rangle$ acts irreducibly on $Z(S\cap K)$ and normalizes $Y_M$, we have  $$Y_M C_{C_S(y)}(K) =Z(S\cap K) C_{C_S(y)}(K) .$$

Recall that $U_Q$ is elementary abelian. Hence $U_Q \le O_2(M)$ and \fref{clm:abelian in Z} implies that $$U_Q \le Z(S\cap K) C_{C_S(y)}(K) = Y_M   C_{C_S(y)}(K) \le U_Q C_{C_S(y)}(K).$$
Therefore  $$U_QC_{C_S(y)}(K)= Y_MC_{C_S(y)}(K).$$ Hence $[U_Q,O_2(M)]\le C_{C_S(y)}(K)$ and $[U_Q,O_2(M)]$ is normalized by $Q$. We conclude from \fref{lem:Ty cap ZS} that $[U_Q,O_2(M)]=1$ and so  $U_Q= Y_M$ and  $$N_G(Q) \le N_G(Y_M)= M^\dagger.$$ Suppose that $J \in \mathcal L_G(S)$.  Then, as $Y_M$ is not characteristic $2$-tall, $Y_M \le O_2(J)$. Hence $Y_J \le C_S(Y_M)=O_2(M)$ and so $Y_J \le Y_MC_{C_S(y)}(K)$ by \fref{clm:abelian in Z}. Thus $[O_2(M), Y_J] \le C_{C_S(y)}(K)$ and again this is normalized by $Q$. Hence $[O_2(M), Y_J]=1$ and so $Y_J \le Y_M$. This contradicts  \fref{lem:more}. Hence $K/Z(K) \not \cong \PSU_3(q)$.
\end{proof}

\begin{lemma}\label{lem:L34} Suppose that  $K/Z(K) \cong \PSL_3(q)$, $q= 2^a \ge 4$. Then $K/Z(K) \cong \PSL_3(4)$.
\end{lemma}

\begin{proof}
We assume that $q \geq 8$ and seek a contradiction. As $q \ne 4$, \cite[Table 6.1.3]{gls2} and    \fref{lem:char2} combine to give  $Z(K)=1$.  The structural information we require about $S\cap K$ is given in \fref{lem:L3qsylow}. From there we see that $S\cap K$ has  exactly two elementary abelian subgroups of order $q^2$. Name these subgroups $E_1$ and $E_2$. Furthermore, every element of $(S\cap K) \setminus (E_1 \cup E_2)$ has order $4$.

We have $|L_K|= q^3(q-1)/(q-1,3)$ and so, as $q> 4$, $L_K$ is not a $2$-group.  Thus by \fref{lem:compnormal1} $C_G(Y_M)C_Q(y)$ normalizes $K$ and so therefore does $Y_M$. By \fref{lem:L3qsylow}(ii) $Y_M\le (S_y\cap K)C_{C_S(y)}(K)$.  If $Y_M \not \le Z(S\cap K) C_{C_S(y)}(K)$, then we may assume  $Y_M \le E_1C_{C_S(y)}(K)$ and $Y_M \not \le E_2C_{C_S(y)}(K)$. Since $O_2(M)$ centralizes $Y_M$, we infer that $O_2(M)$ normalizes $E_1$.   Hence $O_2(M)$ normalizes $N_K(E_1)$ and also $N_K(E_2)$. Then \fref{lem:YM in O2J} implies that $Y_M \le  E_2C_{C_S(y)}(K)$, which is a contradiction.  Hence

\begin{claim}\label{claim:9.41}$Y_M  \le Z(S_y\cap K) C_{C_S(y)}(K).$\\\end{claim}

\noindent Since $[Z(S_y\cap K),L_K]=1$, we now have $O^2(L_K) \le C_G(Y_M)$ and so, as $O_2(M)$ normalizes $L_K$, $K= E(C_G(y))$ and $M^\circ$ normalizes $L_K=O^2(L_K)$ by \fref{lem:2group}. Therefore

  \begin{claim}\label{claim:9.42}$E_y=K$ and  $M^\circ \text{ normalizes }O_2(L_K)=S_y.$\\\end{claim}

By \fref{lem:L3qsylow1}, $K$ satisfies the Sylow centralizer property (\fref{def:2}) with respect to $C_S(y)$.  Furthermore all elements in $Z(S_y)$ are $K$-conjugate and, as $Q$ normalizes  $Z(S_y)$, they all have a centralizer which has characteristic 2. Thus \fref{lem:Tyabel} yields   that $$T_y\text{ is elementary abelian}.$$

 Since $Y_M  \le Z(S_y) T_y$, $[S_y,Y_M] \le  T_y$.  As $Q \leq M^\circ$ normalizes $S_y$ by \ref{claim:9.42}, \fref{lem:Ty cap ZS} implies that $S_y\le C_S(Y_M)=O_2(M)$.
 Combining \fref{lem:J structure} with \fref{lem:L3qsylow}(iii) yields

 \begin{claim}\label{claim:9.43}$J(O_2(M)) = S_yT_y = E_1E_2T_y.$\\\end{claim}

  \noindent  In particular, \ref{claim:9.43} implies $N_K(S_y) \le N_G(J(O_2(M)))\le M^\dagger$ and so $N_K(S_y)$  normalizes $Y_M$. As $N_K(S_y)$ acts irreducibly on $Z(S_y)$, \ref{claim:9.41} produces $Y_MT_y= Z(S_y) T_y$. Therefore $[T_yZ(S_y),O_2(M)] \leq T_y$. Because $T_yZ(S_y)=J(O_2(M))$ is normalized by $S$,  we obtain $[T_yZ(S_y),O_2(M)] =1$. Thus

\begin{claim}\label{claim:9.46} $Y_M = T_yZ(S_y).$\\\end{claim}

Assume that some element in $S$ conjugates $T_yE_1$ to $T_yE_2$.  Let $A$ be an elementary abelian normal subgroup of $S$ contained in $O_2(M)$.  Then $\wt A \le \wt {S_y}$ by \fref{lem:L3qsylow}(ii).  Therefore, as $S$ does not normalize  $T_yE_1$,  $$A \leq T_yE_1 \cap T_yE_2 = T_y Z(S_y) = Y_M$$ by \ref{claim:9.46}.  Applying \fref{lem:YMnotmaxabelian} provides a contradiction. Thus we have

\begin{claim}\label{claim:9.47} $S \leq N_G(T_yE_1)$.\\\end{claim}

By  \ref{claim:9.43} and \ref{claim:9.47},  $M \leq N_G(E_1T_y)$. Thus $ N_G(E_1T_y) \le M^\dagger$.  As $N_K(E_1) \le N_G(E_1T_y)$ does not normalize $Y_M$, this is impossible.
\end{proof}
%
%
%

\begin{lemma}\label{lem:Sp4} We have   $K/Z(K) \not\cong \Sp_{2n}(q)$ with $n \ge 3$ and  $q=2^a$.
\end{lemma}

\begin{proof}  We follow the notation from \fref{lem:spstruk}. Thus $Z(S_y \cap K)=R_1R_2$, and, for $i=1,2$, $L_i= C_K(R_i)$. We have  $\wt{z}$ is centralized by $\wt{L_1 \cap L_2}$. Thus $L_K$ contains a section isomorphic to $\Sp_{2n-4}(q)$, which is not a $2$-group as $n \ge 3$. Thus \fref{lem:compnormal1} yields $K$ is normalized by $C_G(Y_M)C_Q(y)$. Since $L_1$ and $L_2$ are not isomorphic, we have $L_1$ and $L_2$ are normalized by $O_2(M)$. Lemmas~\ref {lem:YM in O2J} and \ref{lem:constrK} imply that $$Y_MC_{C_S(y)}(K) \le  O_2(N_K(R_1))C_{C_S(y)}(K)   \cap O_2(N_K(R_2))C_{C_S(y)}(K).$$ Furthermore, for $i=1,2$,  $\langle Y_M^{N_K(R_i)} \rangle$ is  elementary abelian. If $Y_M \not \le Z(O_2(N_K(R_2)))C_{C_S(y)}(K)$, then
 \fref{lem:spstruk} (ii) implies that $$\wt{\langle Y_M^{N_K(R_i)}\rangle Z(O_2(N_K(R_2)))}= \wt{O_2(L_2)}$$ which is non-abelian. As $\langle Y_M^{N_K(R_i)}\rangle$ is abelian, we conclude
 $$\wt {Y_M} \le \wt{Z(O_2(N_K(R_2)))}.$$ As $Z((S_y \cap K)/Z(K)) = (Z(O_2(N_K(R_2))) \cap O_2(N_K(R_1)))/Z(K)$, we have that $$\wt{Y_M} \le  Z(\wt{S_y \cap K}).$$ Now $Y_M$ is centralized by $O^2(L_1 \cap L_2)$ and so \fref{lem:compnormal2} gives $K = E_y$ and $S_y= S_y\cap K$.

Assume   that $Z(K) \not= 1$. Then, by \cite[Table 6.1.3]{gls2}, $K/Z(K) \cong \Sp_6(2)$. Further the preimage of $Z(S_y/Z(K))$ is isomorphic to  $\Z_4 \times \Z_2$. This contradicts \fref{lem:project at least 4}. Hence $Z(K) = 1$ and $K$ is simple.

 Since $K=E_y$ is simple, $Y_M \le Z(S_y)T_y$ which implies $[Y_M,S_y] \le K \cap T_y=1$. Thus $S_y\le C_S(Y_M)= O_2(M)$. Using Lemmas~\ref{lem:J structure} and \ref{lem:JS Sp} we obtain $$J(O_2(M)) = J(S_y) \times J(T_y \cap O_2(M)).$$
 Therefore $N_K(J(S_y))\le N_K(J(O_2(M))) \leq M^\dagger$. But $N_K(J(S_y))$  normalizes no non-trivial subgroup in $Z(S_y)$ by \fref{lem:JS Sp}, which is a contradiction as $Y_M \le Z(S_y)T_y$ and $Y_M \not \le T_y$. We have proved $K/Z(K) \not \cong \Sp_{2n}(q)$ with $n \ge 3$ and $q=2^a \ge 2$.
\end{proof}

\begin{lemma}\label{lem:notF4} We have $K/Z(K) \not\cong \F_4(q)$ with $q =2^a$, $a\ge 1$.
\end{lemma}

\begin{proof} By \cite[Table 6.1.3]{gls2} we have $Z(K) = 1$ unless  $q =2$. We use  Lemmas~\ref{lem:F42struk} and \ref{lem:F42middle} for structural information about $K/Z(K)$.  Let $R_1$ and $R_2$ be the preimages of root subgroups of $K$ in $S_y \cap K$. Thus $|R_i|= q$ unless $Z(K) \ne 1$ in which case they are elementary abelian of order $4$. We also let $Z$ be the preimage of $Z((S\cap K)/Z(K))$.

 By \fref{lem:centsylow}, $\wt z \in \wt Z $  and, by \fref{lem:char2}, $K=[K,z]$. We have $L_K \ge I_{12}$ (using the notation of \fref{lem:F42middle}) and so $L_K$ is not a $2$-group. By   \fref{lem:compnormal1}

\begin{claim}\label{clm:7.51}$C_G(Y_M)C_Q(y)$ normalizes $K$.\end{claim}

\medskip

We next intend to show

\begin{claim}\label{clm:7.52} $\wt{Y_M} \le \wt{Z}$.\end{claim}

\medskip
Suppose that $\wt {Y_M} \cap \wt{R_i}\ne 1$ for some $i\in \{1,2\}$. Without loss of generality we assume that $i=1$.  Let $w \in Y_M$ be such that $\wt w \in \wt{R_1}^\#$. Then  $O^{2}(C_K(R_1))$ centralizes $w$. Hence $O_2(M)$ normalizes $O^2(C_K(R_1))$.  It follows that $O_2(M)$ normalizes $N_K(R_1)$ and $N_K(R_2)$.  By \fref{lem:constrK} $$\wt{Y_M} \le \wt{O_2(N_K(R_1))} \cap \wt{O_2(N_K(R_2))}.$$
For $i=1,2$, set $W_i = \langle Y_M^{N_K(R_i)} \rangle$. Then $W_1$ and $W_2$ are elementary abelian and
 \fref{lem:F42struk} gives  $\wt{W_i}\le  \wt{Z(O_2(N_K(R_i)))}$. By \fref{lem:F42middle} (ii), $Z(O_2(N_K(R_1))) \cap Z(O_2(N_K(R_2))) = Z $. Therefore $\wt{Y_M} \le \wt{Z}$  in this case.

 To complete the proof of \fref{clm:7.52} we   assume that $$\wt {Y_M} \cap \wt{R_1}= \wt {Y_M} \cap \wt{R_2}=1.$$ Since $O_2(M)$ normalizes $I_{12}$,  \fref{lem:constrK} implies that $\wt{Y_M}\le \wt{O_2(I_{12})} \le \wt K$.  Thus $\wt{Y_M}$ normalizes $\wt{O_2(N_K(R_i))}$ for $i=1,2$.  If, for some $i$, we have  $\wt{Y_M}\cap  \wt{O_2(N_K(R_i))} \not \le Z(\wt{O_2(N_K(R_i))})$, then $1\ne [\wt{O_2(N_K(R_i))}, \wt{Y_M}] \le \wt{R_i} \cap \wt{Y_M}=1$, a contradiction. Hence $\wt{Y_M}\cap  \wt{O_2(N_K(R_i))}   \le Z(\wt{O_2(N_K(R_i))})$ and so \fref{lem:F42struk} implies that first $\wt{Y_M} \le \wt{O_2(N_K(R_i))}$ for  both  $i=1,2$  and then that $$\wt{Y_M} \le Z(\wt{O_2(N_K(R_1))})\cap Z(\wt{O_2(N_K(R_2))}) = \wt Z.$$ This completes the proof of \fref{clm:7.52}.\qedc

 Since $O^2(I_{12})$ centralizes $Z$, we have $O^2(I_{12}) \le C_K(Y_M)\le C_G(Y_M)$. Hence $J=C_K(Y_M)$ is not a $2$-group.   Notice that $J= O^2(N_K(R_1))$ or $J=O^2(N_K(Z(S\cap K)))$, Lemmas~\ref{lem:F42struk} and \ref{lem:F42middle} show that the non-central $O^2(J)$-chief factors in $O_2(O^2(J))/\Phi(O_2(O^2(J)))$ are pairwise non-isomorphic. In this way  \fref{lem:2group} provides a contradiction. This proves $K/Z(K) \not \cong \F_4(q)$.
\end{proof}

\begin{lemma}\label{lem:nottwistedF4} We have $K/Z(K) \not\cong {}^2\F_4(q)^\prime$ with $q=2^a \ge 2$.
\end{lemma}

\begin{proof} Suppose false. By \fref{lem:tits} we have $q > 2$ and  \cite[Theorem 6.1.4]{gls2} states  $Z(K) = 1$. Furthermore  \cite[Theorem 2.5.12]{gls2} gives  $\Out(K)$ has odd order.

For the structure of the 2-local subgroups of $K$ we shall  use \fref{lem:F4twiststruk} in the arguments that  follow, without specific reference. In particular, we know $L_K$ is not a $2$-group and so \fref{lem:compnormal1} yields
$$C_G(Y_M)C_Q(y) \text{ normalizes } K.$$

We introduce some notation. Let $R = Z(S \cap K)$ and $W_1 = Z_2(O_2(L_K))$. Then $L_K=C_K(R)$,  $|W_1| = q^5$ and $W_1/R$ is the natural $L_K/O_2(L_K)$-module where $L_K/O_2(L_K) \cong{}^2\B_2(q)$. Put $W_2 = C_K(W_1)$. Then $|W_2| = q^6$ and $W_1 = \Omega_1(W_2)$. Let $P$ be the maximal parabolic subgroup of $K$ containing $N_K(S\cap K)$ but not containing $L_K$.

Since $L_K$ and $P$ are normalized by $O_2(M)$, by \fref{lem:YM in O2J} we have  $Y_M \le O_2(L_KO_2(M))\cap O_2(PO_2(M))$ and $\langle Y_M^{L_K}\rangle$ and $\langle Y_M^{P}\rangle$ are elementary abelian.

As $\wt {\langle Y_M^{L_K}\rangle}$ is elementary abelian and normal in $\wt L_K$, we have   $\wt {Y_M} \le \wt {\langle Y_M^{L_K}\rangle} \le \wt {W_1}$. If $\wt{Y_M} =\wt {W_1}$, then $\wt {\langle Y_M^{P}\rangle}$ is not abelian, a contradiction.  Hence $\wt {Y_M} < \wt{W_1}$.  In particular, any subgroup $W$ of $Y_M$ which is normalized by  $O^2(L_K)$ has $\wt W \le \wt R$ and so $W$ is centralized by $O^2(L_K)$. \fref{lem:compnormal1} implies that $$E(C_G(y))= K \text{ and } Q \text{ normalizes }O^2(L_K).$$ In particular, some element of $R$ is centralized by $Q$. Since all root elements are conjugate, all the involutions in $R$ have characteristic $2$ centralizers. Therefore \fref{lem:Tyabel} yields $T_y$ is elementary abelian of order at most $q$.
Since $O_2(M) \le S_yT_y$, and $T_y \le Z(S_yT_y)$, $T_y \le Y_M$. Hence $Y_M = (Y_M \cap K)T_y$.

Assume that $\wt {Y_M} \le \wt R$.  Then $L_K=O^2(L_K) \le C_G(Y_M)$.  Hence \fref{lem:2group} implies that $$L_K = O^2(L_K)\text{ is normalized by }M^\circ.$$ Considering the action of $Q$ on $V=O_2(L_K)/W_1$, which is an indecomposable $5$-dimensional $\GF(q)$-module for $L_K/O_2(L_K)$, we see that $C_V(Q)$ has a non-trivial $O^2(L_K)$ chief factor by the $A\times B$-lemma. Hence $Q$ centralizes $V$. As $Z_2(O_2(L_K))= \Phi(O_2(L_K))$, $O^2(M^\circ)$ centralizes $O_2(L_K)$ which is   normalized by $Q$ and so this contradicts \fref{lem:McircZ(Q)}. Hence $$\wt Y_M \not \le \wt R.$$

 Because  $\wt {Y_M}\not \le \wt R$, $\wt {Y_M}\cap   Z_2(\wt  {S_y})\not \le \wt R$ and so as $Z(O_2(P))$ is a natural $P/O_2(P)$-module where $P/O_2(P) \cong \SL_2(q)$, $[\wt {Y_M}\cap   Z_2(\wt {S_y}),\wt {S_y}] = \wt R$ and  we deduce $[ S_y,Y_M] \ge R$.
Since $O_2(M)P= C_{O_2(M)P}(K)P$ and $\wt {O_2(M)}$ centralizes $\wt{Y_M \cap Z_2(S_y)}$ we conclude that $P$ normalizes $O_2(M)T_y$. But then $P$ normalizes $O_2(M)$ as $O_2(M)$ is weakly closed in $S$ by \fref{lem:weakcl}.  In particular, $Y_M \cap K$ is normalized by $P$ and so $Y_M \cap K \ge Z_2(S_y)$.

 Since $O_2(L_K)/C_{O_2(L_K)}(W_1)$ and $W_1/R$ are irreducible $O^2(L_K)$-modules, $Q \leq M^\circ$ normalizes $L_K$ and $L_K \leq C_G(z)$ normalizes $Q$, $Q$ centralizes these sections. Therefore $$[O_2(L_K),Q, W_1] =1$$ and $$[W_1,Q,O_2(L_K)] =1.$$ The Three Subgroups Lemma implies that $[W_1,O_2(L_K),Q]=1$.  Since $[W_1,O_2(L_K)]$ is normalized by $N_K(R)$ and $N_K(R)$ permutes the involutions in $R$ transitively, $R=[W_1,O_2(L_K)] \le Z(Q)$. Let $\omega \in O^2(L_K)$ have order $q+\sqrt{2q}+1$. Then $\omega$ acts fixed-point-freely on the natural module for $L_K/O_2(L_K)\cong {}^2\B_2(q)$. Since $\omega$ normalizes $Q$, we have $Q= [Q,\omega]C_Q(\omega)$ and $[Q,\omega]\le L_K$.  Now $C_Q(\omega)$ normalizes $[W_1,\omega]$ and as $[W_1,Q] \le R=C_{W_1}(\omega)$, we have $$[[W_1, \omega], C_Q(\omega)]\le  [W_1,\omega]\cap C_{W_1}(\omega)=1.$$ Hence $C_Q(\omega)$ centralizes $W_1=R[W_1,\omega]$.

By \fref{lem:YMnotmaxabelian}, $O_2(M)$ is not abelian. Since $O_2(M) \le S_y T_y$ and $T_y$ is abelian, we have $O_2(M)' \le S_y$ and $O_2(M)' \cap Y_M$ is normalized by $P$. In particular, $$ R < Z(O_2(P))\le O_2(M)' \cap Y_M.$$

 Choose $x \in C_Q(\omega) \setminus C_Q(y)$ such that $x^2 \in C_Q(y)$.  Then $$x^2 \in C_{S_yT_y}(\omega) = C_{S_y}(\omega)T_y \le C_{S_yT_y}(W_1T_y) \le C_{S_yT_y}(Y_M) \le O_2(M).$$
We consider the action of $x$ on $Y_M$.  As $x$ centralizes $W_1$ and $T_y \cap T_y^x=1$ by \fref{lem:Ty TI}  and \fref{lem:NormalTy}, $C_{Y_M}(x) = Y_M \cap W_1= Y_M \cap K$. As $x^2 \in O_2(M)$,   $$[Y_M,x] \le C_{Y_M}(x) =Y_M\cap K.$$   We also have $[O_2(L_K),x] \le [O_2(L_K),Q] \le C_{O_2(L_K)}(W_1)$ and so $$[O_2(L_K),x,Y_M] = 1.$$ As $[Y_M,O_2(L_K),x]\le [R,x]=1$, the Three Subgroups Lemma implies that $$[Y_M,x] \le  C_{Y_M}(O_2(L_K))=R\le O_2(M)' \cap Y_M \le K \cap Y_M = C_{Y_M}(x). $$ Hence $x$ centralizes the sections of the $M$-invariant series $$Y_M >Y_M \cap O_2(M)'>1$$ which means that $x \in O_2(M/C_M)=1$. Thus $x$ centralizes $Y_M$. But then $x$ centralizes $y$ and this is impossible. Therefore $C_Q(\omega) \le C_Q(y)$ and $Q=[Q,\omega]C_Q(y) \le O_2(L_K)C_Q(y) \le C_G(y)$, a contradiction.
\end{proof}

\begin{lemma}\label{lem:notLnq}  Suppose that  $K/Z(K) \cong  \PSL_n(q)$, $n \ge 4$ and $q=2^a$. Then $K/Z(K) \cong \PSL_3(4)$.
\end{lemma}

\begin{proof} Suppose false. By Lemmas \ref{lem:rank1-done} and \ref{lem:L34}, $n \geq 4$. Furthermore, as $\PSL_4(2) \cong \Alt(8)$, we have by \fref{prop:altdone} that $K/Z(K) \not\cong \PSL_4(2)$. Now by Lemmas \ref{lem:centsylow} and \ref{lem:char2} and \cite[Theorem 6.1.4]{gls2} we have $K$ is simple.
 Let $R=Z(S\cap K)$. Then  \fref{lem:centsylow} implies that  $\wt z\in \wt R $. Thus  $L_K$ is not a $2$-group and so $K$ is normalized by $C_G(Y_M)C_Q(y)$ by \fref{lem:compnormal1}.
 Since $(n,q)\ne (4,2)$, Lemmas \ref{lem:irr} and \ref{lem:irr2} show that $O_2(L_K)/R$ is a direct sum of two non-isomorphic $C_K(R)$-modules. Let their preimages be $E_1$ and $E_2$. We have $E_1$ and $E_2$ are elementary abelian and they have order $q^{n-1}$.

 If $Y_M \le RC_{C_S(y)}(K)$, then $O^2(L_K)$ centralizes $Y_M$ and the aforementioned module structure of $O_2(L_K)$ provides a contradiction via \fref{lem:2group} (ii).  Therefore  $Y_M \not \le RC_{C_S(y)}(K)$. Since $U_Q$ is elementary abelian and is normalized by $L_K$ and $U_Q \not \le RC_{C_S(y)}(K)$, we have, without loss of generality, $U_QC_{C_S(y)}(K) = E_1$. Hence $O_2(M)$ normalizes $E_1$ and hence also $E_2$.  But then $Y_M \le O_2(O_2(M)N_K(E_2)) \le E_2C_{C_S(y)}(K)$ by \fref{lem:constrK}.  Therefore $Y_M \le  E_1C_{C_S(y)}(K) \cap  E_2C_{C_S(y)}(K) = RC_{C_S(y)}(K)$ and we have a contradiction. \end{proof}

\begin{lemma}\label{lem:notG24} We have $K/Z(K) \not\cong \G_2(4)$.
\end{lemma}

\begin{proof} By  \cite[Table 6.1.3]{gls2} we have $|Z(K)| \le 2$. Let $R$  be the preimage of $Z((S\cap K)/Z(K))$. By  \fref{lem:centsylow},  $\wt z \in \wt R^\# $.   As $U_Q$ is elementary abelian and normalized by $L_K$, application of  \fref{lem:G2irr} shows that  $Y_MC_S(K)$ is centralized by $O^2(L_K)$.  Therefore $K= E_y$ and $M^\circ$ normalizes $O^2(L_K)$ by \fref{lem:2group}. Since, by \fref{lem:project at least 4}, $4 \le |Y_MT_y/T_y|  \le |Z(S\cap K)T_Y/T_Y|\le 4$, we have $Y_MT_y/T_Y =Z(S\cap K)T_Y/T_Y$ has order~$4$. Lemmas~\ref{lem:YM8} and \ref{lem:|Y_M| bound} imply that $$|Y_M|=16.$$
Furthermore by \fref{lem:Tyabel} we have that
$$T_y = Y_M \cap T_y.$$
Notice that $O_2(M) \le C_{S_yO_2(M)}(R/Z(K))$ and so $O_2(M)$ is normalized by $N_K(S\cap K)$. This subgroup contains an element $\rho$ of order $3$ which operates  non-trivially on $R/Z(K)$. It follows that $[Y_M,\rho]$ has order $4$ and $\rho$ centralizes $C_{Y_M}(K)$.  Since $z\in O_2(M)' \le K$ and we have $Y_M \cap O_2(M)'$ has order $4$ or $8$.  Suppose first that the order is $8$.  Then $1\ne Z(K) \le R$. Since $O_2(M^\dagger/C_M)=1$ and $M^\dagger$ is not transitive on $R$, we have $M^\dagger/C_M \cong \Sym(3)$ and $Z(K)=C_{R}(\rho)$ is centralized by $Q$, a contradiction. Hence $R= Y_M \cap O_2(M)'$ has order $4$. Since $\rho$ centralizes $Y_M/R$, we have $\langle \rho\rangle$ is normalized by $S$ and hence so is $C_{Y_M}(\rho)=C_{Y_M}(K)$.  But then $Z(Q) \cap T_y \ne 1$, and so we have a contradiction.
\end{proof}

\begin{proposition}\label{prop:complie2}  If $K/Z(K)$ is a  simple group of Lie type in characteristic $2$, then $K/Z(K) \cong \PSL_3(4)$ or $\Sp_{4}(q)$, $q= 2^a \geq 4$.
\end{proposition}

\begin{proof}  Suppose false and let $q=2^a$. Then combining the lemmas of this section,  we have $K/Z(K) \cong \U_n(q)$, $n \geq 4$,  $\Omega^\pm_{2n}(q)$, $n \geq 4$, $\G_2(q)$, $q \geq 8$, ${}^3\mathrm D_4(q)$, ${}^2\E_6(q)$ or $\E_n(q)$, $n = 6,7,8$. Also, by \fref{prop:Lieodd}, $K/Z(K) \not \cong \PSU_4(2)\cong \PSp_4(3)$.   \fref{lem:centsylow} implies that $z$ acts on $K/Z(K)$ as a $2$-central element. Set $L_K= C_K(z)$ and $J= C_{O^2(L_K)}(Z(O_2(O^2(L_K))))$.  Then    \fref{lem:irr} implies $O_2(O^2(L_K))$ is non-abelian and  $O^2(J)$ acts irreducibly on $O_2(O^2(L_K))/Z(O_2(O^2(L_K)))$. This contradicts   \fref{lem:special} and proves the proposition.
\end{proof}

\section{The groups $\PSL_3(4)$ and $\PSp_4(q)$ as components}

In this section we assume that $K/Z(K) \cong \PSL_3(4)$ or $\PSp_4(q)$, $q=2^a > 2$. We will show that this is not possible. By \cite[Theorem 6.1.4]{gls2} we have $Z(K)  = 1$ if $K \cong \PSp_4(q)$. By \fref{lem:centsylow}, $z$ acts as an inner automorphism on $K$ and so $Z(K)$ is a $2$-group by \fref{lem:char2}.

\begin{lemma}\label{lem:ZK} If $K/Z(K) \cong \PSL_3(4)$ and $Z(K) \not= 1$, then $Z(K)$ is elementary abelian of order at most $4$.
\end{lemma}

\begin{proof} Assume that $Z(K)$ contains an element of order four.
Then $Z(S\cap K) = Z(K)$ by \fref{lem:ZK1} and $z$ acts on $K$ as an element of $Z(S\cap K)$ and so $z$ centralizes $K$. This contradicts  \fref{lem:char2}. Thus $Z(K)$ is elementary abelian of order at most 4 by  \cite[Theorem 6.1.4]{gls2}.
\end{proof}

We now establish the notation which will be used throughout this section.  We write $$E_y= K_1\dots K_r$$ where each $K_i$ is a component of $C_G(y)$ with $K_i/Z(K_i) \cong K/Z(K)$ and $|K_i|= |K|$.  For $1 \le i \le r$, we define  $$S_i= S_y\cap K_i.$$  If $K\cong \PSp_4(q)$, we let $E_{ij}$, $j=1,2$,  be the maximal order elementary abelian subgroups of order $q^3$ in $S_i$ described in \fref{lem:sp4sylow}. If $K/Z(K)\cong \PSL_3(4)$, then we let $E_{ij}$, $j=1,2$, be the elementary abelian subgroups of order $16|Z(K)|$ as described in \fref{lem:L34facts} (i). In all cases  we have that every elementary abelian subgroup of $S_i$ is contained in $E_{i1}$ or in $E_{i2}$. When discussing a fixed component $K$, we often abbreviate our notation using $E_1$ and $E_2$ in place of $E_{i1}$ and $E_{i2}$.

Define $$D_y = J(O_2(M) \cap T_y).$$
The proof  takes different directions depending upon whether or not  $D_y$ is abelian.

\begin{lemma}\label{lem:Sp4new} Suppose that  $K/Z(K) \cong \Sp_{4}(q)$, $q=2^a \ge 4$. Then $Z(K) =1$ and no element of $\Omega_1(Z(S))$ projects on to a root element of $\wt{K}$.
\end{lemma}

\begin{proof} By \cite[Table 6.1.3]{gls2} we have $Z(K)=1$. Let $Z(S\cap K)= R_1R_2$ with $R_1$ and $R_2$ root subgroups. Suppose that $z\in \Omega_1(Z(S))$ is such that $\wt z$ is a root element in $\wt{S\cap K}$.   Then $L_K$ is not a $2$-group and so \fref{lem:compnormal1} implies $C_G(Y_M)C_Q(y)$ normalizes $K$ and then $C_Q(y)$ normalizes  $O^2(L_K)$.

Suppose $\wt{Y_M}=\wt{O_2(N_K(R_1))}.$  Then $O_2(M)$ normalizes $N_K(R_1)$ and $N_K(R_2)$ and $Y_M \not \le O_2(N_G(R_2))$. Employing   \fref{lem:YM in O2J} and \fref{lem:constrK} we have a contradiction.
Now  \fref{lem:sp4sylow}  shows that $O^2(L_K)$ centralizes any subgroup of $Y_M$, which is normalized by $O^2(L_K)$. Application of  \fref{lem:compnormal2} shows that  $Q$ normalizes $O^2(L_K)$.  Then, as $O_2(L_K)$ is elementary abelian and contains exactly one non-central $O^2(L_K)$-chief factor, \fref{lem:O2L abelian} implies $O_2(L_K) \le Z(Q)$.  Hence $K= \langle C_K(R_1),C_K(R_2) \rangle \le N_G(Q)$, a contradiction.  This proves the lemma.\end{proof}

We remark that the next lemma does not require that $|C_S(y)|$ is chosen to be maximal.

\begin{lemma}\label{lem:rootinYM} Suppose that $y \in \mathcal Y^*_S$. Then the following hold.
 \begin{enumerate} \item $N_{E_y}(S_y) \le M^\dagger$.
 \item $Y_M = (Y_M \cap T_y)(Y_M \cap S_y)$ and $(Y_M \cap K)Z(K) =Z(S_y \cap K)$.
 \item $O_2(M)$ normalizes $K$ and $J(O_2(M)) = S_yD_y$.
 \item Either $O_2(M) = S_y(O_2(M) \cap T_y)$ or  $\wt{O_2(M)K}$ is isomorphic to  $\PSL_3(4)$ extended by a graph automorphism.\end{enumerate}
\end{lemma}

\begin{proof} By \fref{lem:YM normalizes}, $Y_M$ normalizes $K$. Thus $[S_y \cap K,Y_M] \leq Y_M \cap K$.

Assume that $\wt {Y_M}\cap \wt K \not \le Z( \wt{S_y\cap K})$. Then  $[Y_M,S_y \cap K] \not\le Z(K)$ and so $Y_M \cap K \not\leq Z(K)$. Therefore, as $O_2(M)$ centralizes $Y_M$,  $O_2(M)$ normalizes $K$. We may assume that $\wt{Y_M}\cap \wt{S_y\cap K}$ is contained in $\wt{E_1}$ but not in $Z(\wt{S_y \cap K})$.  In particular $\wt{O_2(M)}$ normalizes $\wt{E_1}$. But then $O_2(M)$ normalizes $E_1$ and also normalizes $E_2$. We have that  $J = N_{K}(E_2)$ is of characteristic 2 and is normalized by $O_2(M)$. However $\wt {Y_M} \not\leq \wt{E_2}$ and this contradicts \fref{lem:YM in O2J}. Thus

\begin{claim}\label{claim:11.31}$\wt {Y_M} \cap \wt {S_y\cap K} \le Z(\wt{S_y\cap K})$.\\\end{claim}

 Assume there exits $x \in Y_M^\#$, which induces an outer automorphism on $K$. Then $[\wt x,\wt{S_y \cap K}]\le [\wt {Y_M}, \wt {S_y\cap K}]\le \wt {Y_M}\cap \wt {S_y\cap K} \leq Z(\wt{S_y \cap K})$. In particular $x$ cannot interchange $E_1$ and $E_2$. This yields that $x$ induces a field automorphism on $K$. But such an automorphism is non-trivial on $E_1/Z(S_y \cap K)$. Therefore $\wt {Y_M}  \le \wt{S_y \cap K}$ which then means by \ref{claim:11.31}

  \begin{claim}\label{claim:11.32}$\wt {Y_M} \le Z(\wt{S_y\cap K})$.\\\end{claim}

    That is $$Y_M \le Z(S_y\cap K)C_{C_S(y)}(K)\le C_{C_{S}(y)}(S_y\cap K).$$
 Since this is true for all the components of $E_y$, we have

 \begin{claim}\label{claim:11.33}$S_y \le C_S(Y_M)= O_2(M)$.\\\end{claim}

Consider $\ov {E}_y = E_yO_2(M)/Z(E_y)$. Then $\ov{O_2(M)}$ is a Sylow 2-subgroup of $\ov{E_y}$. We have  $ \ov{E_1}\ov{E_2} = J(\ov {S_y \cap K})= J(\ov {O_2(M)})$ by \fref{lem:L34facts} and \fref{lem:sp4sylow}. Therefore, by \fref{prop:JS normalizes K} and Lemmas \ref{lem:l34J} and \ref{lem:ZK}     we get  $J(O_2(M))$ normalizes $K$. It follows that $J(O_2(M)) \cap E_y = S_y$ and $J(O_2(M)) = S_yD_y$ by \fref{lem:J structure} and the definition of $D_y$.   In particular, $J(O_2(M))$ is normalized by $N_{E_y}(S_y)$ and so $N_{E_y}(S_y) \le M^\dagger$ by \fref{lem:charO2M}. Hence  (i) is true.

Using (i) and the fact that $1\ne \wt {Y_M}\le Z(\wt {S_y\cap K})$ by \ref{claim:11.32}, we have $ \wt {Y_M}=Z(\wt {S_y\cap K}) $, as by \fref{lem:Sp4new}  $\wt Y_M$ is not contained in a root group when $K \cong \Sp_4(q)$. Further  $$1 \not=[Y_M,N_K(S_y\cap K)] Z(K)/Z(K) = Z((S\cap K)/Z(K)).$$  In particular, $Y_M \cap K \not \le Z(K)$ and so $O_2(M)$ normalizes $K$. Furthermore, letting $C$ be a complement to $S_y$ in $N_{E_y}(S_y)$, we have $$Y_M= [Y_M,C]C_{Y_M}(C)= (Y_M\cap S_y)(Y_M \cap T_y)$$ and $[Y_M,C]Z(K) = Z(S_y)$. This is (ii).

We have just seen that $O_2(M)$ normalizes all the components of $E_y$ and $S_y\le O_2(M)$. We have also proved  $J(O_2(M))= S_y D_y$. This is (iii).

Since $O_2(M)$ normalizes $K$ and $\wt{O_2(M)}$ centralizes $Z(\wt{S_y \cap K})$ by (ii), either $\wt{O_2(M)} = \wt{S_y} $ or $\wt{O_2(M)K}$ is isomorphic to  $\PSL_3(4)$ extended by a graph automorphism (see \cite[Chapter 10, Lemma 2.1]{gls5}). Hence (iv) holds.
\end{proof}

\begin{lemma}\label{lem:E} Suppose that $D_y$ is  abelian and $r \ge 2$.
 If  $x \in (Y_M \cap K_i )\setminus Z(K_i)$ for some $1 \le i \le r$, then $E_x=\prod_{j \not= i} K_j < E_y$.
\end{lemma}

\begin{proof}  By \fref{lem:K into E_w}, $\prod_{j \not=i}K_j \leq E_x$ and $\prod_{j \not= i}K_j$  is non-trivial as  $r \ge~2$.
 Assume that $E_x > \prod_{j \not= i}K_j$. Then there is a component $L$ of $C_G(x)$ with $L/Z(L) \cong K/Z(K)$, $|L|=|K|$ and $L \not \le \prod_{j \not= i}K_j$. By \fref{lem:K into E_w},
  the normal closure of $\prod_{j \not= i}K_j$ in $E_x$ has at least $r-1$ components of $C_G(x)$ and only has $r-1$ components if every component of $\prod_{j \not= i}K_j$ is a component of $C_G(x)$.
 Therefore $E_x$ has exactly $r$ components.

Suppose that $K$ is simple. Then $E_x \cong E_y$ and we can apply \fref{lem:rootinYM} to find $S_x \cap L=O_2(M)  \cap L \in \syl_2(L)$ and
$S_x\cap L\le J(O_2(M))=S_yD_y$.  Consider  $K_j \le E_x$ and let $C$ be a complement to $N_{K_j}(S_j)$. Then $[S_yD_y,C]= S_j$ and so, as $C$ normalizes $L$,  $$[S_x \cap L, C] \le S_j \cap L \le K_j \cap L \le Z(K_j)=1.$$ Hence, temporarily setting $\ov {E_x} = E_x/C_{E_x}(L)$, we have  $$\ov C \le C_{\ov L}(\ov{S_x\cap L})= Z(\ov {S_x \cap L})$$ and this means that $C \le C_{E_x}(L)$. Hence $K_j \cap C_{E_x}(L) \not \le Z(K_j)$ which means that $K_j$ centralizes $L$. Therefore $L$ centralizes $\prod_{j \not= i}K_j$ and so
 $$S_x \cap L \le C_{J(O_2(M))}( \prod_{j \not= i}K_j) = D_yS_i.$$ Since $D_y$ is abelian, this shows that $(S \cap L)' = S _i'$.  As $x\in S _i'\le L$ and $x$ centralizes $L$, we deduce $Z(L) \ne 1$, a contradiction. Hence
 $$Z(K) \not =1.$$

 As $O_2(M)$ normalizes $K_j$, $1 \le j \le r$, by \fref{lem:rootinYM} (iii), we can choose an involution $w\in Z(\prod_{j \not= i}K_j) \cap Y_M$.  Then $w \in D_y$ and $E_w= E_y$ by \fref{lem:comps to max}. Since $L$ is a component of $C_G(x)$ and $K_j$ is quasisimple for $2 \le j \le r$, we have $$L/Z(L) \cong L C_{E_x}(L)/C_{E_x}(L) \cong K_jC_{E_x}(L)/C_{E_x}(L)\cong K_j/Z(K_j)$$ and so $\prod_{i\ne j} Z(K_j) \le C_{E_x}(L)$. It follows that $L$ centralizes $w$ and so $L$ normalizes $E_w=E_y$. Thus $L$ normalizes  $E(C_{E_y}(x )) =\prod_{j \not= i}K_j$ and so $L$ normalizes $K_i= E(C_{E_y}(\prod_{j \not= i}K_j))$. Since $L$ normalizes $\prod_{j \not= i}K_j$ and $L$ is a component in $E_x$, $L$ centralizes $\prod_{j \not= i}K_j$.  Because  $L$ centralizes $x$ and $L$ is a component of $C_G(x)$, $L$ centralizes $C_{K_i}(x)$. We know that $C_{K_i}(x)= S_i$. Hence $L$ centralizes $ S_i \prod_{j \not= i}K_j \ge S_y$. From \fref{lem:rootinYM}(iii), $J(O_2(M)) = D_yS_y$.  Therefore, as  $D_y$ is abelian, $  1\ne J(O_2(M))' =S_y'$. Thus $L$ is a component in $C_G(J(O_2(M))' )$ and this contradicts \fref{lem:charpy} (ii). Thus $E_x=\prod_{j \not= i}K_j$, a claimed
\end{proof}

\begin{lemma}\label{lem:L3Sp4DyNA}
$D_y$ is non-abelian.
\end{lemma}

\begin{proof} Assume that $D_y $ is abelian. By \fref{lem:rootinYM} (ii), $J(O_2(M))=D_yS_y$. As $D_y= J(D_y)$, $D_y$ is elementary abelian.
Now $\Omega_1(Z(J(O_2(M)))) = D_yZ(S_y \cap K) = D_yY_M$ by \fref{lem:rootinYM}(ii). Hence $[D_yY_M,O_2(M)] \leq D_y$ and as $Z(S) \cap D_y = 1$, we see that $[D_y,O_2(M) ] = 1$. Hence
$$D_y = Y_M \cap T_y = \Omega_1(T_y).$$

Recall that $S _i$ contains exactly two maximal rank  elementary abelian subgroups $E_{i1}$, $E_{i2}$ of order $16|Z(K_i)|$ if $K/Z(K) \cong \PSL_3(4)$, and $q^3$ otherwise. Thus the set of maximal order elementary abelian subgroups in $D_yS_y$ is
$$\mathfrak  A = \left\{D_y\prod_{i = 1}^r E_{i i_j}\mid  i_j \in\{1,2\}\right\}$$
and $M$ permutes $\mathfrak  A$ by conjugation. The group $M$ also permutes the pairs $(F_1,F_2)\in \mathfrak  A \times \mathfrak  A $ which have the property that $$|F_1 : F_1 \cap F_2| = |F_2 : F_1 \cap F_2| = \begin{cases} 4& K/Z(K) \cong \PSL_3(4)\\q&\text{otherwise}\end{cases}.$$  Then $M$ permutes the set of commutators  $[F_1,F_2]$ for all such pairs $(F_1,F_2) $.
Let the set of such commutators be $\Theta$. Then, as $[E_{i1},E_{i2}]= S_i'$, $$\Theta=  \{ S_i'\mid 1 \le i \le r\}$$  and we have explained that
$$M \mbox{ permutes the groups in }\Theta.$$

Assume that $r>1$.
Then, as $M$ permutes $\Theta$,
 $M$ normalizes
\begin{eqnarray*}N = \langle L&\mid& L \mbox{ a component of }C_G(Y_M \cap K_i),\\&& L/Z(L) \cong K/Z(K) \text{ and } |L|=|K|, i = 1, \ldots ,r\rangle.\end{eqnarray*}
By Lemmas~\ref{lem:rootinYM}  and \ref{lem:E}, $N= \langle \prod_{j\ne i}K_j\mid 1 \le i \le r\rangle = E_y$ is normalized by $M$. It follows that $D_y = C_{J(O_2(M))}(E_y)$ is normalized by $M$ and this contradicts \fref{lem:Ty cap ZS} as $y \in D_y$ shows that $D_y\ne 1$. Thus
$$r=1 \text{ and }\mathfrak  A= \{D_yE_{11}, D_yE_{12}\}.$$
  If $D_yE_{11}$ is normal in $M$, then  $N_K(D_yE_{11})= N_K(E_{11}) \leq M^\dagger$, but by \fref{lem:rootinYM} $Y_M$ is not normal in $N_K(E_{11})$. Hence  $D_yE_{11}$ and $D_yE_{12}$
  are conjugate in $M$ and so in $S$.  Suppose that $A\le O_2(M)$ is an elementary abelian normal subgroup of $S$. By \fref{lem:rootinYM}, we either have $O_2(M)= (O_2(M) \cap T_y)J(O_2(M))$ or $K/Z(K) \cong \PSL_3(4)$ and $O_2(M)$ induces a graph automorphism. In the latter case, \fref{lem:L34facts} implies $A \le S_yT_y=J(O_2(M))$. Hence always $A \le \Omega_1(T_yS_y) =  J(O_2(M))$, and so, as $A$ is normal in $S$, we obtain $A \le D_yE_{11} \cap D_yE_{12} = D_y(E_{11}\cap E_{12}) = Y_M$, as $D_y = Y_M \cap T_y$.  This contradicts \fref{lem:YMnotmaxabelian}.
\end{proof}

\begin{proposition}\label{prop:nol34} Suppose that $y \in \mathcal Y_S^*$ and let $K \leq  E_y$ be a component of  $C_G(y)$. Then  $K/Z(K) \not\cong\PSL_3(4)$ or $\Sp_4(q)$, $q =2^a> 2$.
\end{proposition}

\begin{proof} Assume the proposition is false.  By \fref{lem:L3Sp4DyNA} $D_y $ is non-abelian.
We first prove the following claim.

\begin{claim}\label{clm:N} {\it The component $K$ is simple and there exists $N\le G$ normalized by $M$ such that  $$N = E(N) = E_yK_{r+1}= K_1 \cdots K_{r+1}$$ with $[E_y,K_{r+1}]=1$ and $K \cong K_{r+1}$. Furthermore, $S \cap N = J(O_2(M))$,  $S$ permutes the components of $N$ transitively by conjugation, and  $M= SN_M(K_1)$.}
\end{claim}

\medskip

 Let $g \in N_S(N_S(T_y)) \setminus N_S(T_y)$ with $g^2 \in N_S(T_y)$. By \fref{lem:Ty TI},  $D_y \cap D_y^g = 1$ and $[D_y,D_y^g] = 1$.
As $D_y \le J(O_2(M))$ and $g$ normalizes $O_2(M)$, $D_y^g\le J(O_2(M))$. As $D_y$ is normal in $N_S(T_y)$ the same applies for $D_y^g$.

For $1 \le i \le r$, set $$ C_i = \prod_{j\ne i} S_j D_y.$$ As, for $i \ne j$, $S_i \cap S_j \le K_i\cap K_j \le Z(K_i)\cap Z(K_j) \le D_y $, the Modular Law implies $\bigcap_{i=1}^r C_i= D_y$. In addition, we also have $[S_i,C_i]\le [K_i,C_i]=1$.

If $D_y^gC_i/C_i$ is abelian for all $i$, then $(D_y^g)' \le \bigcap_{i=1}^r C_i=D_y$ contrary to $(D_y^g)' \ne 1$ and $D_y \cap D_y^g=1$.
Thus we may fix notation so that $D_y^gC_1/C_1$ is not abelian.     Set $\ov{S_yD_y}= S_yD_y/C_1$.  Then $\ov{D_y^g} \le \ov{S_1} $ and $\ov{D_y^g} \not \le\ov{E_{1j}} $ for $j=1,2$ as $\ov{D_y^g}$ is not abelian. Let $\rho\in N_{K_1}(S_1)$ be arbitrary  of  maximal odd order and such that $\rho$ acts fixed-point-freely on $S_1/Z(S_1)$. By \fref{lem:rootinYM} $\rho \in N_{M^\dagger}(O_2(M))$ and by \fref{lem:Ty TI}  $T_y^g \cap O_2(M)$ is a trivial intersection group in $N_G(O_2(M))$. As $D_y^g = J(O_2(M) \cap T_y^g)$, we see that $\rho$ normalizes $D_y^g$ if and only if it normalizes $T_y^g \cap O_2(M)$. Suppose that $\rho$ does not normalize $D_y^g$. Then $[D_y^g,D_y^{g\rho}]=1$, as both groups are normal in $O_2(M)$. But then $[\ov{D_y^g}, \ov{D_y^{g\rho}}] =1$, and this contradicts Lemmas~\ref{lem:sp4sylow+} and \ref{lem:L34facts} (iii).  Hence $\rho$ normalizes $D_y^g$. It also normalizes $\ov {D_y^g}$ and, as $D_y=J(D_y)$ is generated by involutions, it follows that $\ov {D_y^g}=\ov{S_1}$. In particular, $[S_1,D_y^g]= S_1' = Z(S_1)$. We have $D_y^g= C_{D_y^g}(\rho)[D_y^g,\rho]$. Now $[S_yD_y,\rho] = S_1$. Hence $[D_y^g,\rho]\le S_1$ and $C_{D_y^g}(\rho) \le Z(S_1)C_1$. We conclude that $S_1 \le D_y^g$ as   \begin{eqnarray*}\ov {S_1} =\ov {D_y^g} =   \ov{[D_y^g,\rho]}\ov{C_{D_y^g}(\rho)} \le \ov{[D_y^g,\rho]} \ov {Z(S_1)} \le \ov{[D_y^g,\rho]} \ov {(S_1\cap D_y^g)} \le \ov{(D_y^g \cap S_1)}.\end{eqnarray*}
If $K_1$ is not simple, then $Z(K_1) \le S_1 \le D_y^g$ and so $D_y \cap D_y^g\ne 1$, a contradiction. Hence the components in $E_y$ are simple groups. This proves the first statement in \fref{clm:N}.

Since $D_y^g \ge S_1$ and $Y_M \cap S_1 = Z(S_1)$ by \fref{lem:rootinYM}, using \fref{lem:comps to max}  we have $E_x = E_y^g$ for all $x \in (Y_M \cap S_1)^\#$.

If $K_1\cong \Sp_4(q)$, then $Z(S_1)$ contains root subgroups $R_1$ and $R_2$ and so $K_1=\langle C_{K_1}(R_1), C_{K_1}(R_2)\rangle$ normalizes $E_y^g$.

Suppose that $K_1 \cong \PSL_3(4)$.  Then all the involutions in $K_1$ are $K_1$-conjugate. Thus for involutions $t \in S_1 \setminus Z(S_1)$, the group $E_t$ is conjugate to $E_x$. Since $t \in D_y^g$, we get $E_t= E_y^g$ by \fref{lem:E to E}. Hence this time we see that $K_1 = \langle C_K(t) \mid t \in S_1\rangle$ normalizes $E_y^g$. Thus in all cases
$$K_1\text{ normalizes }E_y^g.$$
Furthermore, by \fref{lem:K into E_w}, $K_2 \dots K_r\le E_y^{g}$ (this is obviously true if $r=1$). Hence
$$E_y\text{ normalizes }E_y^g.$$

Since $g^2\in N_G(E_y)$, we also have $E_y^g $ normalizes $E_y^{g^2}=E_y$.  It follows that the components of $E_y$ and the components of $E_y^g$ are components of $E_yE_y^g$. It now follows that $E_y\cap E_y^g = K_2 \cdots K_r$ and that we can write $$N= E_yE_y^g = K_1K_2 \cdots K_rK_{r+1}$$ where $E_y^g=K_2 \dots K_{r+1}.$

  We have  $J(O_2(M))= S_y^gD_y^g$ and so $S_y^gD_y^g \cap N \in \syl_2(N)$.
Furthermore $C_{O_2(M)}(N) = 1$, as otherwise $C_{Y_M}(N) \not= 1$, but this is not possible as $y \in \mathcal Y^*$. Therefore
$$J(O_2(M)) \in \syl_2(N).$$
This verifies the second and third statement in \fref{clm:N}.
\\
\\
Define $S_{r+1}= O_2(M) \cap K_{r+1}$. Then $$\prod_{i=1}^{r+1} S_i \in \syl_2(N).$$

We now argue as in the case when $D_y$ was abelian. The subgroups $S _i$ contains exactly two maximal rank  elementary abelian subgroups $E_{i1}$, $E_{i2}$ of order $16$ if  $K \cong \PSL_3(4)$, and $q^3$ otherwise. Thus the set of maximal order elementary abelian subgroups in $J(O_2(M))$ is
$$\mathfrak  A = \left\{\prod_{i = 1}^{r+1} E_{i i_j}\mid  i_j \in\{1,2\}\right\}$$
and $M$ permutes $\mathfrak  A$ by conjugation. The subgroup $M$ also permutes the pairs $(F_1,F_2)\in \mathfrak  A \times \mathfrak  A $ which have the property that $$|F_1 : F_1 \cap F_2| = |F_2 : F_1 \cap F_2| = \begin{cases} 4& K  \cong \PSL_4(3)\\q&\text{otherwise}\end{cases}.$$  Thus $M$ permutes the set of commutators  $[F_1,F_2]$ for all such pairs $(F_1,F_2) $. As $[F_1,F_2] = Y_M \cap S_i$ for some $i$,   $M$ permutates the set
 $$\Theta=  \{ Y_M \cap S_i\mid 1 \le i \le r\}.$$
Hence $M$ normalizes the subgroup
\begin{eqnarray*}N^* = \langle L&\mid& L \mbox{ a component of }C_G(Y_M \cap K_i),\\&& L/Z(L) \cong K/Z(K) \text{ and } |L|=|K|, i = 1, \ldots ,r\rangle.\end{eqnarray*}
Using the structure of $N$, we see that $N^*=N$. Hence $M$ normalizes $N$.
In particular,  $S$ permutes the set $\{ K_1, \ldots K_{r+1}\}$ by conjugation.  Suppose that $\{K_j\mid j \in J\}$ is an $S$-orbit. We get that $Z(S) \cap Y_M \cap \prod_{j \in J}K_j \not= 1$. Application of \fref{lem:charpy} (v) gives $J = \{1, \ldots , r+1\}$.  Thus $S$ acts transitively on   $\{K_i\mid 1 \le i \le r+1\}$. Finally, as $S$ is transitive on  $\{K_i\mid 1 \le i \le r+1\}$, $M = N_M(K_1)S$ by the Frattini Argument. This completes the explanation of \fref{clm:N}.  \qedc

  By \fref{clm:N},  $M= N_M(K_1)S$ and so $S\cap N_M(K_1) \in \syl_2(N_M(K_1))$. We also know  $N_M(K_1)$ normalizes $J(O_2(M))\cap K_1= S_1$. Suppose that  $E_{11}$ is not conjugate to $E_{12}$ in $N_M(K_1)$. Then $E_{11}$ is normal in $N_M(K_1)$ and $F = \langle E_{11}^M \rangle$ is elementary abelian with $F\cap K_j \in \{E_{j1},E_{j2}\}$. Thus $F$ is normalized by  $\langle M ,N_{K_1}(E_{11})\rangle\le M^\dagger$.  Since $ N_{K_1}(E_{11})$ does not normalize $Y_M$, this is impossible. Hence  $E_{11}$ is   conjugate to $E_{12}$ in $N_M(K_1)$. Since $S \cap N_M(K_1) \in \syl_2(N_M(K_1))$, \fref{clm:N} implies $S$ acts transitively on $\{E_{ij}\mid 1 \le i \le r+1, j=1,2\}$.

Let  $A$ be an elementary abelian normal subgroup of $S$ contained in $O_2(M)$. Put  $\ov{N_{M}(K_1)}=N_{M}(K_1)/C_M(K_1)$. Then $\ov{N_M(K_1)}$ normalizes $\ov{J(O_2(M))\cap K_1}= \ov {S_1}$ and $\ov{A}$ is normal in $ \ov{S \cap N_M(K_1)}$.  It follows that $$\ov{A}\le \ov{E_{11}} \cap \ov{E_{12}}=\ov{ Y_M}.$$ Hence $[A,O_2(M)]\le C_M(K_1)$.  Since $S$ acts transitively on $\{K_1, \dots, K_{r+1}\}$ and $S$ normalizes $[A,O_2(M)]$, we have $$[A,O_2(M)] \le O_2(M) \cap \bigcap_{i=1}^{r+1} C_M(K_i)=  C_{O_2(M)}(N)=1.$$ Hence $A \le Y_M$. Now application of \fref{lem:YMnotmaxabelian} yields the contradiction. This proves the proposition.
\end{proof}

\section{Proof of the Theorem}

Let $M$ and $Y_M$ be  as in the assumption of the theorem. That is $Y_M$ is tall, asymmetric but not characteristic 2-tall. By \fref{lem:sig} there is some $y \in Y_M^\#$ with $E(C_G(y)) \not= 1$.  In particular $\mathcal Y_S^* \not= \emptyset$. For $y \in \mathcal Y_S^*$ we have $E_y \not= 1$. Let $K$ be a component of $E_y$. By \fref{prop:sporadic} $K/Z(K)$ is not a sporadic simple group. By \fref{prop:Lieodd} and \fref{lem:rank1-done} $K/Z(K)$ is not a group of Lie type in odd characteristic.  \fref{prop:altdone} states that $K/Z(K)$ is not an alternating group. Hence $K/Z(K)$ a group of Lie type in characteristic 2.  \fref{prop:complie2} shows that $K/Z(K) \cong \PSL_3(4)$ or $\Sp_4(q)$, $q \geq 4$. Finally  \fref{prop:nol34} provides the contradiction which proves the theorem.

\section*{Acknowledgments}

We are indebted to an anonymous referee for valuable suggestions which helped to improve the clarity and accuracy of the work presented here. In particular, the proof of Lemma 4.7 (i) was provided by the referee.
The second author was partially supported by the DFG.

\end{document}